\documentclass[11pt]{article}
\usepackage[margin = 1in]{geometry}

\usepackage{lipsum}
\usepackage{amsfonts}
\usepackage{graphicx}
\graphicspath{ {images/} }
\usepackage{epstopdf}
\usepackage{dsfont}
\usepackage{mathtools}
\usepackage{empheq}
\usepackage{amsfonts}
\usepackage{amsgen,amsthm,amsmath,amstext,amsbsy,amsopn,amssymb,subfigure,stmaryrd,mathabx,mathrsfs}
\usepackage{comment}
\usepackage[ruled, vlined, lined, commentsnumbered]{algorithm2e}
\usepackage[font=footnotesize,labelfont=bf]{caption}

\usepackage{blkarray}

\usepackage{url}
\usepackage{longtable}
\usepackage{mathtools}
\usepackage{multirow}
\usepackage{array}

\usepackage{enumitem}
\usepackage{hyperref}
\usepackage{natbib}
\usepackage{booktabs}
\usepackage{algpseudocode}
\usepackage{xcolor}
\usepackage[utf8]{inputenc} 
\usepackage[T1]{fontenc}

\ifpdf
  \DeclareGraphicsExtensions{.eps,.pdf,.png,.jpg}
\else
  \DeclareGraphicsExtensions{.eps}
\fi

\newcommand{\cH}{\mathcal{H}}
\newcommand{\cS}{\mathcal{S}}

\newcommand{\cP}{\mathcal{P}}
\newcommand{\cX}{\mathcal{X}}

\newcommand{\cF}{\mathcal{F}}

\newcommand{\cE}{\mathcal{E}}

\newcommand{\cI}{\mathcal{I}}
\newcommand{\cW}{\mathcal{W}}

\newcommand{\htheta}{\widehat{\theta}}

\newcommand{\bbN}{\mathbb{N}}
\newcommand{\bbP}{\mathbb{P}}
\newcommand{\bbE}{\mathbb{E}}
\newcommand{\bbR}{\mathbb{R}}

\renewcommand{\exp}{{\rm{exp}}}
\newcommand{\TV}{{\sf TV}}
\newcommand{\KL}{{\rm{KL}}}

\newcommand{\indi}{{\mathds{1}}}

\newcommand{\CI}{{\textnormal{CI}}}
\newcommand{\CS}{{\textnormal{CS}}}
\newcommand{\wh}{\widehat}
\newcommand{\wt}{\widetilde}

\newtheorem{Definition}{Definition}
\newtheorem{Theorem}{Theorem}

\newtheorem{Lemma}{Lemma}

\newtheorem{Proposition}{Proposition}

\DeclareMathAlphabet\mathbfcal{OMS}{cmsy}{b}{n}

\hypersetup{
    colorlinks=true,
    linkcolor=blue,
    filecolor=blue,      
    urlcolor=cyan,
    citecolor=blue
}
\newcommand\blfootnotea[1]{%
  \begingroup
  \renewcommand\thefootnote{}\footnote{#1}%
  \endgroup
}
\makeatletter
\newcommand*{\rom}[1]{\expandafter\@slowromancap\romannumeral #1@}
\makeatother

\allowdisplaybreaks

\begin{document}
	\title{Adaptive Robust Confidence Sets in High Dimensions\blfootnotea{Authors are listed in alphabetical order.}}
	    \author{
Ilias Diakonikolas\thanks{Supported by NSF Medium Award CCF-2107079 and a Guggenheim Fellowship.}\\
University of Wisconsin-Madison \\
{\tt ilias@cs.wisc.edu} \\
\and 
Chao Gao\thanks{Supported by NSF Grants ECCS-2216912 and DMS-2310769 and an Alfred Sloan fellowship.}\\
University of Chicago \\
{\tt chaogao@uchicago.edu}
\and 
Daniel M.\ Kane\thanks{Supported by NSF Medium Award CCF-2107547.}\\
University of California, San Diego \\
{\tt dakane@ucsd.edu} 
\and 
Yuetian Luo \\
Rutgers University\\
{\tt yl2785@stat.rutgers.edu}
}

	\date{}
	\maketitle

\begin{abstract}
We study adaptive confidence sets for mean vectors of multivariate Gaussians with identity covariance under Huber contamination when the contamination proportion is unknown. 
We show that an information-theoretically optimal adaptive confidence set can be constructed by 
simultaneously quantifying uncertainty across all one-dimensional projections. The 
resulting confidence set achieves the optimal diameter and reveals a substantial 
adaptation cost compared with the minimax estimation error. 

Although statistically optimal, this construction is computationally
infeasible. We establish a Statistical Query (SQ) lower bound showing
that achieving diameter $O(m^{-1/2})$ requires either exponentially many
queries or query tolerance $d^{-\Omega(m)}$. Thus, in contrast to robust
mean estimation, adaptation to the unknown contamination proportion
incurs both a statistical cost and a computational barrier within the
SQ framework.

Finally, we give a low-degree
moment-matching algorithm that constructs an adaptive confidence
set using $d^{O(m)}$ samples and time. Its diameter bound contains
an additional $O(m^{-1/2})$ term beyond the information-theoretic
rate, matching the dependence on $m$ in the SQ lower bound.
Together, these results characterize the statistical limits of
adaptive robust confidence sets under Huber contamination and
quantify a computational barrier within the SQ framework.

\end{abstract}

\thispagestyle{empty}
\setcounter{page}{0}

\newpage

\begin{sloppypar}
\section{Introduction} \label{sec: introduction}

Given $X_1, \ldots, X_n \overset{iid}\sim N(\theta,I_d)$, the classical confidence set
$$\left\{\theta\in\mathbb{R}^d: n\|\bar{X}-\theta\|^2\leq \chi_{d,1-\alpha}^2\right\}$$
is known to achieve the exact coverage level $1-\alpha$, where $\chi_{d,1-\alpha}^2$ denotes the $(1-\alpha)$ quantile of chi-squared distribution with degrees of freedom $d$. Moreover, its diameter is of order $\sqrt{d/n}$, which is minimax optimal. In this paper, we consider the same problem of constructing confidence sets in a more general setting where data may contain arbitrary outliers. Specifically, we consider the Huber contamination model \citep{huber1964robust}, where observations are generated according to 
\begin{equation} \label{def:model}
	X_1, \ldots, X_n \overset{iid}\sim P_{\epsilon,\theta,Q} = (1- \epsilon ) N(\theta, I_d) + \epsilon Q.
\end{equation}
Under this model, roughly an $\epsilon$ fraction of the observations are generated from an arbitrary unknown contamination distribution $Q$ that may or may not be related to the model distribution $N(\theta, I_d)$. The parameter $\epsilon$ is referred to as the contamination proportion or the contamination level. Compared with the classical (clean data) setting, a confidence set under (\ref{def:model}) must account for uncertainty arising from both finite-sample fluctuations and the presence of outliers.

The difficulty of constructing a robust confidence set depends critically on whether the contamination proportion $\epsilon$ is known. When $\epsilon$ is known, a robust confidence set can be obtained by directly inverting a high-probability estimation error bound. For instance, the Tukey median \citep{tukey1975mathematics} is known to satisfy
\begin{equation} \label{ineq:opt-estimation}
	\begin{split}
		\inf_{\theta,Q}
		P_{\epsilon,\theta,Q}
		\left(
		\|\widehat{\theta}-\theta\|
		\le
		C\left(
		\sqrt{\frac{d}{n}}
		+\epsilon
		+\sqrt{\frac{\log(1/\alpha)}{n}}
		\right)
		\right)
		\ge
		1-\alpha,
	\end{split}
\end{equation}
for some universal constant $C>0$, as long as $\epsilon\le 1/5$ \citep{chen2018robust}. Therefore, the $\ell_2$ ball with radius $O(\sqrt{d/n}+\epsilon)$ centered at the Tukey median achieves the desired coverage property according to (\ref{ineq:opt-estimation}). Its diameter is also of order $O(\sqrt{d/n}+\epsilon)$, which is minimax optimal under (\ref{def:model}). In addition to the exact Tukey median, which is computationally infeasible in high dimensions \citep{johnson1978densest}, a number of polynomial-time computable estimators \citep{diakonikolas2016robust,lai2016agnostic,diakonikolas2018robustly,diakonikolas2024near} can also be used as the center, and therefore yield polynomial-time procedures for constructing robust confidence sets. Thus, when $\epsilon$ is known, both the statistical and computational aspects of robust confidence set construction are well understood.

On the other hand, the problem becomes substantially more challenging 
when $\epsilon$ is unknown. This is because the statistical error bound of a robust point estimator depends on the unknown contamination level $\epsilon$, and therefore cannot be directly translated into a valid confidence set. As a result, one is led to an adaptive confidence set construction problem.

The construction of adaptive confidence sets is known to be a challenging topic in nonparametric and high-dimensional inference. The difficulty typically arises because uncertainty quantification depends on an unknown regularity parameter, such as function smoothness, sparsity, or structural complexity. When statistical models corresponding to different values of the regularity parameter are nested, honest adaptation is usually impossible \citep{low1997nonparametric,juditsky2003nonparametric,cai2004adaptation,baraud2004confidence,cai2006adaptive,robins2006adaptive,hoffmann2011adaptive,bull2013adaptive,carpentier2013honest}. Similar impossibility phenomena have also been established for high-dimensional sparse and low-rank regression models \citep{nickl2013confidence,cai2017confidence,carpentier2015signal,carpentier2018adaptive}. Nevertheless, adaptive confidence sets can sometimes be constructed by relaxing the coverage requirement \citep{wahba1983bayesian,genovese2008adaptive,cai2014adaptive} or by restricting attention to subclasses of signals satisfying additional structural assumptions, such as self-similarity, polished-tail conditions, or shape constraints \citep{gine2010confidence,nickl2016sharp,chernozhukov2014anti,szabo2015frequentist,dumbgen2003optimal,cai2013adaptive,van2017adaptive}. We refer the reader to Chapter 8 of \cite{gine2015mathematical} for a comprehensive overview of this topic.

In contrast to much of the nonparametric literature, where the construction of adaptive confidence sets is impossible, \cite{luo2024adaptive} showed that adaptive confidence intervals can be constructed under (\ref{def:model}) when $d=1$. However, a substantial cost of adaptation is unavoidable. The optimal length of an adaptive confidence interval when $\epsilon$ is unknown scales as
\begin{equation}
\frac{1}{\sqrt{\log n}} + \frac{1}{\sqrt{\log(1/\epsilon)}},
\label{eq:rate-luogao}
\end{equation}
which is exponentially slower than the optimal estimation error rate $n^{-1/2}+\epsilon$. This adaptation cost stems from a fundamental information-theoretic phenomenon: a zero-mean Gaussian distribution can be statistically indistinguishable from a nonzero-mean Gaussian distribution under contamination. In particular, the rate (\ref{eq:rate-luogao}) remains $1/\sqrt{\log n}$ when $\epsilon=0$. Thus, even when the data are completely clean, a statistician must collect exponentially more samples in order to achieve the same level of uncertainty quantification associated with the parametric rate $n^{-1/2}$. This phenomenon arises because the confidence interval must remain valid under contamination, even though the observed data may in fact contain no outliers.

Extending adaptive confidence intervals to multivariate confidence sets is far from straightforward. The one-dimensional construction in \cite{luo2024adaptive} relies heavily on order statistics and quantile-based arguments, which do not have direct analogues in higher dimensions. Moreover, the geometry of multivariate confidence sets introduces new challenges in both the characterization of the optimal diameter and the design of computationally efficient procedures.

In this paper, we characterize the optimal diameter of an adaptive robust confidence set under (\ref{def:model}) for general dimension $d$. We show that the optimal diameter scales as
\begin{equation}
\frac{1}{\sqrt{\log (n/d)}} + \frac{1}{\sqrt{\log(1/\epsilon)}},
\label{eq:rate-opt-diam}
\end{equation}
which exhibits a similar adaptation cost to the one-dimensional result when compared with the optimal estimation rate $\sqrt{d/n}+\epsilon$. An adaptive confidence set achieving the diameter (\ref{eq:rate-opt-diam}) can be constructed by intersecting all halfspaces whose supporting hyperplanes are determined by one-dimensional adaptive confidence bounds along projection directions. However, this information-theoretically optimal construction is computationally infeasible. Interestingly, we provide formal evidence that the computational hardness is 
an intrinsic feature of the problem. 
Specifically, we establish a Statistical Query (SQ) lower bound showing 
that any SQ confidence set with diameter $m^{-1/2}$ must either use exponentially many oracle queries or require query tolerance at most $d^{-\Omega(m)}$. Since achieving tolerance $d^{-\Omega(m)}$ requires $d^{\Omega(m)}$ samples to simulate, 
this result rules out a broad class of polynomial-time algorithms 
for constructing adaptive confidence sets with shrinking diameter ($m\to\infty$). Equivalently, to achieve diameter $m^{-1/2}$, one can at best hope 
for an algorithm using $d^{O(m)}$ samples and $d^{O(m)}$ computational complexity.

Such an algorithm can indeed be constructed through the idea of low-degree moment matching. Our algorithm starts with a robust pilot estimator and then quantifies its uncertainty by searching for a reweighting of the observations whose moments up to degree $m$ agree with those of a Gaussian distribution centered at the pilot estimator. This optimization problem can be solved by linear programming in time $d^{O(m)}$ using $d^{O(m)}$ samples. It leads to an adaptive robust confidence set with diameter
\begin{equation}
\frac{1}{\sqrt{\log (n/d)}} + \frac{1}{\sqrt{\log(1/\epsilon)}}+\frac{1}{\sqrt{m}},
\label{eq:rate-quasi-poly-diam}
\end{equation}
which matches the SQ lower bound up to constant factors. Consequently, the proposed procedure essentially characterizes the optimal statistical--computational tradeoff for adaptive confidence set construction under Huber contamination. 
For a prescribed threshold $\eta\in(0,0.05]$, choosing
$m=\Theta(\log(1/\eta))$ gives sample and time complexity
$d^{O(\log(1/\eta))}$. For contamination levels
$\epsilon\geq\eta$, the additional $m^{-1/2}$ term is absorbed
by the information-theoretic diameter bound, up to constant
factors. 

A recurring theme throughout the paper is moment matching. For the lower bound results, both the information-theoretic lower bound and the SQ lower bound are established through connections to non-Gaussian component analysis (NGCA) \citep{blanchard2006search,diakonikolas2017statistical,diakonikolas2023sq}. While the NGCA construction underlying the information-theoretic lower bound only matches the first moment of a Gaussian distribution, the SQ lower bound relies on matching the first $m$ moments. For the upper bound results, our efficient algorithm likewise searches for a reweighting of the observations that matches the first $m$ moments of a Gaussian distribution. Interestingly, the theoretical analysis of the algorithm ultimately relies on exactly the same non-Gaussian construction that appears in the proof of the SQ lower bound.

Before ending this section, we invite the reader to compare the problems of adaptive estimation and adaptive confidence set construction. Recall that the optimal estimation error rate is $\sqrt{d/n}+\epsilon$, and that this rate can be achieved by a polynomial-time algorithm without knowledge of $\epsilon$. In other words, there is essentially no adaptation cost or computational cost for the estimation problem. In contrast, adaptive confidence set construction is substantially more difficult, both statistically and computationally. Without knowledge of $\epsilon$, the optimal diameter deteriorates from the polynomial rate to the logarithmic rate (\ref{eq:rate-opt-diam}). Moreover, the lack of knowledge of $\epsilon$ introduces a computational barrier, leading to the rate (\ref{eq:rate-quasi-poly-diam}) under computational complexity $d^{O(m)}$.

\subsection{Related Work}
\label{sec:literature review}

\paragraph{High-Dimensional Robust Estimation.}
Robust statistics originated in the pioneering work of Huber, Tukey, and Hampel \citep{huber1964robust,huber1967behavior,huber1968robust,tukey1975mathematics,hampel1974influence}. While the classical literature focused primarily on low-dimensional and asymptotic 
settings, the past decade has witnessed substantial progress in developing robust procedures 
with non-asymptotic guarantees in high dimensions. For robust mean estimation under Huber 
contamination, statistically optimal estimators were established in 
\cite{chen2018robust,lugosi2021robust}, while computationally efficient algorithms were 
developed in a sequence of works including~\cite{diakonikolas2016robust,lai2016agnostic,DKK+17,diakonikolas2018robustly,cheng2019high,diakonikolas2020outlier, dalalyan2022all,depersin2022robust,diakonikolas2024near}. Beyond mean estimation, robust methods have been extensively studied for covariance estimation, regression, sparse estimation, clustering, and many other learning tasks \citep{cheng2019faster,minsker2018sub,gao2020robust,loh2017statistical,diakonikolas2019efficient,DKK+19,sun2020adaptive,prasad2020robust,zhu2022robust,balakrishnan2017computationally,hopkins2018mixture,diakonikolas2020robustly, BakshiDJKKV26}. We refer the reader to the monographs \citep{huber2011robust,hampel2011robust,diakonikolas2023algorithmic} for comprehensive treatments of the subject.

\paragraph{Robust Confidence Sets.}
Compared with robust estimation, uncertainty quantification under contamination has received considerably less attention. Confidence sets under Huber contamination without adaptation to the contamination level were studied by \cite{bateni2020confidence,wang2023huber}. When the contamination proportion is unknown, adaptive robust confidence intervals for the location model were developed by \cite{luo2024adaptive}, who established both the optimal adaptation cost and an explicit adaptive procedure in one dimension. Related adaptive uncertainty quantification problems have also been investigated for discrete models \citep{cho2025robust}, regression with oblivious contamination \citep{xie2025confidence}, Efron's two-groups model \citep{wang2026adaptive} and missing data \citep{ma2026adaptive}. To the best of our knowledge, the present paper provides the first systematic study of adaptive robust confidence sets in high dimensions.

\paragraph{Statistical--Computational Tradeoffs.}
Statistical--computational tradeoffs have emerged as a central theme in modern high-dimensional inference since the influential works of \cite{chandrasekaran2013computational,berthet2013complexity}. A large body of literature has identified gaps between information-theoretic limits and the performance of polynomial-time algorithms in problems such as sparse estimation, community detection, tensor decomposition, and planted structure detection \citep{ma2015computational,chen2016statistical,wang2016statistical,cai2017computational,hajek2015computational,zhang2018tensor,brennan2018reducibility,luo2020tensor,pananjady2022isotonic,hopkins2017power,barak2019nearly,gamarnik2014limits,gamarnik2022sparse,schramm2020computational}. In robust statistics, statistical--computational tradeoffs have been established for mean estimation under total variation contamination, sparse mean estimation, covariance estimation, regression, and list-decodable learning; see, e.g.,~\citet{diakonikolas2017statistical,brennan2020reducibility,diakonikolas2019efficient,diakonikolas2025information,diakonikolas2026robust,diakonikolas2018list} and Chapter~8 of~\cite{diakonikolas2023algorithmic}. The present paper moves beyond estimation and testing, and provides evidence for statistical--computational tradeoffs in adaptive uncertainty quantification. In particular, we show that the computational hardness of adaptive robust confidence sets is fundamentally linked to the adaptation cost incurred by the unknown contamination level.

\subsection{Paper Organization} \label{sec:organization}

The remainder of the paper is organized as follows. In Section \ref{sec:adaptiveRCI}, we formulate the problem of adaptive robust confidence set construction and establish the information-theoretic limits, including an optimal construction that achieves the minimax diameter. Section \ref{sec:GSQlow} develops a Statistical Query (SQ) lower bound, which characterizes the computational barrier of adaptive robust confidence sets. In Section \ref{sec:quasi-poly-algorithm}, we present our algorithm together with its theoretical guarantees. All technical proofs are deferred to 
the Appendix.

\subsection{Notation}\label{sec:notations}

We use $\mathbb{N}$ and $\mathbb{R}$ to denote the sets of nonnegative integers and real numbers, respectively. Define $[n]=\{1,\ldots,n\}$ for any positive integer $n$. The notation $\indi\{\cdot\}$ denotes the indicator function.
For any two sequences $\{a_n\}$ and $\{b_n\}$, we write $a_n \asymp b_n$ if there exist constants $c,C>0$ such that
$ca_n \le b_n \le Ca_n$
for all $n$. We write $a_n \lesssim b_n$ if
$a_n \le C b_n$
for some constant $C>0$ independent of $n$.
We use $B(\theta,r)$ to denote the Euclidean ball centered at $\theta$ with radius $r$. For any vector $v$, we use $\|v\|$ to denote its Euclidean norm. Let $\mathcal{S}^{d-1}$ denote the unit sphere in $\mathbb{R}^d$.
Given a cumulative distribution function (CDF) $F(\cdot)$, the corresponding quantile function is defined by
$F^{-1}(q)
=
\inf\{t\in\mathbb{R}:F(t)\ge q\}$.
In particular, given observations $X_1,\ldots,X_n\in\mathbb{R}$, let
$X_{(1)}\le\cdots\le X_{(n)}$
denote the order statistics. We use
$F_n(t)
=
\frac{1}{n}\sum_{i\in[n]}\indi\{X_i\le t\}$
and
$F_n^{-1}(q)
=
X_{(\lceil nq\rceil)}$
to denote the empirical CDF and empirical quantile function, respectively.
For any probability measure $P$, we use $P^{\otimes n}$ to denote the product measure of $n$ independent copies of $P$.
The chi-squared divergence between two probability distributions $P$ and $Q$ is defined by $\chi^2(P,Q)=\int\frac{(dP-dQ)^2}{dQ}$.
For probability distributions $P$ and $Q$ and a constant $a\in[0,1]$, the notation
$P\ge aQ$
means that $Q$ is absolutely continuous with respect to $P$ and that the Radon--Nikodym derivative satisfies
$\frac{dQ}{dP}\le \frac{1}{a}$.
We use $\mathbb{E}$ and $\mathbb{P}$ for generic expectation and probability operators whenever the underlying distribution is clear from the context.

\section{Information-Theoretic Limit} \label{sec:adaptiveRCI}

\subsection{Problem Setup} \label{sec:formulation}

Given i.i.d. observations generated according to \eqref{def:model}, our goal is to construct a confidence set $\wh{\CS}=\wh{\CS}(\{X_i\}_{i=1}^n)$ that characterizes the uncertainty of the mean vector $\theta$. We require that the set contains $\theta$ with high probability, and at the same time its size can be minimized. Given some $\alpha\in(0,1)$, the coverage probability is defined according to
\begin{equation}
\inf_{\epsilon \in [0, \epsilon_{\max}]} \inf_{\theta, Q} P_{\epsilon, \theta, Q} \left( \theta \in \widehat{\CS}\right) =  \inf_{\theta, Q} P_{\epsilon_{\max}, \theta, Q} \left( \theta \in \widehat{\CS}\right) \geq 1- \alpha, \label{eq:uni-robust-coverage}
\end{equation}
where $\epsilon_{\max}$ serves as a known upper bound for $\epsilon$, which is assumed to be a constant throughout the paper, specifically 
$\epsilon_{\max}= 0.05$. The infimum over $\epsilon \in [0, \epsilon_{\max}]$ is present because the confidence set is required to adapt to the unknown contamination proportion. The equality in (\ref{eq:uni-robust-coverage}) follows from the nesting relation $\{P_{\epsilon, \theta, Q}: Q\} \subseteq  \{P_{\epsilon_{\max}, \theta, Q}: Q\}$ for any $\epsilon \in [0, \epsilon_{\max}]$. Then, the collection of all adaptive confidence sets satisfying the coverage property is defined by
\begin{equation} \label{eq:adaptive-CS}
	\cI_{\alpha}(\epsilon_{\max}) = \left\{ \widehat{\CS}: \inf_{\theta, Q} P_{\epsilon_{\max}, \theta, Q} \left( \theta \in \widehat{\CS}\right) \geq 1- \alpha  \right\}.
\end{equation}
For any $\widehat{\CS}\in \cI_{\alpha}(\epsilon_{\max})$, we measure its size by the diameter $|\widehat{\CS}| = \sup_{\theta_1, \theta_2 \in \widehat{\CS}} \|\theta_1 - \theta_2\|$. The optimal diameter for some $\epsilon\in [0, \epsilon_{\max}]$ among all $(1-\alpha)$-level adaptive robust confidence sets is given by
\begin{equation} 
	r_{\alpha}( \epsilon, \epsilon_{\max}) = \inf \left\{r \geq 0: \inf_{\widehat{\CS}\in  \cI_{\alpha}(\epsilon_{\max}) }\sup_{\theta, Q} P_{\epsilon, \theta, Q} \left( |\widehat{\CS}| > r \right) \leq \alpha \right\}. \label{eq:minimax-length-ARCI}
\end{equation}
The first goal of our paper is to characterize this information-theoretically optimal diameter $r_{\alpha}( \epsilon, \epsilon_{\max})$.

From the definition (\ref{eq:minimax-length-ARCI}), it is straightforward that
$$r_{\alpha}( \epsilon, \epsilon)\leq r_{\alpha}( \epsilon, \epsilon_{\max})\leq r_{\alpha}( \epsilon_{\max}, \epsilon_{\max}).$$
The quantity $r_{\alpha}( \epsilon, \epsilon)$ is the optimal size of a confidence set that knows the value of $\epsilon$. It is of order $\sqrt{\frac{d}{n}}+\epsilon$, and can be achieved by an $\ell_2$ ball centered at the Tukey median with radius $C\left(\sqrt{\frac{d}{n}}+\epsilon\right)$ for some constant $C$ depending on $\alpha$. Similarly, the order of $r_{\alpha}( \epsilon_{\max}, \epsilon_{\max})$ is $\sqrt{\frac{d}{n}}+\epsilon_{\max}\asymp 1$ since $\epsilon_{\max}$ is assumed to be a constant. It can be achieved by a conservative strategy of using the upper bound $\epsilon_{\max}$ when the actual contamination proportion $\epsilon$ is unknown. We will show in this section that $r_{\alpha}( \epsilon, \epsilon_{\max})$ is of order $ \frac{1}{\sqrt{\log (n/d )} }+\frac{1}{\sqrt{\log (1/\epsilon)}}$.

\subsection{Upper Bound}\label{sec:optimal-ARCS-Gaussian}

We will first review the construction of an optimal adaptive confidence interval when $d = 1$. Recall that $F_n(\cdot)$ and $F_n^{-1}(\cdot)$ are the empirical CDF and empirical quantile function of the observations. 
The interval constructed by \cite{luo2024adaptive} is given by
\begin{equation}
\widehat{\CI} = \bigcap_{1.6 \leq t \leq  \Phi^{-1}\left(1 - \sqrt{ \log (2/\alpha)/(2n) } \right) }\left[\wh{\theta}_L(t)-\frac{2}{t},\wh{\theta}_R(t)+\frac{2}{t}\right], \label{eq:intersect-quan}
\end{equation}
where
\begin{eqnarray*}
\wh{\theta}_L(t) &=& F_n^{-1}( 2(1 - \Phi(t) ) )  + t, \\
\wh{\theta}_R(t) &=& F_n^{-1}( 1- 2(1 - \Phi(t) ) )  - t.
\end{eqnarray*}
We use $\Phi(\cdot)$ for the standard Gaussian CDF. It is clear that the formula (\ref{eq:intersect-quan}) does not depend on $\epsilon$, and is thus adaptive. It was shown in \cite{luo2024adaptive} that both $\wh{\theta}_L(t)$ and $\wh{\theta}_R(t)$ are robust estimators of $\theta$ with \textit{known} one-sided uncertainty quantification. In particular, with probability at least $1-\alpha$, both
\begin{equation}
\wh{\theta}_L(t)-\theta \leq \frac{2}{t}\quad\text{and}\quad \wh{\theta}_R(t)-\theta\geq -\frac{2}{t}, \label{eq:one-sided-uq-1d}
\end{equation}
hold simultaneously for all $t \in \left[1.6,\Phi^{-1}\left(1 - \sqrt{ \log (2/\alpha)/(2n) } \right)\right]$. This property directly leads to the form of the adaptive interval (\ref{eq:intersect-quan}).

For the general multivariate setting, we extend the one-dimensional idea by projection. 
For $u \in \cS^{d-1}$, define the projected empirical CDF and the projected empirical quantile function as
$F_n(t; u) = \frac{1}{n} \sum_{i \in [n]} \indi(X_i^\top u \leq t)$ and $F_n^{-1}(q; u) = \inf\{t \in \bbR: F_n(t; u) \geq q\}$, respectively. Following the definitions of $\wh{\theta}_L(t)$ and $\wh{\theta}_R(t)$ in one dimension, we define a robust estimator for $\theta^{\top}u$ as
\begin{equation}
\wh{\theta}^u(t) = F_n^{-1}( 2(1 - \Phi(t) ); u )  + t.\label{eq:theta-u-quantile}
\end{equation}
The estimator is obtained by inverting a Gaussian tail probability. When the contamination proportion is smaller than $1-\Phi(t)$, the  quantile $F_n^{-1}( 2(1 - \Phi(t) ); u )$ is unaffected by contamination and can be used to estimate $\theta^\top u$. Larger values of $t$ yield tighter uncertainty quantification, while smaller values of $t$ allow robustness against larger contamination levels.

The property of (\ref{eq:theta-u-quantile}) is characterized by the following proposition.
\begin{Proposition}\label{th:estimation-guarantee}
	Suppose $\alpha\in(0,1)$ and  $n\geq C(\log(1/\alpha)+d)$ for a sufficiently large $C > 0$. Simultaneously for all $t\in\left[1.6, \Phi^{-1}\left(1 - 300 \sqrt{ \frac{d}{n} } - \sqrt{ \frac{\log(1/\alpha)}{n} } \right)\right]$ and for all $u \in \cS^{d-1}$,
\begin{equation} \label{ineq:upper-side}
	\wh{\theta}^u(t)-\theta^\top u\leq \frac{2}{t},
\end{equation}
with $P_{\epsilon, \theta, Q}$-probability at least $1-\alpha$, uniformly over all $\epsilon\in[0,0.05]$, all $\theta\in\mathbb{R}^d$, and all $Q$. Moreover, set 
\begin{equation} 
		t_{\epsilon} = \Phi^{-1}\left( 1- \epsilon - 300 \sqrt{ \frac{d}{n} } - \sqrt{ \frac{\log(1/\alpha)}{n} }  \right), \label{def:t-constant}
\end{equation} and then simultaneously for all $u \in \cS^{d-1}$,
\begin{equation}
\wh{\theta}^u(t_{\epsilon})-\theta^\top u \geq 0, \label{eq:the-other-side}
\end{equation}
with $P_{\epsilon, \theta, Q}$-probability at least $1-\alpha$, uniformly over all $\epsilon\in[0,0.05]$, all $\theta\in\mathbb{R}^d$ and all $Q$.
\end{Proposition}
Similar to (\ref{eq:one-sided-uq-1d}), the inequality (\ref{ineq:upper-side}) is a one-sided uncertainty quantification of $\wh{\theta}^u(t)$. Intersecting (\ref{ineq:upper-side}) for all $u$ and $t$, we obtain the following adaptive confidence set,
\begin{equation} \label{eq:adaptive-CI}
	\begin{split}
		\widehat{\CS}=  \left\{ \theta \in \bbR^d: \theta^\top u  \geq \max_{1.6 \leq t \leq  \Phi^{-1}\left(1 - 300 \sqrt{ \frac{d}{n} } - \sqrt{ \frac{\log(1/\alpha)}{n} } \right) } \left(\wh{\theta}^u(t)-\frac{2}{t} \right) \textnormal{ for all } u \in \cS^{d-1} \right\},
	\end{split}
\end{equation}  
and its coverage property immediately follows. To analyze the diameter, we note that $|\widehat{\CS}|=\sup_{u\in \cS^{d-1}}\sup_{\theta_1,\theta_2\in\widehat{\CS}} \left(\theta_1^{\top}u-\theta_2^{\top}u\right)$. To bound $\theta_1^{\top}u-\theta_2^{\top}u$, we use the property that any $\theta_1,\theta_2\in \widehat{\CS}$ satisfy $-\theta_1^{\top}u\geq \wh{\theta}^{-u}(t)-\frac{2}{t}$ and $\theta_2^{\top}u\geq \wh{\theta}^u(t)-\frac{2}{t}$, and thus
$$\theta_1^{\top}u-\theta_2^{\top}u\leq -\wh{\theta}^{-u}(t)-\wh{\theta}^u(t)+\frac{4}{t}.$$
Taking $t=t_{\epsilon}$ in (\ref{def:t-constant}) and applying (\ref{eq:the-other-side}), we have
$$-\wh{\theta}^{-u}(t_{\epsilon})-\wh{\theta}^u(t_{\epsilon})+\frac{4}{t_{\epsilon}}\leq \theta^{\top}u-\theta^{\top}u+\frac{4}{t_{\epsilon}}=\frac{4}{t_{\epsilon}}.$$
Hence, the diameter of $\widehat{\CS}$ is bounded by
\begin{equation}\label{ineq:length-guarantee}
    |\widehat{\CS}|\leq \frac{4}{t_{\epsilon}}\asymp \frac{1}{\sqrt{\log (n/d )} }+\frac{1}{\sqrt{\log (1/\epsilon)}} + \frac{1}{\sqrt{ \log(n/\log(1/\alpha)) }}.
\end{equation}
We refer the reader to Lemma \ref{lm:t-epsilon-bound} in Appendix \ref{sec:proof-additional-lemmas} for the order of $t_{\epsilon}$. The following theorem summarizes the properties of the confidence set (\ref{eq:adaptive-CI}).

\begin{Theorem} \label{thm:upper-bound-general-Gaussian}
There exist universal constants $C,C'>0$, such that for any $\alpha\in(0,1)$ and $n\geq C(\log(1/\alpha)+d)$, the set $\widehat{\CS}$ defined in (\ref{eq:adaptive-CI}) satisfies
	\begin{equation*}
		\begin{split}
			&\inf_{ \epsilon \in [0, 0.05], \theta, Q} P_{\epsilon, \theta, Q}\left( \theta \in\widehat{\CS}\right) \geq 1-\alpha,\\
			&\inf_{\epsilon \in [0, 0.05], \theta, Q } P_{\epsilon, \theta, Q}\left( |\widehat{\CS}| \leq C' \left(\frac{1}{\sqrt{\log (n/d )} }+\frac{1}{\sqrt{ \log(1/\epsilon) }}+ \frac{1}{\sqrt{ \log(n/\log(1/\alpha)) }} \right) \right) \geq 1-\alpha.
		\end{split}
	\end{equation*}
\end{Theorem}

Despite its statistical optimality, the adaptive confidence set 
(\ref{eq:adaptive-CI}) is not computationally feasible. 
Specifically, it requires time exponential in the dimension. 
Section \ref{sec:quasi-poly-algorithm} introduces an algorithm that partially addresses this issue.

\subsection{Lower Bound} \label{sec:optimality-of-ARCI}

In this section, we establish a lower bound that matches the upper bound rate of Theorem \ref{thm:upper-bound-general-Gaussian}.
\begin{Theorem} \label{th:lower-bound-Gaussian} There exist universal constants $c, C$ such that for any $\alpha \in (0,1/4)$ and for any $n \geq C (d + \log(1/\alpha))$,
\begin{equation} \label{ineq:r-stat-lower}
	r_{\alpha}(\epsilon, 0.05) \geq c \left( \frac{1}{\sqrt{\log (n/d)}}+ \frac{1}{\sqrt{\log(1/\epsilon)}}+ \frac{1}{\sqrt{ \log(n/\log(1/\alpha)) }} \right),
\end{equation}
for any $\epsilon\in[0,0.05]$.
\end{Theorem} 
Recall the definition of $r_{\alpha}(\epsilon, \epsilon_{\max})$ in (\ref{eq:minimax-length-ARCI}). Theorems \ref{thm:upper-bound-general-Gaussian} and \ref{th:lower-bound-Gaussian} jointly imply that
$$r_{\alpha}(\epsilon, 0.05)\asymp\frac{1}{\sqrt{\log (n/d)}}+ \frac{1}{\sqrt{\log(1/\epsilon)}}+ \frac{1}{\sqrt{ \log(n/\log(1/\alpha)) }},$$
extending the one-dimensional result in \cite{luo2024adaptive}. We note that the sharp dependence on $\alpha$ is new, even when $d=1$.

To prove the lower bound result, we only need to show $r_{\alpha}(\epsilon, 0.05)\gtrsim \frac{1}{\sqrt{\log (n/d)}}+\frac{1}{\sqrt{ \log(n/\log(1/\alpha)) }}$, since $r_{\alpha}(\epsilon, 0.05)\gtrsim \frac{1}{\sqrt{\log(1/\epsilon)}}$ can be derived by leveraging the hardness of constructing adaptive confidence intervals
in \cite{luo2024adaptive}. For any confidence set $\wh{\CS}$, we define an interval
\begin{equation} \label{eq:CS-to-CI}
	\wh{\CI} =  \left[\inf\{ e_1^\top \theta: \theta \in \wh{\CS}\}, \sup\{ e_1^\top \theta: \theta \in \wh{\CS} \}\right],
\end{equation}
where $e_1$ is the first canonical basis vector in $\mathbb{R}^d$.
It is clear that $\theta\in \wh{\CS}$ implies $e_1^{\top}\theta\in \wh{\CI}$ so that $\wh{\CI}$ is an adaptive confidence interval for the first coordinate of $\theta$. Moreover, the length of $\wh{\CI}$ is at most the diameter of $\wh{\CS}$. Thus, the length lower bound of \cite{luo2024adaptive}, which is of order $\frac{1}{\sqrt{\log(1/\epsilon)}}$, directly implies the diameter lower bound for any adaptive confidence set.

To prove $r_{\alpha}(\epsilon, 0.05)\gtrsim \frac{1}{\sqrt{\log (n/d)}}+\frac{1}{\sqrt{ \log(n/\log(1/\alpha)) }}$, it suffices to show $r_{\alpha}(0, 0.05)\gtrsim \frac{1}{\sqrt{\log (n/d)}}+\frac{1}{\sqrt{ \log(n/\log(1/\alpha)) }}$ given the inequality $r_{\alpha}(\epsilon, 0.05)\geq r_{\alpha}(0, 0.05)$. It turns out $r_{\alpha}(0, 0.05)$ can be bounded below by the hardness of the following hypothesis testing problem.
\begin{equation}\label{eq:test-def}
	\begin{split} 
		& H_{0}: X_1, \ldots, X_n \overset{iid}\sim P = N(0, I_d) \\ 
		& H_{1}: X_1, \ldots, X_n \overset{iid}\sim P \in \left\{ (1 - \epsilon_{\max})N(ru, I_d) + \epsilon_{\max}Q : Q, u\in\cS^{d-1} \right\},
	\end{split}
\end{equation}
where $\epsilon_{\max} = 0.05$ and $r>0$ is the separation between the null and the alternative. In (\ref{eq:test-def}), one needs not only to test whether the mean is zero, but also to determine whether the data are clean or contaminated. A testing function for (\ref{eq:test-def}) can be constructed from an adaptive confidence set,
\begin{equation}\label{eq:reverse-t-new}
\phi=  \sup_{u\in\cS^{d-1}}  \indi \left( ru  \in  \widehat{\CS}\right) ,
\end{equation}
which rejects the null whenever the confidence set contains some mean vector in the alternative.

\begin{Lemma}\label{prop:CI-imply-test-new}
Suppose $\widehat{\CS}$ is a confidence set that satisfies
\begin{equation*}
		\begin{split}
			\textnormal{(coverage)} &\quad \inf_{\theta, Q} P_{0.05, \theta, Q}\left( \theta \in\widehat{\CS}\right) \geq 1-\alpha,\\
			\textnormal{(diameter at } \epsilon = 0) & \quad P_{0}\left( |\widehat{\CS}| < r  \right) \geq 1-\alpha,
		\end{split}
	\end{equation*} where $P_0 = N(0,I_d)$.
Then, the testing function $\phi$ defined in (\ref{eq:reverse-t-new}) satisfies the following Type-1 and Type-2 error control:
\begin{equation*}
		\begin{split}
			 &  P_{0}\left(\phi=1\right) + \sup_{Q, u\in \cS^{d-1}} P_{0.05, ru,Q}\left(\phi=0\right)\leq 3\alpha.
		\end{split}
	\end{equation*}
\end{Lemma}

In other words, for an $r$ such that the null and the alternative in (\ref{eq:test-def}) cannot be distinguished, it must lower bound $r_{\alpha}(0, 0.05)$. The following proposition characterizes the hardness of the testing problem.

\begin{Proposition}\label{prop:test-lower-bound-gaussian}
There exist universal constants $C,c>0$ such that for any $\alpha \in (0,0.05)$ and $n \geq C\left(d + \log(1/\alpha) \right)$,
	we have
		\begin{equation} \label{ineq:testing-lower-bound-claim}
			\inf_{\phi} \left\{ P_{0 } \left( \phi = 1 \right) + \sup_{Q, u\in \cS^{d-1}} P_{0.05, ru,Q}\left(\phi=0\right) \right\} \geq 4\alpha,
		\end{equation}
        with $r = c \left(\frac{1}{\sqrt{ \log(n/\log(1/\alpha)) }} + \frac{1}{\sqrt{\log (n/d)}} \right)$.
\end{Proposition}
Lemma \ref{prop:CI-imply-test-new} and Proposition \ref{prop:test-lower-bound-gaussian} immediately imply $r_{\alpha}(\epsilon, 0.05) \geq r_{\alpha}(0, 0.05)\gtrsim \frac{1}{\sqrt{\log(n/d)}}+\frac{1}{\sqrt{ \log(n/\log(1/\alpha)) }}$, and therefore establish Theorem \ref{th:lower-bound-Gaussian}.

The proof of Proposition \ref{prop:test-lower-bound-gaussian} relies on a connection between (\ref{eq:test-def}) and the non-Gaussian component analysis (NGCA). The goal of NGCA is to test whether a high-dimensional distribution has a one-dimensional direction whose marginal distribution is different from the standard Gaussian. We first give a definition of the high-dimensional hidden direction distribution, following the formalism of~\cite{diakonikolas2017statistical}.

\begin{Definition}[High-Dimensional Hidden Direction Distribution] \label{def:high-dim-distribution}
For a unit vector $u\in\cS^{d-1}$ and a distribution $A$ on the real line with probability density function $A(\cdot)$, define $P^{A}_u$ to be the distribution over $\mathbb{R}^d$, where $P^{A}_u$ is the product distribution whose orthogonal projection onto the direction of $u$ is $A$, 
and onto the subspace perpendicular to $u$ is the standard $(d-1)$-dimensional normal distribution. 
That is, $P^{A}_u(x) = A(u^\top x) \exp\left(-\|x - (u^\top x)u\|_2^2/2\right)/(2\pi)^{(d-1)/2}$. 
\end{Definition}
NGCA refers to the following hypothesis testing problem:
\begin{equation}\label{eq:NGCA-final}
	\begin{split} 
		& H_{0}: X_1, \ldots, X_n \overset{iid}\sim P = N(0, I_d) \\ 
		& H_{1}: X_1, \ldots, X_n \overset{iid}\sim P\in\left\{P_u^A: u\in\mathcal{S}^{d-1}\right\} \;.
	\end{split}
\end{equation}
NGCA was originally introduced in \cite{blanchard2006search}, and is commonly used to establish computational lower bounds for various high-dimensional problems; see \citep{diakonikolas2017statistical,diakonikolas2023sq} and Chapter~8 of~\cite{diakonikolas2023algorithmic}. In this paper, NGCA is used as a tool 
to prove both the information-theoretic lower bound 
in Proposition \ref{prop:test-lower-bound-gaussian} and the 
computational (SQ) lower bound in Section \ref{sec:GSQlow}. 
The information-theoretic hardness of NGCA under first-moment matching 
is given by the following lemma.

\begin{Lemma}\label{lm:NGCA-stat}
	Suppose $\alpha \in (0,0.05)$ and the non-Gaussian component distribution $A$ satisfies $\int x A(x) dx = 0$ and $ \chi^2(A, N(0,1)) \leq \frac{d + \log(1/\alpha)}{8n} $. Then,
\begin{equation*}
	\inf_{\phi} \left\{ P_0(\phi = 1) + P_{\mu}(\phi = 0) \right\} \geq 4\alpha,
\end{equation*}
where $P_0=N(0,I_d)^{\otimes n}$ and $P_{\mu} = \int (P_u^A)^{\otimes n} d\mu(u)$ with $\mu$ being the uniform measure on $\cS^{d-1}$.
\end{Lemma}

In order to apply Lemma \ref{lm:NGCA-stat} to establish (\ref{ineq:testing-lower-bound-claim}), we specify a contamination distribution $Q$ in the alternative of (\ref{eq:test-def}) to be non-Gaussian only in the direction of $u$. This implies that the alternative class of (\ref{eq:NGCA-final}) is a subset of that of (\ref{eq:test-def}). In particular, for a $u\in\cS^{d-1}$, we consider $Q_u=P_u^q$ for some $q$ to be specified. With this choice, the alternative is clearly an instance of the high-dimensional hidden direction distribution (Definition \ref{def:high-dim-distribution}). That is,
\begin{equation}
(1-\epsilon_{\max})N(ru,I_d) + \epsilon_{\max}Q_u = P_{u}^A,\label{ineq:NGCA-stat}
\end{equation}
where $A(x) = (1 - \epsilon_{\max}  )\phi(x - r) + \epsilon_{\max} q(x)$. The one-dimensional density $q$ is defined by 
	\begin{equation} \label{def:A-NGCA-stat}
		\begin{split}
			q(x)  = \frac{\phi(x) - (1 - \epsilon_{\max})  \phi(x - r) \frac{\indi( |x - r| \leq t )}{ \bbP(|N(0,1)| \leq t) } }{\epsilon_{\max} }.
		\end{split}
	\end{equation}
One can check that (\ref{def:A-NGCA-stat}) is a proper density function as long as $r\lesssim t^{-1}$, and it is easy to verify the first moment condition $\int xA(x) dx = 0$. Moreover, it can be shown that $\chi^2(A, N(0,1)) \lesssim \exp(-t^2)$ through a Gaussian tail analysis, which leads us to set $t \asymp \sqrt{\log(n/(d + \log(1/\alpha)))}$ in order to satisfy the conditions of Lemma \ref{lm:NGCA-stat}. It turns out this choice of $t$ satisfies the requirement $r\lesssim t^{-1}$ for the $r$ specified in Proposition \ref{prop:test-lower-bound-gaussian}, and therefore the desired conclusion (\ref{ineq:testing-lower-bound-claim}) is established.

\section{Computational Lower Bounds}\label{sec:GSQlow}

Despite its statistical optimality, the adaptive confidence set (\ref{eq:adaptive-CI}) requires exponential time to compute. 
Before presenting our algorithm in Section \ref{sec:quasi-poly-algorithm}, we establish a computational lower bound showing that achieving the statistically optimal rate requires super-polynomial computational complexity, at least within the Statistical Query framework.

\subsection{Statistical Query Confidence Sets} \label{sec:SQ-model}

The Statistical Query (SQ) model \citep{kearns1998efficient} is a common framework for providing rigorous evidence of computational barriers in high-dimensional statistical inference problems \citep{feldman2017statistical,diakonikolas2017statistical,diakonikolas2023sq,diakonikolas2026algorithms}. In the SQ framework, algorithms do not have direct access to i.i.d.\ samples from a distribution $P$. Instead, they interact with an SQ oracle, which returns approximate expectations of queried functions under the distribution $P$.

To motivate the framework, consider a bounded query function $q$. Given i.i.d.\ samples $\{X_i\}_{i=1}^n$ from $P$, the empirical average satisfies
\[
\left|
\frac{1}{n}\sum_{i=1}^n q(X_i)
-
\mathbb{E}_{X\sim P}[q(X)]
\right|
=
O_{\mathbb{P}}\left(\frac{1}{\sqrt{n}}\right).
\]
Thus, the empirical average can be viewed as an approximation to the population quantity $\mathbb{E}_{X\sim P}[q(X)]$. More generally, an SQ oracle responds to a query $q$ with a value that is sufficiently close to $\mathbb{E}_{X\sim P}[q(X)]$. One example of such an oracle is the above empirical average computed from i.i.d.~samples, though the SQ framework is more general.

An SQ algorithm is only allowed to interact with the distribution through adaptive 
queries to the oracle. The number of oracle queries required by an algorithm serves 
as a surrogate for its computational complexity. The SQ framework naturally 
captures many commonly used procedures in statistics and optimization, including 
including many first-order methods for stochastic convex optimization; see, e.g.,~\cite{FeldmanGV21} for a discussion.

We consider the following standard SQ oracle.

\begin{Definition}[STAT Oracle]\label{def:vstat}
	Let $P$ be a distribution over a domain $\cX$. A statistical query is a bounded function $q: \cX \to [-1, 1]$. We define 
	$\mathtt{STAT}_{P, \tau}$ to be an oracle that, given any statistical query $q$, outputs a value $v = \mathtt{STAT}_{P, \tau}(q)$ satisfying $|v - \bbE_{X \sim P}[q(X)]| \leq \tau$, where $\tau \in (0,1)$ is the tolerance parameter of the SQ oracle.
\end{Definition}

This definition is abstract and does not explicitly involve samples generated from $P$. In a statistical setting where i.i.d.~samples are available, a $\mathtt{STAT}_{P,\tau}$ oracle can be simulated by the empirical average of $q(X_1),\dots,q(X_n)$ using $n=\Theta(1/\tau^2)$ samples from $P$. Thus, if an SQ algorithm requires access to $\mathtt{STAT}_{P,\tau}$ with smaller tolerance $\tau$, this corresponds to higher sample complexity. More generally, one may interpret the sample complexity required to simulate $\mathtt{STAT}_{P,\tau}$ as scaling polynomially in $1/\tau$.
Formally, let $\mathcal{F}$ denote the collection of all statistical queries, namely all functions bounded in $[-1,1]$. An SQ oracle $\mathtt{STAT}_{P,\tau}$ is a map from $\mathcal{F}$ to $\mathbb{R}$ such that
$
\left|
\mathtt{STAT}_{P,\tau}(q)
-
\mathbb{E}_{X\sim P}[q(X)]
\right|
\le \tau
$
for every $q \in \mathcal{F}$. We define $\mathtt{Oracle}_{P,\tau}$ to be the set of all such oracles:
\[
\mathtt{Oracle}_{P,\tau}
=
\left\{
g:\mathcal{F}\to\mathbb{R}
\;\middle|\;
\forall q\in\mathcal{F},
\;
|g(q)-\mathbb{E}_{X\sim P}[q(X)]|
\le \tau
\right\}.
\]

A statistical query algorithm $\mathcal{A}=\mathcal{A}_{k,\tau}$, parameterized by the number of queries $k$ and tolerance parameter $\tau$, interacts with an unknown oracle $\mathtt{STAT}_{P,\tau}\in\mathtt{Oracle}_{P,\tau}$ over $k$ rounds. At round $i\in[k]$, the algorithm adaptively chooses a query $q_i\in\mathcal{F}$ based on the previous query-response pairs $(q_j,v_j)_{j=1}^{i-1}$, and receives the response
$
v_i=\mathtt{STAT}_{P,\tau}(q_i).
$
After $k$ rounds, the algorithm outputs a value $\mathcal{A}(\mathtt{STAT}_{P,\tau})$ based on the full interaction history $(q_j,v_j)_{j=1}^k$. We use $\mathrm{SQ}(k,\tau)$ to denote the collection of all SQ algorithms that make at most $k$ oracle queries with tolerance parameter $\tau$. The restriction to SQ algorithms rules out procedures that rely on fine-grained access to individual samples rather than aggregate statistics.

Depending on the statistical task, the output $\mathcal{A}(\mathtt{STAT}_{P,\tau})$ may take values in $\mathbb{R}^d$ for estimation problems or in $\{0,1\}$ for testing problems. For confidence set construction, the output is itself a set, which we denote by $\widehat{\CS}(\mathtt{STAT}_{P,\tau})$ and refer to as an SQ confidence set.

We define $\cP_{\epsilon,\theta}$ to be the Huber contamination neighborhood around $N(\theta,I_d)$. That is,
$$\cP_{\epsilon,\theta}=\{(1-\epsilon)N(\theta,I_d)+\epsilon Q: Q\}.$$

\begin{Definition}[SQ Confidence Set]\label{def:SQCS}
Given some $\epsilon\in[0,\epsilon_{\max}]$, we say that an SQ algorithm $\wh{\CS}$ is an adaptive robust confidence set with $k$ queries, tolerance $\tau$, and diameter $r$ at contamination level $\epsilon$, if $\wh{\CS}\in \mathrm{SQ}(k,\tau)$ and
\begin{eqnarray*}
\inf_{\theta}\inf_{P \in \cP_{\epsilon_{\max},\theta} } \,  \,\inf_{\mathtt{STAT}_{P,\tau} \in \mathtt{Oracle}_{P,\tau} } \bbP\left( \theta \in \widehat{\CS}(\mathtt{STAT}_{P,\tau})\right) \geq 0.95, \\
\inf_{\theta}\inf_{P \in \cP_{\epsilon,\theta} } \inf_{\mathtt{STAT}_{P,\tau} \in \mathtt{Oracle}_{P,\tau}} \bbP \left( |\widehat{\CS}(\mathtt{STAT}_{P,\tau})| < r \right) \geq 0.95,
\end{eqnarray*}
where the probability is taken over the randomness of the algorithm.
\end{Definition}

\subsection{SQ Lower Bounds for Confidence Set Diameter} \label{sec:SQ-lower-bound}

We now state the main computational lower bound for adaptive robust confidence sets in the SQ framework.
\begin{Theorem} \label{th:SQ-lower-bound}
Let $\epsilon\in[0,0.05]$. There exist universal constants $c,c_1$ such that for any $0<m\leq d^{c}$, any SQ confidence set $\wh{\CS}\in \mathrm{SQ}(k,\tau)$ (Definition \ref{def:SQCS}) with diameter $r$ at contamination level $\epsilon$ satisfying
$$r\leq \frac{c_1}{\sqrt{m}}$$
must require either
	\begin{equation} \label{ineq:SQ-option}
		\begin{split}
			&\text{(i) number of queries  }k \geq 2^{d^{\Omega(1)}}, \text{ or }\\
			&\text{(ii) tolerance of the SQ oracle } \tau =d^{- \Omega(m)}. 
		\end{split}
	\end{equation}
\end{Theorem}

Theorem \ref{th:SQ-lower-bound} shows that any adaptive confidence set satisfying the coverage property and achieving diameter $O(m^{-1/2})$ must either use $2^{d^{\Omega(1)}}$ oracle queries or require tolerance $\tau = d^{-\Omega(m)}$. As discussed in Section \ref{sec:SQ-model}, simulating an SQ oracle with such a tolerance requires $d^{\Omega(m)}$ samples. Since processing $d^{\Omega(m)}$ samples already incurs computational cost at least polynomial in $d^{\Omega(m)}$, the theorem rules out a broad class of polynomial-time algorithms for constructing a shrinking confidence set, as shrinking diameter requires $m\to\infty$.
In particular, achieving the statistically optimal diameter $\frac{1}{\sqrt{\log(1/\epsilon)}}$ requires $m \gtrsim \log(1/\epsilon)$. Consequently, sample and computational complexity at least
$d^{\Omega\!\left(\log(1/\epsilon)\right)}$
are necessary within the SQ framework. This stands in sharp contrast to the information-theoretically optimal construction in Section \ref{sec:optimal-ARCS-Gaussian}, which achieves the same diameter using only
$n \asymp d \cdot \mathrm{poly}\!\bigl(1/\epsilon\bigr)$
samples, albeit with exponential computational complexity.
In summary, an algorithm with sample and computational complexity $d^{O(m)}$ can achieve diameter at best of order $O(m^{-1/2})$.

Similar to the proof of the information-theoretic lower bound in Theorem \ref{th:lower-bound-Gaussian}, the proof of Theorem \ref{th:SQ-lower-bound} also relies on the hardness of a hypothesis testing problem. In particular, the testing problem (\ref{eq:test-def}) continues to serve as the hard instance. The main difference is that, within the SQ framework, the algorithm no longer has direct access to i.i.d.\ samples. A formal definition of SQ testing is given below.

\begin{Definition}[SQ Testing]\label{def:SQtest}
Let $P_0$ be a distribution and let $\mathcal{P}$ be a collection of distributions over a common support such that $P_0\notin \mathcal{P}$. We say that an SQ algorithm $\mathcal{A}$ solves the testing problem
\[
H_0: P=P_0
\qquad\text{versus}\qquad
H_1: P\in\mathcal{P}
\]
with $k$ queries and tolerance $\tau$ if $\mathcal{A}\in\mathrm{SQ}(k,\tau)$ and
\[
\sup_{P\in\{P_0\}\cup\mathcal{P}}
\sup_{\mathtt{STAT}_{P,\tau}\in\mathtt{Oracle}_{P,\tau}}
\mathbb{P}
\left(
\mathcal{A}(\mathtt{STAT}_{P,\tau})
\neq
\indi\{P\neq P_0\}
\right)
\le 0.1,
\]
where the probability is taken over the randomness of the algorithm.
\end{Definition}

Analogous to Lemma \ref{prop:CI-imply-test-new}, an SQ adaptive confidence set also induces an SQ test for the SQ version of (\ref{eq:test-def}).

\begin{Lemma}[SQ Confidence Set Implies SQ Testing]
\label{lm:SQ-CS-testing}
For any $\epsilon\in[0,\epsilon_{\max}]$, suppose there exists an SQ confidence set with diameter $r$, $k$ queries, and tolerance $\tau$ in the sense of Definition \ref{def:SQCS}. Then there exists an SQ test that solves the testing problem
\begin{equation}
P_0=N(0,I_d)
\qquad\text{versus}\qquad
\mathcal{P}
=
\left\{
(1-\epsilon_{\max})N(ru,I_d)+\epsilon_{\max}Q
:
Q,\;
u\in\mathcal{S}^{d-1}
\right\} \label{eq:SQ-test-CI}
\end{equation}
in the sense of Definition \ref{def:SQtest}, with the same number of queries $k$ and tolerance $\tau$.
\end{Lemma}

\subsection{Reduction from NGCA}\label{sec:sq-ngca}

By Lemma \ref{lm:SQ-CS-testing}, it suffices to establish the SQ hardness of (\ref{eq:SQ-test-CI}) in order to prove Theorem \ref{th:SQ-lower-bound}. To this end, we consider a special subclass of alternatives in (\ref{eq:SQ-test-CI}) arising from the non-Gaussian component analysis (NGCA) model (\ref{eq:NGCA-final}). General SQ hardness results for NGCA under moment matching were established in~\cite{diakonikolas2017statistical} and further developed in \cite{diakonikolas2023sq}. Here we will leverage Theorem 1.5 of \cite{diakonikolas2023sq}.

\begin{Lemma}[Theorem 1.5 of \cite{diakonikolas2023sq}] \label{lm:SQ-NGCA}
Consider the testing problem
\begin{equation}
P_0=N(0,I_d)
\qquad\text{versus}\qquad
\mathcal{P}
=
\left\{
P_u^A
:
u\in\mathcal{S}^{d-1}
\right\}, \label{eq:NGCA-SQ}
\end{equation}
and let $m\leq d^{c}$ be a nonnegative integer where $c> 0$ is sufficiently small.
Suppose the distribution $A$ matches the first $m$ moments of $N(0,1)$. Then, every SQ
algorithm that solves the testing problem with $k$ queries and tolerance $\tau$ in the sense of Definition \ref{def:SQtest} must satisfy either
	\begin{equation} \label{ineq:SQ-option-NGCA}
		\begin{split}
			&\text{(i)  number of queries }k \geq 2^{d^{\Omega(1)}}, \text{ or }\\
			&\text{(ii)  tolerance of the SQ oracle  } \tau =d^{-\Omega(m)}. 
		\end{split}
	\end{equation} 
\end{Lemma}

We now construct a distribution $Q_u$ such that
\begin{equation}
(1-\epsilon_{\max})N(ru,I_d)+\epsilon_{\max}Q_u=P_u^A,
\label{eq:SQ-con}
\end{equation}
for some distribution $A$ whose first $m$ moments match those of $N(0,1)$. This implies that the alternative in (\ref{eq:NGCA-SQ}) is a subset of the alternative in (\ref{eq:SQ-test-CI}), so that the SQ hardness of (\ref{eq:NGCA-SQ}) directly yields Theorem \ref{th:SQ-lower-bound}.
Define $Q_u=P_u^q$, so that (\ref{eq:SQ-con}) holds with
\begin{equation}
A(x)
=
(1-\epsilon_{\max})\phi(x-r)+\epsilon_{\max}q(x).\label{eq:sq-con-A}
\end{equation}
The one-dimensional density $q$ is specified by
\begin{equation} \label{eq:A-NGCA-SQ}
	q(x) = \frac{\phi(x) - (1 - \epsilon_{\max}) \phi(x-r) }{\epsilon_{\max}} \indi(|x| \leq t) + p(x) \indi(|x| \leq t/4),
\end{equation}
where $t$ is some tuning parameter, and $p(x)$ is some polynomial of degree $m$. The construction (\ref{eq:A-NGCA-SQ}) is more flexible than (\ref{def:A-NGCA-stat}) that is used to establish the information-theoretic lower bound. While (\ref{def:A-NGCA-stat}) only matches the first moment of $N(0,1)$, here we will choose the polynomial $p(x)$ to match the first $m$ moments. The choices of $t$ and $p(x)$ need to satisfy
\begin{align}
\label{eq:pcon1} (i)\quad\, & \phi(x) - (1-\epsilon_{\max})\phi(x-r) \geq 0, \qquad\qquad\qquad\text{ for }x\in [-t,-t/4]\cup[t/4,t], \\
\label{eq:pcon2} (ii)\quad\, & \phi(x) - (1-\epsilon_{\max})\phi(x-r) + \epsilon_{\max}p(x) \geq 0, \quad \text{ for }x\in[-t/4,t/4], \\
\nonumber (iii)\quad\, & \int_{-t/4}^{t/4}x^kp(x)dx = \frac{1}{\epsilon_{\max}}\int_{|x|>t}x^k\phi(x)dx - \frac{1-\epsilon_{\max}}{\epsilon_{\max}}\int_{|x|>t}x^k\phi(x-r)dx, \\
\label{eq:pcon3}& \qquad\qquad\qquad\qquad\qquad\qquad\qquad\qquad\qquad\qquad \text{for }k=0,1,\cdots,m.
\end{align}
The conditions (\ref{eq:pcon1}) and (\ref{eq:pcon2}) ensure that the formula (\ref{eq:A-NGCA-SQ}) is non-negative. Condition (\ref{eq:pcon3}) with $k=0$ is equivalent to $\int q(x)dx=1$, which implies that (\ref{eq:A-NGCA-SQ}) is a valid density function. The moment matching condition is also encoded in (\ref{eq:pcon3}) with $k=1,\cdots,m$.

Similar to the analysis of (\ref{def:A-NGCA-stat}), one can verify that (\ref{eq:pcon1}) holds whenever $t\lesssim r^{-1}$. Moreover, building on  \cite{diakonikolas2023algorithmic} (see Lemma 8.18 and Exercise 8.3), we can show that there exists a degree-$m$ polynomial $p(x)$ satisfying (\ref{eq:pcon2}) and (\ref{eq:pcon3}) provided that $t\gtrsim \sqrt{m}$. Therefore, the constraints (\ref{eq:pcon1})--(\ref{eq:pcon3}) can be satisfied simultaneously by choosing $t\asymp \sqrt{m}$ and $r\asymp m^{-1/2}$.
Applying Lemma \ref{lm:SQ-NGCA} to the resulting NGCA instance, together with Lemma \ref{lm:SQ-CS-testing}, yields the conclusion of Theorem \ref{th:SQ-lower-bound}.

\section{A Moment-Matching Algorithm} \label{sec:quasi-poly-algorithm}
Theorem \ref{th:SQ-lower-bound} shows that any adaptive confidence set with diameter $O(m^{-1/2})$ must either incur computational complexity $2^{d^{\Omega(1)}}$, or require sample complexity of order $d^{\Omega(m)}$. While the information-theoretically optimal construction (\ref{eq:adaptive-CI}) falls into the first category, the goal of this section is to construct an algorithm that matches the second possibility, using $d^{O(m)}$ samples and $d^{O(m)}$ computational complexity. The parameter $m$ controls the accuracy–complexity tradeoff. For fixed $m$, the algorithm uses polynomially many samples and polynomial time in $d$. For $m = \log(1/\epsilon)$, this gives a quasi-polynomial time algorithm with an optimal diameter guarantee.

\subsection{Population-Level Analysis}\label{sec:pop}
To motivate the algorithmic development, we first consider an idealized setting in which we have direct access to the probability distribution
\[
P:=P_{\epsilon,\theta,Q}
=
(1-\epsilon)N(\theta,I_d)+\epsilon Q,
\]
corresponding to the limit of infinitely many i.i.d.\ samples. We seek to construct confidence sets in the form of $\ell_2$-balls. This requires specifying both a center and a radius. The center can be obtained using any polynomial-time robust estimator of $\theta$. The main challenge is therefore to determine an appropriate radius, which amounts to quantifying the uncertainty of the robust estimator.

Let $\widehat{\theta}$ be a robust estimator of $\theta$. Our goal is to find a radius $r$ such that $\|\widehat{\theta}-\theta\|\le r$, in which case the ball $B(\widehat{\theta},r)$ satisfies the desired coverage property. A natural starting point is the quantity
\begin{equation}
\label{ineq:hat-epsilon}
\widetilde{\epsilon}
=
\inf\left\{
\delta\in[0,1]:
P\ge (1-\delta)N(\widehat{\theta},I_d)
\right\}.
\end{equation}
Recall that the relation $P\ge (1-\delta)N(\widehat{\theta},I_d)$ is equivalent to the existence of a distribution $Q'$ such that
$P=(1-\delta)N(\widehat{\theta},I_d)+\delta Q'$.
Thus, $\widetilde{\epsilon}$ measures how well $P$ can be explained as a contaminated Gaussian centered at $\widehat{\theta}$, and may be viewed as a data-dependent surrogate for the unknown contamination proportion $\epsilon$. Unlike $\epsilon$, however, the quantity $\widetilde{\epsilon}$ is defined entirely in terms of $P$ and $\widehat{\theta}$.

\begin{Lemma} \label{lm:hat-epsilon-guarantee}
	Suppose $P = P_{\epsilon, \theta, Q} $ for some $\epsilon \in [0, 0.05], \theta \in \bbR^d$ and some distribution $Q$. Suppose $\wh{\theta}\in\mathbb{R}^d$ satisfies $\|\wh{\theta} - \theta\| \leq 1$. For the $\wt{\epsilon}$ defined in \eqref{ineq:hat-epsilon}, we have 
    \[ \|\wh{\theta}- \theta\| \leq \frac{C}{\sqrt{\log(1 /\wt{\epsilon} )} } \]
   for some universal constant $C>0$.
\end{Lemma}

According to Lemma \ref{lm:hat-epsilon-guarantee}, the $\ell_2$ ball
$B\left(\widehat{\theta},
\frac{C}{\sqrt{\log(1/\widetilde{\epsilon})}}
\right)$
is an adaptive confidence set satisfying the coverage property. The drawback is that the resulting radius can be much larger than necessary. To illustrate this issue, consider the example $P=(1-\epsilon)N(\theta,I_d)+\epsilon Q$
with $Q=N(\theta,I_d)$. In this case, $P=N(\theta,I_d)$ regardless of the value of $\epsilon$. Then the condition $P\ge (1-\delta)N(\widehat{\theta},I_d)$ in (\ref{ineq:hat-epsilon}) becomes
$N(\theta,I_d)\ge (1-\delta)N(\widehat{\theta},I_d)$.
Unless $\widehat{\theta}=\theta$, this inequality can only hold with $\delta=1$. Therefore, $\widetilde{\epsilon}=1$ by definition, even though the underlying distribution has no contamination. This example shows that $\widetilde{\epsilon}$ is highly sensitive to the error of $\widehat{\theta}$ and may significantly overestimate the effective contamination level.

To overcome the instability of (\ref{ineq:hat-epsilon}), we replace the Gaussian distribution \(N(\widehat{\theta},I_d)\) by a larger class of distributions that agree with it only through low-order moments. Let
\[
\mathcal{P}_{m,d}
=
\left\{
p:\mathbb{R}^d\to\mathbb{R}
|
\deg(p)\le m
\right\}
\]
be the collection of polynomials on \(\mathbb{R}^d\) of degree at most \(m\). For two distributions \(P_1\) and \(P_2\) on \(\mathbb{R}^d\), we write
\[
P_1 \equiv_m P_2
\]
if
\[
\mathbb{E}_{X\sim P_1}[p(X)]
=
\mathbb{E}_{X\sim P_2}[p(X)]
\]
for every \(p\in\mathcal{P}_{m,d}\).
Motivated by the SQ lower bound construction in Section \ref{sec:sq-ngca}, we define
\begin{equation}
\label{ineq:hat-epsilon-relax}
\widetilde{\epsilon}
=
\inf
\left\{
\delta\in[0,1]:
P\ge (1-\delta)D
\text{ for some }
D\equiv_{m_\delta}N(\widehat{\theta},I_d)
\right\},
\end{equation}
where \(m_\delta\) is the largest even integer not exceeding \(\log(1/\delta)/9\).
Compared with (\ref{ineq:hat-epsilon}), the definition (\ref{ineq:hat-epsilon-relax}) only requires the uncontaminated component to match finitely many moments of \(N(\widehat{\theta},I_d)\), with the number of matched moments increasing as \(\delta\) decreases. The following theorem shows that this relaxation retains the coverage guarantee while also providing an upper bound on \(\widetilde{\epsilon}\).

\begin{Theorem}
\label{th:hat-epsilon-guarantee}
Suppose $P=P_{\epsilon,\theta,Q}$
for some \(\epsilon\in[0,0.05]\), \(\theta\in\mathbb{R}^d\), and some distribution \(Q\). Assume that \(\widehat{\theta}\in\mathbb{R}^d\) satisfies
$\|\widehat{\theta}-\theta\|\le 1$.
Then the quantity \(\widetilde{\epsilon}\) defined in (\ref{ineq:hat-epsilon-relax}) satisfies
\[
\|\widehat{\theta}-\theta\|
\le
\frac{C}{\sqrt{\log(1/\widetilde{\epsilon})}}
\]
and
\[
\widetilde{\epsilon}
\le
C\Bigl(
\sqrt{\|\widehat{\theta}-\theta\|}
+
\epsilon
\Bigr)
\]
for some universal constant \(C>0\).
\end{Theorem}

As a direct consequence of Theorem \ref{th:hat-epsilon-guarantee}, the confidence set $\widehat{\CS}
=
B\left(
\widehat{\theta},
\frac{C}{\sqrt{\log(1/\widetilde{\epsilon})}}
\right)$
satisfies the coverage property. Moreover, its diameter is bounded by
\[
|\widehat{\CS}|
\lesssim
\frac{1}{\sqrt{\log(1/\|\widehat{\theta}-\theta\|)}}
+
\frac{1}{\sqrt{\log(1/\epsilon)}}.
\]
Therefore, whenever $\widehat{\theta}$ is consistent,
the confidence set automatically adapts to the unknown contamination level $\epsilon$
and achieves diameter
$O\!\left(
\frac1{\sqrt{\log(1/\epsilon)}}
\right)$
up to an additional term determined by the estimation error
$\|\widehat{\theta}-\theta\|$.

We note that the upper bound $\widetilde{\epsilon}
\le
C\Bigl(
\sqrt{\|\widehat{\theta}-\theta\|}
+
\epsilon
\Bigr)$
is closely connected to the SQ lower bound construction in Section \ref{sec:sq-ngca}. To establish this bound, it suffices to construct a distribution $D$ satisfying
\[
P_{\epsilon,\theta,Q}
\ge
(1-\delta)D
\qquad\text{and}\qquad
D\equiv_{m_\delta}N(\widehat{\theta},I_d),
\]
with $\delta
=
C\Bigl(
\sqrt{\|\widehat{\theta}-\theta\|}
+
\epsilon
\Bigr)$.
By a translation argument, this is equivalent to finding a distribution $D$ such that
\begin{equation}
\label{eq:D-condition2}
P_{\epsilon,\theta-\widehat{\theta},Q}
\ge
(1-\delta)D
\qquad\text{and}\qquad
D\equiv_{m_\delta}N(0,I_d).
\end{equation}
Recall the hidden-direction distribution $P_u^A$ defined in (\ref{eq:SQ-con}) with $A$ given by (\ref{eq:sq-con-A}) and (\ref{eq:A-NGCA-SQ}). Consider the hidden-direction distribution $D=P_u^A$, where
$u=\frac{\theta-\widehat{\theta}}
{\|\theta-\widehat{\theta}\|}$
and where the parameters $\epsilon_{\max}$ and $r$ in the definition of $A$ are replaced by $\delta$ and $\|\widehat{\theta}-\theta\|$, respectively.
Then the condition $D\equiv_{m_\delta}N(0,I_d)$
is automatically satisfied because $A$ is constructed to match the moments of $N(0,1)$. Remarkably, the same construction also satisfies
$P_{\epsilon,\theta-\widehat{\theta},Q}
\ge
(1-\delta)D$,
which leads to the desired upper bound on $\widetilde{\epsilon}$.

\subsection{Finite-Sample Algorithm and Theoretical Guarantees} \label{sec:quasi-finite-sample}

Having developed the population-level quantity (\ref{ineq:hat-epsilon-relax}), we now introduce its finite-sample counterpart. We work with i.i.d.~observations $X_1,\ldots,X_n\sim P_{\epsilon,\theta,Q}$ together with a pilot location estimator $\widehat{\theta}$. The estimator $\widehat{\theta}$ is assumed to be independent of $(X_1,\ldots,X_n)$, a property that can always be achieved by sample splitting.

The condition $P\geq (1-\delta)D$ means that samples generated from $D$ can be obtained by rejection sampling from the distribution $P$. More concretely, given i.i.d.\ observations $X_1,\ldots,X_n$ from $P=P_{\epsilon,\theta,Q}$, there exists a subset $S\subset[n]$ with cardinality approximately $n(1-\delta)$ such that $\{X_i\}_{i\in S}$ can be regarded as samples generated from $D$. Consequently, the population moment $\mathbb{E}_{X\sim D}[p(X)]$ can be estimated by
$\frac{1}{|S|}\sum_{i\in S} p(X_i)$.
Searching over all such subsets is computationally intractable. We therefore replace subsets by weights and define
\[
\mathcal{W}_{n,\delta}
=
\left\{
w\in\left[0,\frac{1}{n(1-\delta)}\right]^n:
\sum_{i=1}^n w_i=1
\right\}.
\]
For some $w\in\mathcal{W}_{n,\delta}$,
the weighted empirical moment $\sum_{i=1}^n w_i p(X_i)$
can then be viewed as a relaxation of the empirical average over a subset of size approximately $n(1-\delta)$.

The condition $D\equiv_{m_\delta}N(\widehat{\theta},I_d)$ means that
$\mathbb{E}_{X\sim D}[p(X)]
=
\mathbb{E}_{X\sim N(\widehat{\theta},I_d)}[p(X)]$ 
for every polynomial $p\in\mathcal{P}_{m_\delta,d}$. To obtain a finite-sample analogue, we replace population moments by weighted empirical moments and exact moment matching by high-probability error bounds. Moreover, instead of checking every polynomial in $\mathcal{P}_{m_\delta,d}$, we only check a basis.
Define
\[
\mathcal{H}_{m,d}
=
\left\{
h_\alpha(x)
=
\prod_{j=1}^d h_{\alpha_j}(x_j)
:
\alpha\in\mathbb{N}^d,\,
\sum_{j=1}^d \alpha_j\le m
\right\},
\]
where
\[
h_k(x)
=
\frac{(-1)^k}{\sqrt{k!}}
e^{x^2/2}
\frac{d^k}{dx^k}
e^{-x^2/2}
\]
is the $k$th normalized Hermite polynomial \citep{szeg1939orthogonal}. The collection $\mathcal{H}_{m,d}$ has cardinality $\binom{d+m}{m}$ and forms an orthonormal basis of $\mathcal{P}_{m,d}$ under the standard Gaussian measure.
For $m\in\mathbb{N}$, define the discrepancy functional
\[
T_m(w,\widehat{\theta})
=
\max_{h\in\mathcal{H}_{m,d}}
\frac{\left|\sum_{i=1}^n w_i h(X_i-\widehat{\theta})-\mathbb{E}_{X\sim N(0,I_d)}[h(X)]\right|}{\sqrt{\binom{d+m}{m}}}.
\]
The quantity $T_m(w,\widehat{\theta})$ measures how well the weighted empirical moments match those of a Gaussian distribution centered at $\widehat{\theta}$ up to degree $m$.

We are now ready to define the finite-sample counterpart of (\ref{ineq:hat-epsilon-relax}). Recall that our goal is to develop an algorithm using $d^{O(m)}$ samples and $d^{O(m)}$ computational complexity. For a given parameter $m>0$, we set $n=\lceil d^m\rceil$
and redefine $m_\delta$ as
\begin{equation}\label{def:m-delta}
	m_\delta = \arg\max\left\{
a \in \mathbb{Z}: a \text{ is even and } a \leq 
\left\lfloor
\frac{1}{9}\log(1/\delta)
\wedge
\frac{m}{4}
\right\rfloor \right\}.
\end{equation}
We then define
\begin{equation}
\label{eq:hat-epsilon-finite}
\widehat{\epsilon}
=
\inf
\left\{
\delta\in\Delta:
\exists w\in\mathcal{W}_{n,\delta}
\text{ such that }
T_{m_\delta}(w,\widehat{\theta})
\le
2\left(1+\frac{4}{\alpha}\right)
\sqrt{
\frac{1}
     {n(1-\delta)}
}
\right\},
\end{equation}
where
\begin{equation}
\Delta
=
\{
\delta_{\min},
2\delta_{\min},
4\delta_{\min},
\ldots,
2^{\lfloor\log_2(1/\delta_{\min})\rfloor}\delta_{\min},
1
\}, \label{eq:grid-Delta}
\end{equation}
with $\delta_{\min}
=
4\sqrt{\frac{\log(8/\alpha)}{n}}$. Intuitively, $\wh{\epsilon}$ is the smallest contamination level for which one can assign weights to the observations so that the resulting weighted empirical moments agree with those of a Gaussian distribution centered at $\wh{\theta}$.

The quantity $\widehat{\epsilon}$ defined in (\ref{eq:hat-epsilon-finite}) is computationally tractable. The grid $\Delta$ contains at most $O(\log n)$ points. For each fixed $\delta\in\Delta$, the condition $T_{m_\delta}(w,\widehat{\theta})
\le
2\left(1+\frac{4}{\alpha}\right)
\sqrt{\frac{1}{n(1-\delta)}}$
can be expressed as a collection of linear constraints in the variable $w$. Since the constraint $w\in\mathcal{W}_{n,\delta}$ is also linear, determining whether a given $\delta$ is feasible reduces to a linear programming feasibility problem.
For each $\delta\in\Delta$, the corresponding feasibility problem has $n=\lceil d^m\rceil$ variables and
\[
2\left(|\mathcal{H}_{m_\delta,d}|+n+1\right)
=
2\left(\binom{d+m_\delta}{m_\delta}+n+1\right)
=
d^{O(m)}
\]
linear constraints. Therefore, the feasibility problem can be solved in time polynomial in $d^{O(m)}$. Using state-of-the-art algorithms for linear programming \citep{van2020deterministic,cohen2021solving}, the overall computational complexity of evaluating $\widehat{\epsilon}$ is bounded by $d^{O(m)}$. The computation of $\wh{\epsilon}$ and the corresponding confidence set is outlined in Algorithm \ref{alg:ARCS}.

\begin{algorithm}[ht]
\DontPrintSemicolon
	\SetKwInOut{Input}{Input}\SetKwInOut{Output}{Output}
\textbf{Input: }observations $X_1,\ldots,X_n$, pilot estimator $\widehat{\theta}$, $m > 0$, confidence level $\alpha$.\\
 \textbf{Output: }confidence set $\wh{\CS}$.

\nl Construct the grid $\Delta$ according to (\ref{eq:grid-Delta}).

\nl For each $\delta\in\Delta$, check whether there exists $w\in\mathcal{W}_{n,\delta}$ satisfying
$$T_{m_\delta}(w,\widehat{\theta})
\le
2\left(1+\frac{4}{\alpha}\right)
\sqrt{\frac{1}{n(1-\delta)}}.$$

\nl Set $\wh{\epsilon}$ to be the smallest feasible $\delta\in\Delta$.

\nl Set $\wh{\CS} = B\left(\wh{\theta}, \frac{11}{\sqrt{ m_{\wh{\epsilon}}}} \right)$. 
\caption{Adaptive Robust Confidence Set}
\label{alg:ARCS}
\end{algorithm}

The following theorem shows that Algorithm \ref{alg:ARCS} inherits the adaptive guarantees suggested by the population-level analysis in Section \ref{sec:pop}.

\begin{Theorem}\label{thm:finite-sample-analysis}
There exist universal constants $C,C'>0$, such that for any $\alpha\in(0,1)$ and $n=\lceil d^m\rceil\geq C/\alpha^4$, as long as the pilot estimator $\wh{\theta}$, which is independent of the observations, satisfies
\begin{equation}
\inf_{\epsilon\in[0,0.05],\theta,Q}P_{\epsilon,\theta,Q}\left(\|\wh{\theta}-\theta\|\leq C\left(\sqrt{\frac{d}{n}}+\epsilon + \sqrt{\frac{\log(1/\alpha)}{n}}\right)^{0.1}\right) \geq 1-\frac{\alpha}{6}, \label{eq:pilot}
\end{equation}
the confidence set $\wh{\CS}$ output by Algorithm \ref{alg:ARCS} satisfies
	\begin{equation*}
		\begin{split}
			 &\quad \inf_{ \epsilon \in [0, 0.05], \theta, Q} P_{\epsilon, \theta, Q}\left( \theta \in\widehat{\CS} \right) \geq 1-\alpha,\\
			 & \quad\inf_{\epsilon \in [0, 0.05], \theta, Q } P_{\epsilon, \theta, Q}\left( |\widehat{\CS}| \leq C' \left(\frac{1}{\sqrt{\log (n/d )} }+\frac{1}{\sqrt{ \log(1/\epsilon) }}+ \frac{1}{\sqrt{ \log(n/\log(1/\alpha)) }} +\frac{1}{\sqrt{m}}\right)\right) \geq 1-\alpha.
		\end{split}
	\end{equation*} 
\end{Theorem}

We note that the diameter upper bound contains an additional term of order $m^{-1/2}$ compared with the information-theoretic lower bound. This additive $O(m^{-1/2})$ term is unavoidable for all Statistical Query algorithms with complexity $d^{O(m)}$ in view of Theorem \ref{th:SQ-lower-bound}.

The condition (\ref{eq:pilot}) on the pilot estimator is very mild. The exponent $0.1$ is not essential and can be replaced by any fixed constant in $(0,1]$. In particular, replacing $0.1$ by $1$ corresponds to the information-theoretically optimal estimation rate, which can be achieved by the Tukey median \citep{chen2018robust}. However, the Tukey median is not computationally feasible. Among polynomial-time algorithms, both \cite{diakonikolas2024near} and \cite{dalalyan2022all} satisfy (\ref{eq:pilot}), although the procedure of \cite{diakonikolas2024near} requires a slight modification based on Lepski's method \citep{lepskii1991problem,lepskii1992asymptotically,jain2022robust} in order to adapt to the unknown contamination level $\epsilon$.

Taken together, Theorems \ref{th:SQ-lower-bound} and \ref{thm:finite-sample-analysis} characterize the computational--statistical tradeoff for adaptive 
confidence sets under Huber contamination. The lower bound shows that achieving 
confidence set diameter $O(m^{-1/2})$ requires computational complexity at least 
$d^{\Omega(m)}$ within the SQ framework, while Algorithm \ref{alg:ARCS} attains an 
essentially matching upper bound up to constant factors. Thus, the proposed 
procedure provides a continuum of confidence sets that interpolate between 
statistical accuracy and computational efficiency, and the SQ lower bound shows 
that this tradeoff is essentially optimal.



\vskip 0.2in
\bibliographystyle{apalike}
\bibliography{reference}

		\newpage
		\appendix
		\
			\begin{center}
			{\LARGE Supplement to "Adaptive Robust Confidence Sets in High Dimensions" 
				
			}		
			\medskip
%
			\medskip
		\end{center}
%
		In this supplement, we provide a table of contents and the rest of the technical proofs.
		\tableofcontents
%

\section{Proofs in Section \ref{sec:adaptiveRCI}} \label{sec:app-adaptiveCI-additional-result}

\subsection{Proof of Theorem \ref{thm:upper-bound-general-Gaussian}}
First, we note that the coverage guarantee directly follows from Proposition \ref{th:estimation-guarantee}. Now, we move to the diameter guarantee. It is easy to check that $t_{\epsilon} \in \left[1.6, \Phi^{-1}\left(1 - 300 \sqrt{ \frac{d}{n} } - \sqrt{ \frac{\log(1/\alpha)}{n} } \right)\right]$ for all $\epsilon \in [0,0.05]$ when $n/(\log(1/\alpha)+d)$ is sufficiently large. Since Proposition \ref{th:estimation-guarantee} only uses concentration from the uniform DKW inequality for multivariate distributions (see Lemma \ref{lm:multi-DKW}), by the calculation in \eqref{ineq:length-guarantee}, we have $|\widehat{\CS}| \leq 4/t_{\epsilon}$ with probability at least $1-\alpha$. Finally, we have $1/t_{\epsilon} \lesssim  \frac{1}{\sqrt{\log(1/\epsilon)}} + \frac{1}{\sqrt{ \log(n/\log(1/\alpha)) }} + \frac{1}{\sqrt{\log( n/d )}}$ for all $\epsilon \in [0, 0.05]$ by Lemma \ref{lm:t-epsilon-bound} in Appendix \ref{sec:proof-additional-lemmas}.

\subsection{Proof of Proposition \ref{th:estimation-guarantee}}
 By the uniform DKW inequality for multivariate distributions (see Lemma \ref{lm:multi-DKW} in Appendix \ref{sec:proof-additional-lemmas}), we have with probability at least $1 - \alpha$, the following event holds:
	\begin{equation*}
		(A) = \left \{\textnormal{for all } u \in \cS^{d-1} \text{ and } \textnormal{for all } x \in \bbR, |F_n(x; u) - F_{\epsilon, \theta, Q}(x;u)| \leq 300 \sqrt{ \frac{d}{n} } + \sqrt{ \frac{\log(1/\alpha)}{n} } \right \},
	\end{equation*}where $F_{\epsilon, \theta, Q }(\cdot; u)$ denotes the CDF of $X^\top u$ under $X \sim P_{\epsilon, \theta, Q}$. 	

Now we prove \eqref{ineq:upper-side} based on ($A$). For simplicity, we denote the range of $t$ we consider as $[t_{\min}, t_{\max}]$. Since $n \geq C(d + \log(1/\alpha))$ for a sufficiently large $C>0$, we have $\Phi^{-1}\left(1 - \epsilon_{\max} - 300 \sqrt{ \frac{d}{n} } - \sqrt{ \frac{\log(1/\alpha)}{n} } \right) \geq 1.6$. We start by deducing a sufficient condition on the guarantee of \eqref{ineq:upper-side} given ($A$). For any $t \in [t_{\min}, t_{\max}]$ and any $u \in \cS^{d-1}$,
	\begin{equation}\label{eq:left-estimator-reduction}
		\begin{split}
			& \quad \quad \htheta^u(t)  - \theta^\top u \leq \frac{2}{t} \\
			& \Longleftrightarrow \theta^\top u - t + 2/t \geq F_n^{-1}( 2(1 - \Phi(t) ); u ) \\
			& \Longleftrightarrow F_n\left(\theta^\top u - t + 2/t ; u \right) \geq 2 (1 - \Phi(t) ) \\
			& \overset{(A)}\Longleftarrow  F_{\epsilon, \theta , Q} (\theta^\top u - t + 2/t; u ) - \left( 300 \sqrt{ \frac{d}{n} } + \sqrt{ \frac{\log(1/\alpha)}{n} } \right) \geq 2 (1 - \Phi(t) ) \\
			& \Longleftarrow (1 - \epsilon_{\max} )\bbP_{N(0,1)}(X \geq t -2/t) - \left( 300 \sqrt{ \frac{d}{n} } + \sqrt{ \frac{\log(1/\alpha)}{n} } \right)\geq 2 (1 - \Phi(t) ).
		\end{split}
	\end{equation} Based on \eqref{eq:left-estimator-reduction}, to show $\htheta^u(t)  - \theta^\top u \leq 2/t$ holds simultaneously for all $t \in [t_{\min}, t_{\max}]$, all $u \in \cS^{d-1}$ and uniformly over all $\epsilon\in[0,0.05]$, all $\theta\in\mathbb{R}^d$, and all $Q$ with probability at least $1 - \alpha$, it is enough to show 
	\begin{equation} \label{ineq:second-step-reduction}
		\begin{split}
& 2 \bbP_{N(0,1)}(X \geq t) < (1 - \epsilon_{\max} )\bbP_{N(0,1)}(X \geq t -2/t) - \left( 300 \sqrt{ \frac{d}{n} } + \sqrt{ \frac{\log(1/\alpha)}{n} } \right), \forall t \in [t_{\min}, t_{\max}] \\
			\Longleftarrow & \left\{ \begin{array}{l}
				\bbP_{N(0,1)}(X \geq t) \geq 300 \sqrt{ \frac{d}{n} } + \sqrt{ \frac{\log(1/\alpha)}{n} } , \forall t \in [t_{\min}, t_{\max}] \\
				(1 - \epsilon_{\max} )\bbP_{N(0,1)}(X \geq t -2/t) > 3 \bbP_{N(0,1)}(X \geq t),\forall t \in [t_{\min}, t_{\max}]
			\end{array} \right. \\
			\overset{ (a) }\Longleftarrow & \left\{ \begin{array}{l}
				t_{\max} \leq \Phi^{-1} \left( 1- \left( 300 \sqrt{ \frac{d}{n} } + \sqrt{ \frac{\log(1/\alpha)}{n} } \right) \right), \\
				(1 - \epsilon_{\max} )\phi(x - 2/t) > 3 \phi(x),\forall x \geq t, \forall t \in [t_{\min}, t_{\max}]
			\end{array} \right. \\
			\Longleftrightarrow & \left\{ \begin{array}{l}
				t_{\max} \leq \Phi^{-1} \left( 1- \left( 300 \sqrt{ \frac{d}{n} } + \sqrt{ \frac{\log(1/\alpha)}{n} } \right) \right), \\
				\exp( 2 - 2/t^2_{\min} ) > \frac{3}{1- \epsilon_{\max} },
			\end{array} \right.
		\end{split}
	\end{equation} where in (a), $\phi(\cdot)$ denotes the density of $N(0,1)$ and it is because if $(1 - \epsilon_{\max} )\phi(x - 2/t) > 3 \phi(x),\forall x \geq t$, then
	\begin{equation} \label{ineq:density-ratio-to-tail-prob}
	\begin{split}
		(1 - \epsilon_{\max} )\bbP_{N(0,1)}(X \geq t -2/t) & = \int_{x \geq t} (1 - \epsilon_{\max} )\phi(x - 2/t) dx  \\
		& >  \int_{x \geq t} 3 \phi(x) dx = 3 \bbP_{N(0,1)}(X \geq t).
	\end{split}
	\end{equation}

	Plug in $\epsilon_{\max} = 0.05$. Then $t_{\min} = 1.6$ and $t_{\max} = \Phi^{-1} \left( 1- \left( 300 \sqrt{ \frac{d}{n} } + \sqrt{ \frac{\log(1/\alpha)}{n} } \right) \right)$ satisfy the above conditions. This finishes the proof of \eqref{ineq:upper-side}.

	Next, we move to the proof of \eqref{eq:the-other-side} given $(A)$. Given $(A)$ happens, for any $t$ and $u \in \cS^{d-1}$, we have
	\begin{equation} \label{ineq:left-estimator-gua2}
		\begin{split}
			0 \leq \htheta^u(t)  - \theta^\top u&\Longleftrightarrow F_n^{-1}( 2(1 - \Phi(t) ); u )  + t - \theta^\top u\geq 0 \\
			& \Longleftarrow 2(1 - \Phi(t) ) > F_n(\theta^\top u-t; u ) \\
			& \overset{(A)}\Longleftarrow 2(1 - \Phi(t) ) > F_{ \epsilon, \theta, Q }(\theta^\top u-t; u ) +  300 \sqrt{ \frac{d}{n} } + \sqrt{ \frac{\log(1/\alpha)}{n} } \\
			& \Longleftarrow 2(1 - \Phi(t) ) \geq  \bbP_{N(0,1)}( X \leq -t ) + \epsilon +   300 \sqrt{ \frac{d}{n} } + \sqrt{ \frac{\log(1/\alpha)}{n} } \\
			& \Longleftrightarrow 2(1 - \Phi(t) ) \geq   1- \Phi(t) + \epsilon +   300 \sqrt{ \frac{d}{n} } + \sqrt{ \frac{\log(1/\alpha)}{n} }\\
			& \Longleftrightarrow 1 - \Phi(t) \geq \epsilon +  300 \sqrt{ \frac{d}{n} } + \sqrt{ \frac{\log(1/\alpha)}{n} }.
		\end{split}
	\end{equation}
	Notice that if we plug in $ t = t_{\epsilon} = \Phi^{-1}\left( 1- \epsilon-  300 \sqrt{ \frac{d}{n} } - \sqrt{ \frac{\log(1/\alpha)}{n} } \right) $, then the sufficient condition on the right-hand side of \eqref{ineq:left-estimator-gua2} will be satisfied. This finishes the proof of this proposition.


\subsection{Proof of Theorem \ref{th:lower-bound-Gaussian} }

We first show $r_{\alpha}(\epsilon, 0.05) \geq c  \frac{1}{\sqrt{\log(1/\epsilon)}}$. As we have discussed in Section \ref{sec:optimality-of-ARCI}, the diameter of an adaptive confidence set is lower bounded by the length of an adaptive confidence interval $\wh{\CI}$. Specifically, given $X_1, \ldots, X_n \overset{iid}\sim (1 - \epsilon)N(\theta_1, 1) + \epsilon Q_1$, generate $Y_i \overset{iid}\sim N(0, I_{d-1})$. Then $X_i' = [X_i, Y_i]^\top  \in \bbR^d \overset{iid}\sim (1- \epsilon) N(\theta, I_d) + \epsilon Q$ with $\theta = [\theta_1, 0_{d-1}]^{\top }$. Based on $\{X_i'\}$s, we can construct an adaptive confidence set for $\theta$, which further yields an adaptive confidence interval for $\theta_1$ as discussed in \eqref{eq:CS-to-CI}. Then by the length lower bound established in Theorem 1 of \cite{luo2024adaptive}, we have $r_{\alpha}(\epsilon, 0.05) \geq c  \frac{1}{\sqrt{\log(1/\epsilon)}}$. 

Next, we prove $r_{\alpha}(\epsilon, 0.05) > c\left( \frac{1}{\sqrt{ \log(n/\log(1/\alpha)) }} + \frac{1}{\sqrt{\log(n/d)}} \right)$. By the monotonicity of $r_{\alpha}(\epsilon, 0.05)$ in $\epsilon$, it is enough to show $r_{\alpha}(0, 0.05) > c\left( \frac{1}{\sqrt{ \log(n/\log(1/\alpha)) }} + \frac{1}{\sqrt{\log(n/d)}} \right)$. We divide the proof into two cases.
\begin{itemize}
	\item Case 1: $\alpha \in [0.05,1/4)$. In this case, $\frac{1}{\sqrt{ \log(n/\log(1/\alpha)) }} + \frac{1}{\sqrt{\log(n/d)}} \lesssim \frac{1}{\sqrt{\log n}}$. Since the confidence set implies an adaptive confidence interval $\wh{\CI}$ for a one-dimensional mean estimation problem, following the same analysis as above, we have $r_{\alpha}(\epsilon, 0.05) \geq c  \frac{1}{\sqrt{\log n}}$ by Theorem 1 of \cite{luo2024adaptive}. 
	\item Case 2: $\alpha \in (0,0.05)$.   We prove the claim by a contradiction argument. Take $c$ to be the small constant in Proposition \ref{prop:test-lower-bound-gaussian}. Suppose $r_{\alpha}(0, 0.05) > c\left( \frac{1}{\sqrt{ \log(n/\log(1/\alpha)) }} + \frac{1}{\sqrt{\log(n/d)}} \right)$ does not hold, i.e., there exists a confidence set $\widehat{\CS}$ such that the following conditions hold
\begin{equation} \label{ineq:contradiction-condition}
	\inf_{ \epsilon \in [0, 0.05], \theta, Q} P_{\epsilon, \theta, Q}\left( \theta \in\widehat{\CS}\right) \geq 1-\alpha \quad \textnormal{ and } \quad \sup_{\theta} P_{\theta} \left( |\widehat{\CS}| \geq r \right) \leq \alpha,
\end{equation} with $r = c\left( \frac{1}{\sqrt{ \log(n/\log(1/\alpha)) }} + \frac{1}{\sqrt{\log(n/d)}} \right)$. Then by Lemma \ref{prop:CI-imply-test-new}, we can design a test $\phi$ using this $\widehat{\CS}$ with the following guarantee
\begin{equation} \label{ineq:typeI+II}
\begin{split}
	 P_{0} \left(\phi = 1  \right) +\sup_{Q, u\in\cS^{d-1}} P_{\epsilon_{\max}, ru, Q } \left(\phi = 0  \right)  \leq  3 \alpha.
\end{split}
\end{equation}

On the other hand, by Proposition \ref{prop:test-lower-bound-gaussian}, we have the following lower bound for solving the testing problem \eqref{eq:test-def} when $\alpha \in (0,0.05)$, 
	\begin{equation} \label{ineq:neyman-person}
		\begin{split}
		\inf_{\phi} \left\{ P_{0 } \left( \phi = 1 \right) + \sup_{Q,  u\in\cS^{d-1} } P_{\epsilon_{\max},  ru , Q } \left( \phi = 0  \right) \right\} \geq 4\alpha.
		\end{split}
	\end{equation}
	This contradicts \eqref{ineq:typeI+II}. Thus, the assumption does not hold, and the proof is complete.

\end{itemize}

\subsection{Proof of Lemma \ref{prop:CI-imply-test-new}}
 We start with the Type-1 error guarantee for $\phi$.
 \begin{equation*}  
		\begin{split}
			& P_{ 0 } \left( \phi = 1  \right) \overset{(a)}= P_{ 0 } \left( \sup_{u\in\cS^{d-1} }  \indi \left( ru  \in  \widehat{\CS}\right) = 1   \right) \\
			& =P_{ 0 } \left( \exists u\in\cS^{d-1} ,ru  \in  \widehat{\CS}  \right)  \\
			& \leq P_{0} \left(\exists u\in\cS^{d-1} ,ru  \in  \widehat{\CS}, 0 \in \widehat{\CS}\right) + P_{0} \left( 0 \notin  \widehat{\CS}\right) \\
			& \overset{(b)}\leq P_{0} \left( |\widehat{\CS}| \geq r\right) + P_{0} \left( 0 \notin  \widehat{\CS}\right)  \\
			& \leq \alpha + P_{0} \left( |\widehat{\CS}| \geq r\right) \\
			& \leq 2 \alpha,
		\end{split}
	\end{equation*}
where (a) is by the definition of $\phi$ and in (b), we use the definition of a diameter of a set $\widehat{\CS}$. The penultimate inequality is by the coverage property of $\widehat{\CS}$ and the last inequality is by the diameter guarantee of $\widehat{\CS}$ at $\epsilon = 0$. 

Next, we bound the Type-2 error of $\phi$.	
\begin{equation*}
	\begin{split}
		 & \sup_{Q, u\in\cS^{d-1} } P_{\epsilon_{\max}, ru , Q } \left( \phi = 0  \right)  \\
		 = & \sup_{Q, u\in\cS^{d-1} } P_{\epsilon_{\max}, ru, Q } \left( \sup_{u' \in \cS^{d-1}}  \indi \left( ru'  \in  \widehat{\CS}\right) = 0 \right)  \\
		 \leq & \sup_{Q, u\in\cS^{d-1} } P_{\epsilon_{\max}, ru, Q } \left( ru  \notin  \widehat{\CS}\right) \\
		 \leq &\alpha,
	\end{split}
\end{equation*}	where the equality is by the definition of $\phi$ and the last inequality is by the coverage guarantee of $\widehat{\CS}$.

\subsection{Proof of Proposition \ref{prop:test-lower-bound-gaussian}}

We can take $c$ small enough so that $r \leq 1$. Recall that the testing problem we are interested in is
	\begin{equation} \label{eq:simplied-testing}
	\begin{split} &\begin{array}{l}
		H_{0}: P = N(0,I_d) \quad  \textnormal { vs } \quad H_{1}: P \in \left\{ P_{\epsilon_{\max}, ru, Q}: Q, u \in \cS^{d-1} \right\},
	\end{array}
	\end{split}
\end{equation} where $\epsilon_{\max} = 0.05$. As we mentioned in Section \ref{sec:optimality-of-ARCI}, we will construct an NGCA to demonstrate the hardness of the testing problem. Specifically, for any $u \in \cS^{d-1}$, we construct $Q_{u}$ as follows
	\begin{equation*}
		\begin{split}
			Q_{u}(x) = \frac{\phi_d(x) - (1 - \epsilon_{\max}) \phi_d (x - ru)\cdot \frac{\indi( | x^\top u - r | \leq t )}{\bbP(|N(0,1)| \leq t)} }{\epsilon_{\max} }, \quad \forall x \in \bbR^d,
		\end{split}
	\end{equation*} where $\phi_d(x) = (2 \pi)^{-d/2} \exp\left( - \frac{1}{2} \|x\|_2^2  \right) $ denotes the probability density function of $N(0, I_d)$ and 
    $t \geq 4$ is some value to be chosen later. It is easy to check $\int_{x \in \bbR^d} Q_{u}(x) dx = 1$. To ensure $Q_{u}$ is a valid distribution, we also need
	\begin{equation} \label{eq:valid-construction-condition}
		\begin{split}
			& \frac{(1 - \epsilon_{\max} ) \phi_d(x - ru) \indi( | x^\top u - r | \leq t ) }{\phi_d(x)\bbP(|N(0,1)| \leq t)} \leq 1, \quad  \forall x \in \bbR^d \\
			\Longleftrightarrow & (1 - \epsilon_{\max} ) \exp(r x^\top u - r^2/2) \indi( | x^\top u - r | \leq t ) \leq \bbP(|N(0,1)| \leq t), \quad  \forall x \in \bbR^d \\
			\Longleftarrow &  (1 - \epsilon_{\max} ) \exp(r (r + t) - r^2/2) \leq \bbP(|N(0,1)| \leq t)\\
			\overset{(a)}\Longleftarrow & (1 - \epsilon_{\max} ) \exp(2rt) \leq 1 - \bbP(|N(0,1)| \geq t)\\
			 \overset{(b)}\Longleftarrow  & (1 - \epsilon_{\max} ) \exp(2rt) \leq 1 - \epsilon_{\max}/2\\
			\Longleftarrow & r \leq \frac{\log ((1 - \epsilon_{\max}/2)/(1 - \epsilon_{\max} )) }{2t},
		\end{split}
	\end{equation} where (a) is because $2t \geq r$ and (b) is because 
	$P(|N(0,1)| \geq t) \leq \frac{2}{\sqrt{2\pi}} \cdot \frac{1}{t} \exp(-t^2/2) \leq \epsilon_{\max}/2$ as $t \geq 4$.
	
	By construction, we know
		\begin{equation*}
		\begin{split}
			P_{ \epsilon_{\max}, ru, Q_u }(x) 
			& = \phi_d(x) + (1 - \epsilon_{\max}) \phi_d(x - ru) \left(1  - \frac{ \indi( |x^\top u - r| \leq t )}{\bbP(|N(0,1)| \leq t)} \right)\\
			& = A(x^\top u)\exp( - \| x - (u^\top x) u\|_2^2/2 )/(2\pi)^{(d-1)/2} \;,
		\end{split}
	\end{equation*} 
    where $A$ is a distribution on $\bbR$ with the following density function
	\begin{equation*}
		\begin{split}
			A(x)  = \phi(x) + (1 - \epsilon_{\max} ) \phi(x - r) \left(1 -  \frac{\indi( |x - r| \leq t )}{\bbP(|N(0,1)| \leq t)} \right).
		\end{split}
	\end{equation*} Therefore, $P_{ \epsilon_{\max}, ru, Q_u }$ is a high-dimensional hidden direction distribution $P_u^A$ as defined in Definition \ref{def:high-dim-distribution}. Thus, the testing problem \eqref{eq:simplied-testing} is at least as hard as the NGCA testing problem \eqref{eq:NGCA-final}. It is easy to check that $
		\int x A(x) dx = 0$. Then Lemma \ref{lm:NGCA-stat} suggests that as long as	
		\begin{equation} \label{ineq:chi-square-desired-bound}
		 \chi^2(A, N(0,1)) \leq  \frac{d + \log(1/\alpha)}{8n},
	\end{equation} we can obtain the desired result by observing that minimax risk is lower bounded by the Bayesian risk.

	For the rest of the proof, we will first bound $\chi^2(A, N(0,1))$, and then show that when $r \leq c \left( \frac{1}{\sqrt{ \log(n/\log(1/\alpha)) }} +\frac{1}{\sqrt{\log (n/d)}}\right)$ holds, then we can choose $t$ properly so that the constructed $Q_{u}$ is a valid distribution, i.e., \eqref{eq:valid-construction-condition} holds, and at the same time  \eqref{ineq:chi-square-desired-bound} holds.
	
			Since $r\in [0,1]$, we have 
			\begin{equation} \label{ineq:chi-square-bound}
		\begin{split}
			 &\chi^2(A, N(0,1)) \\
			&= \int \frac{\left( \phi(x) + (1 - \epsilon_{\max} ) \phi(x - r) \left(1 -  \frac{  \indi( |x - r| \leq t )}{\bbP(|N(0,1)| \leq t)} \right) - \phi(x) \right)^2}{\phi(x)} dx \\
			&= \int_{|x - r| \leq t} \frac{(1 - \epsilon_{\max})^2 \phi^2(x - r)  }{\phi(x)} \left (\frac{1}{\bbP(|N(0,1)| \leq t)} - 1\right)^2 dx + \int_{|x - r| > t} \frac{(1 - \epsilon_{\max})^2 \phi^2(x - r)}{\phi(x)} d x\\
			& = \left (\frac{1}{\bbP(|N(0,1)| \leq t)} - 1\right)^2\int_{|x - r| \leq t}  \frac{(1 - \epsilon_{\max} )^2}{\sqrt{2\pi}}  \exp(r^2) \exp(-(x-2r)^2/2) dx  \\
			& \quad + \int_{|x - r| > t} \frac{(1 - \epsilon_{\max} )^2}{\sqrt{2\pi}} \exp(r^2) \exp(-(x-2r)^2/2) dx \\
			& \overset{(a)}\leq \left (\frac{1}{\bbP(|N(0,1)| \leq t)} - 1\right)^2  (1 - \epsilon_{\max} )^2  \exp(r^2)  +  \frac{(1 - \epsilon_{\max} )^2}{\sqrt{2\pi}}  \exp(r^2) \int_{|x - r| > t} \exp(-(x-2r)^2/2) dx \\
			& \overset{(b)}\leq \left (\frac{1}{\bbP(|N(0,1)| \leq t)} - 1\right)^2  \frac{(1 - \epsilon_{\max} )^2}{\sqrt{2\pi}}  \exp(r^2)  +  2\frac{(1 - \epsilon_{\max} )^2}{\sqrt{2\pi}}  \exp(r^2) \exp(-t^2/8) \\
			& \overset{(c)}\leq (1 - \epsilon_{\max} )^2  \exp(r^2) \exp(-t^2) + 2\frac{(1 - \epsilon_{\max} )^2}{\sqrt{2\pi}}  \exp(r^2) \exp(-t^2/8)\\
			& \leq (1 + 2/\sqrt{2\pi}) (1 - \epsilon_{\max} )^2  \exp(r^2) \exp(-t^2/8) \leq 2 (1 - \epsilon_{\max} )^2  \exp(r^2) \exp(-t^2/8)
		\end{split}
	\end{equation} where $(a)$ is because $\int_{|x-r| \leq t} \frac{1}{\sqrt{2\pi}} \exp(-(x-2r)^2/2) dx \leq 1$; (b) is because by a change of variable, we have
	\begin{equation*}
		\begin{split}
			\int_{|x - r| > t} \exp(-(x-2r)^2/2) dx & = \int_{x > t - r} \exp(-x^2/2) dx + \int_{x < -t -r} \exp(-x^2/2) dx \\
			& \leq 2\int_{x > t - r} \exp(-x^2/2) dx \overset{(i)}\leq 2\int_{x > t/2} \exp(-x^2/2) dx  \\
			& \leq \frac{4}{t} \exp(-t^2/8) \overset{(i)}\leq 2\exp(-t^2/8), 
		\end{split}
	\end{equation*} where (i) is because $t \geq 2 \geq 2r$; finally (c) is because
	\begin{equation*}
		\begin{split}
			\frac{1}{\bbP(|N(0,1)| \leq t)} - 1=\frac{P(|N(0,1)| \geq t)}{1 - P(|N(0,1)| \geq t)} \leq 2 P(|N(0,1)| \geq t) \leq \frac{4}{\sqrt{2\pi}} \cdot \frac{1}{t} \exp(-t^2/2) \leq \exp(-t^2/2) 
		\end{split}
	\end{equation*}
as $t \geq 2$ and $P(|N(0,1)| \geq t) \leq \frac{2}{\sqrt{2\pi}} \cdot \frac{1}{t} \exp(-t^2/2)$.

In view of \eqref{ineq:chi-square-bound} and \eqref{ineq:chi-square-desired-bound}, the condition we have on $t$ reduces to 
\begin{equation*}
	\begin{split}
		2 (1 - \epsilon_{\max} )^2  \exp(r^2) \exp(-t^2/8) \leq \frac{d + \log(1/\alpha)}{8n} \Longleftarrow t= \sqrt{8 \log \left( \frac{16e (1 -\epsilon_{\max} )^2 n }{d+ \log(1/\alpha)} \right) },
	\end{split}
\end{equation*} where the implication is because $r \leq 1$. Since $n \gtrsim d + \log(1/\alpha)$, the choice of $t$ above is bigger than $4$, which is required in the previous analysis.

	Then by \eqref{eq:valid-construction-condition}, for the above choice of $t$, the condition we need to guarantee $Q_u$ is a valid distribution is 
	\begin{equation} \label{ineq:r-condition2}
		r \leq \frac{\log ((1-\epsilon_{\max}/2)/(1 - \epsilon_{\max} )) }{2 \sqrt{8 \log \left( \frac{16e (1 -\epsilon_{\max} )^2 n }{d+ \log(1/\alpha)} \right) }}.
	\end{equation} Note that the above condition can be satisfied since $r \leq c \left( \frac{1}{\sqrt{ \log(n/\log(1/\alpha)) }} +\frac{1}{\sqrt{\log (n/d)}}\right)$ for some small enough $c > 0$. This finishes the proof of this proposition.

\subsection{Proof of Lemma \ref{lm:NGCA-stat}}
First, by the Neyman-Pearson lemma
\begin{equation*}
		\begin{split}
			&\inf_{\phi} \left\{ P_0(\phi = 1) + P_{\mu}(\phi = 0) \right\} 
		\geq  1 - \TV(P_{\mu}(\{X_i\}_{i=1}^n), P_{0}(\{X_i\}_{i=1}^n)),
		\end{split}
	\end{equation*} where $\TV(P,Q):= \sup_B|P(B)-Q(B)|$ denotes the total variation distance between two probability distributions $P$ and $Q$. Therefore, to show the result, it is enough to show that when $ \chi^2(A, N(0,1)) \leq \frac{d + \log(1/\alpha)}{8n} $, $ \TV(P_{\mu}(\{X_i\}_{i=1}^n), P_{0}(\{X_i\}_{i=1}^n)) \leq 1 - 4 \alpha$.

Next, we divide the proof into two cases: (i)  $ \chi^2(A, N(0,1)) \leq \frac{\log(1/\alpha)}{4n} $ and (ii) $ \chi^2(A, N(0,1)) \leq \frac{d}{4n} $. We will show that in both cases, we have $ \TV(P_{\mu}(\{X_i\}_{i=1}^n), P_{0}(\{X_i\}_{i=1}^n)) \leq 1 - 4 \alpha$.
	
	\vskip.3cm
	{\bf \noindent Case 1: $ \chi^2(A, N(0,1)) \leq \frac{\log(1/\alpha)}{4n} $.} In this case,
	\begin{equation} \label{ineq:second-moment-calculation1}
		\begin{split}
			& \TV(P_{\mu}(\{X_i\}_{i=1}^n), P_{0}(\{X_i\}_{i=1}^n))  =  \TV\left ( \int_{u \in \cS^{d-1} } P^A_u(\{X_i\}_{i=1}^n) d\mu(u), P_{0}(\{X_i\}_{i=1}^n) \right)  \\
			&\overset{(a)} \leq\int_{u \in \cS^{d-1} } \TV\left ( P^A_u(\{X_i\}_{i=1}^n) , P_{0}(\{X_i\}_{i=1}^n) \right) d \mu(u)\\
			& \overset{(b)}\leq 1 - \frac{1}{2} \int_{u \in \cS^{d-1} } \exp(-\KL(P^A_u(\{X_i\}_{i=1}^n), P_0(\{X_i\}_{i=1}^n)) ) d\mu(u) \\
			&= 1 -\frac{1}{2}\int_{u \in \cS^{d-1} } \exp(-n \KL(P^A_u(X), P_0(X))) d\mu(u) \\
			& \leq 1 - \frac{1}{2}\int_{u \in \cS^{d-1} } \exp(-n \chi^2(P^A_u(X), P_0(X))) d\mu(u) = 1 - \frac{1}{2} \exp(-n \chi^2(A, N(0,1) ))  \\
			& \overset{(c)}\leq 1 - \frac{\alpha^{1/4}}{2} \overset{(d)}\leq 1 - 4 \alpha \;,
		\end{split}
	\end{equation} where (a) is by the convexity of the $\TV(P,Q)$ in $(P,Q)$ pair; in (b), $\KL(P,Q) := \int \log(dP/dQ) dP$ denotes the Kullback-Leibler divergence between two probability distributions $P$ and $Q$ and we use Lemma 2.6 in \cite{tsybakov2009introduction}; (c) is because $\chi^2(A, N(0,1)) \leq \frac{\log(1/\alpha)}{4n} $; (d) is because  $\alpha \in (0,0.05)$.
		
		\vskip.3cm
	{\bf \noindent Case 2: $ \chi^2(A, N(0,1)) \leq \frac{d}{4n} $.} The proof for this case largely follows from the proof of Proposition 7.1 of \cite{diakonikolas2017statistical} where $\alpha$ is treated as a fixed constant. Here we present a proof with an explicit dependence on $\alpha$. First, we introduce the following facts, which are from \cite{diakonikolas2017statistical} Fact 7 and Lemma 7.3:
	\begin{itemize}
	\item (Fact 1) For $x > -1$, $y > -1$ we have that $\int_{0}^{\pi/2} \sin^{x}(\theta)\cos^{y}(\theta)\,d\theta
= \frac{1}{2} B\!\left(\frac{x+1}{2}, \frac{y+1}{2}\right)$, where $B(x,y)$ is the Beta function.
		\item (Fact 2) For all $x,y \in \mathbb{R}$, we have that $B(x,y+1) = B(x,y)\cdot \frac{y}{x+y}$.		\item (Fact 3) If we choose $u$ and $u'$ uniformly at random from $\cS^{d-1}$, the angle $\theta$
between them is distributed with the probability density function $\frac{\sin^{d-2}(\theta)}{B\!\left((d-1)/2,\,1/2\right)}$.
	\end{itemize}

	 Then, we have
	\begin{equation} \label{ineq:second-moment-calculation}
		\begin{split}
			&\bbE_{P_0} [(P_{\mu}(\{X_i\}_{i=1}^n)/P_0(\{X_i\}_{i=1}^n))^2 ] \\
			&= \bbE_{u, u' \overset{iid}\sim \textnormal{Unif}(\cS^{d-1}) } \left[ \int \frac{P_{ \epsilon_{\max}, ru, Q_{u} }(x)  P_{ \epsilon_{\max}, ru', Q_{u'} }(x) }{\phi_d(x)} dx \right]^n \\
			& \overset{(a)}\leq  \bbE_{u, u' \overset{iid}\sim \textnormal{Unif}(\cS^{d-1}) } \left[1 + (u^\top u')^2 \chi^2(A, N(0,1)) \right]^n \\
&\overset{\text{Fact 3}}= 2 \int_{0}^{\pi/2}
\left(1 + \cos^{2}(\theta)\chi^{2}(A,N(0,1))\right)^{n}
\frac{\sin^{d-2}(\theta)}{B\!\left((d-1)/2,1/2\right)} \, d\theta \\
&= \sum_{i=0}^{n} \binom{n}{i} 2 \int_{0}^{\pi/2}
[\chi^{2}(A,N(0,1))]^{i}
\cos^{2i}(\theta)\sin^{d-2}(\theta)
\frac{d\theta}{B\!\left((d-1)/2,1/2\right)} \\
&\overset{\text{Fact 1} }= \sum_{i=0}^{n} \binom{n}{i}
[\chi^{2}(A,N(0,1))]^{i}
\frac{B\!\left((d-1)/2,\, i+1/2\right)}{B\!\left((d-1)/2,\,1/2\right)} \;,
		\end{split} 
	\end{equation} where (a) is by the fact $
		\int x A(x) dx = 0$ and Lemma \ref{lm:NGCA}.

We rewrite the above sum as $\sum_{i=0}^{n} b_i$, where
$
b_i
=
\binom{n}{i}
\chi^2(A, N(0,1))^i
\frac{B\bigl((d-1)/2,\, i+1/2\bigr)}{B\bigl((d-1)/2,\, 1/2\bigr)} .
$ Note that $b_0 = 1$ and 
\[
\frac{b_{i+1}}{b_i}
=
\chi^2(A, N(0,1))
\cdot
\frac{n-i}{i+d/2}
\cdot
\frac{i+1/2}{i+1} \;,
\] by Fact 2. Therefore, $b_1 = \frac{n}{d}\chi^2(A, N(0,1)) $ and for all $i \geq 1$, we have
\[
\frac{b_{i+1}}{b_i}
\le
\chi^2(A, N(0,1)) \cdot \frac{n}{d/2}
=
\frac{2n\,\chi^2(A, N(0,1))}{d}.
\] So if $\chi^2(A, N(0,1)) \leq \frac{d}{4n}  $, we have 
\begin{equation} \label{ineq:LR-square-bound}
\begin{split}
	\bbE_{P_0} [(P_{\mu}(\{X_i\}_{i=1}^n)/P_0(\{X_i\}_{i=1}^n))^2 ] & \leq \sum_{i=0}^n b_i \leq 1 + \frac{n\,\chi^2(A, N(0,1))}{d} \left(1 + 1/2 + 1/4 + \cdots \right) \\
	& \leq 1 + \frac{2n\,\chi^2(A, N(0,1))}{d}.
\end{split}
\end{equation} Then 
	 \begin{equation*} 
 	\begin{split}
 		 \TV(P_{\mu}(\{X_i\}_{i=1}^n), P_{0}(\{X_i\}_{i=1}^n)) &= \frac{1}{2} \bbE_{\bbP_0}( | P_\mu(\{X_i\}_{i=1}^n)/P_0(\{X_i\}_{i=1}^n) - 1   | ) \\
 		 &\leq\frac{1}{2}  \left( \bbE_{\bbP_0} \left(  ( P_\mu(\{X_i\}_{i=1}^n)/P_0(\{X_i\}_{i=1}^n) -1  )^2 \right) \right)^{1/2} \\
 		 &= \frac{1}{2} \sqrt{ \bbE_{P_0} [(P_{\mu}(\{X_i\}_{i=1}^n)/P_0(\{X_i\}_{i=1}^n))^2 ]-1 } \\
 		 & \overset{ \eqref{ineq:LR-square-bound} }\leq \frac{1}{2} \sqrt{\frac{2n\,\chi^2(A, N(0,1))}{d}} \overset{(a)} \leq \frac{1}{2 \sqrt{2} } \overset{(b)} \leq 1 - 4 \alpha,
 	\end{split}
 \end{equation*}	 where (a) is because $\chi^2(A, N(0,1)) \leq \frac{d}{4n}  $ and (b) is because $\alpha \in (0,0.05)$.

\section{Proofs in Section \ref{sec:GSQlow}}

\subsection{Proof of Theorem \ref{th:SQ-lower-bound}}

We proceed by contradiction. 
Suppose $\wh{\CS}\in \mathrm{SQ}(k,\tau)$ is an SQ confidence set with diameter $r =\frac{c_1}{\sqrt{m}}$ at contamination level $\epsilon \in [0,0.05]$. Then by definition, it satisfies 
\begin{equation*}
	\begin{split}
		& \inf_{\theta} \inf_{ P \in \cP_{\epsilon_{\max},\theta} } \,  \,\inf_{\mathtt{STAT}_{P,\tau} \in \mathtt{Oracle}_{P,\tau} } \bbP\left( \theta \in \widehat{\CS}(\mathtt{STAT}_{P,\tau})\right) \geq 0.95,\\
		&  \inf_{\mathtt{STAT}_{P_0,\tau} \in \mathtt{Oracle}_{P_0,\tau} } \bbP\left(  | \wh{\CS}(\mathtt{STAT}_{P,\tau})| \leq \frac{c_1}{\sqrt{m}} \right) \geq 0.95,
	\end{split}
\end{equation*} where $P_0 = N(0,I_d)$.
Then by Lemma \ref{lm:SQ-CS-testing}, we know that there exists an SQ test $\phi \in \mathrm{SQ}(k,\tau)$ that solves the testing problem \eqref{eq:SQ-test-CI} with separation $r = \frac{c_1}{\sqrt{m}} $. Next, we show that this implies one of the conditions in \eqref{ineq:SQ-option} holds by Lemma \ref{lm:SQ-NGCA}.

As we mentioned in Section \ref{sec:GSQlow}, we will show that a special case of \eqref{eq:SQ-test-CI} is an NGCA testing problem \eqref{eq:NGCA-SQ}. In particular, 
by applying Lemma \ref{lm:population-existence-analysis} with $\delta = \epsilon_{\max}$, it follows that given any $r \in [0, \epsilon_{\max}^2]$ and any nonnegative integer $m' \leq \frac{\epsilon_{\max}^2}{256 r^2}$, we can construct a valid distribution $D$ supported on $\bbR^d$ such that 
\begin{equation*}
	D \geq (1 - \delta) N(ru, I_d) \quad \text{and}\quad D \equiv_{m'} N(0, I_d).
\end{equation*} We note that implicitly the construction in Lemma \ref{lm:population-existence-analysis} is the same one as in \eqref{eq:SQ-con}-\eqref{eq:A-NGCA-SQ}. The existence of such a $D$ means that there exists a $Q_u$ such that $P_{\delta, ru, Q_u} \equiv_{m'} N(0, I_d)$. We can take $m'$ to be $m$ as long as $c_1$ is sufficiently small. Then the conclusion follows by applying Lemma \ref{lm:SQ-NGCA}. This completes the proof.

\subsection{Proof of Lemma \ref{lm:SQ-CS-testing}}

	Suppose $\wh{\CS} \in \mathrm{SQ}(k,\tau)$ is an SQ confidence set with $k$ queries, tolerance $\tau$ and diameter $r$. Given any SQ oracle $\mathtt{STAT}_{P,\tau}$, let us construct the test $\phi(\mathtt{STAT}_{P,\tau}) = \sup_{u\in\cS^{d-1} } \indi(ru \in \wh{\CS}(\mathtt{STAT}_{P,\tau}))$ for solving the testing problem \eqref{eq:SQ-test-CI} given separation $r$. Notice that $\phi \in \mathrm{SQ}(k,\tau)$ is an SQ test with $k$ queries of tolerance $\tau$. This is because $\phi$'s interaction with the SQ oracle is only through the $k$ queries with tolerance $\tau$ used to construct $\wh{\CS}$. Then, following the same analysis as in Lemma \ref{prop:CI-imply-test-new}, we have that this SQ test $\phi$ satisfies
	\begin{equation*}
    \begin{split}
      & \sup_{\mathtt{STAT}_{P_0,\tau} \in \mathtt{Oracle}_{P_0,\tau} } \bbP\left(\phi(\mathtt{STAT}_{P_0,\tau})=1\right) \leq  0.1\\
       & \sup_{u\in\cS^{d-1} ,P \in \cP_{\epsilon_{\max},ru}}\, \sup_{\mathtt{STAT}_{P,\tau} \in \mathtt{Oracle}_{P,\tau} } \bbP\left(\phi(\mathtt{STAT}_{P,\tau})=0\right)\leq 0.05.  
    \end{split}
	\end{equation*} This finishes the proof.

\section{Proofs in Section \ref{sec:pop} } \label{sec:proof-pop}

\subsection{Proof of Lemma \ref{lm:hat-epsilon-guarantee} }

	We note that the result can be directly deduced from Theorem \ref{th:hat-epsilon-guarantee} by observing that $\wt{\epsilon}$ in \eqref{ineq:hat-epsilon-relax} is upper bounded by $\wt{\epsilon}$ in \eqref{ineq:hat-epsilon}. Here we provide a simple self-contained proof of this result. Notice that since $\|\wh{\theta}- \theta\|  \leq 1$, if $\wt{\epsilon} > 0.01$, then we automatically have $\|\wh{\theta}- \theta\|  \leq 10\sqrt{\frac{10}{\log(1/\wt{\epsilon})}} $. Therefore, for the rest of the proof, we focus on $\wt{\epsilon} \in [0,0.01]$. Denote $r = \|\wh{\theta}- \theta\|$. When $\wh{\theta} = \theta$, the statement clearly holds. Now let us focus on $\wh{\theta} \neq \theta$. We are going to show that 
	\begin{equation} \label{ineq:wttheta-statement}
		\text{For all } \delta \in [0,0.01],\quad  r > 10\sqrt{\frac{10}{\log(1/\delta)}} \Longrightarrow   P_{\epsilon, \theta, Q} \geq (1-\delta) N(\wh{\theta}, I_d) \text{ does not hold}.
	\end{equation} If we can show the above statement, then by the definition of $\wt{\epsilon}$, we have $ r \leq 10\sqrt{\frac{10}{\log(1/\wt{\epsilon})}}$, which is the desired result.
	
	The following is a structural result we will leverage, and it will also be useful in the proof of Theorem \ref{th:hat-epsilon-guarantee}. 
	\begin{Lemma} \label{lm:structural-result}
		Given any probability measures $P_1, P_2, Q$ supported on the common measurable space $(\Omega, \cF)$, and any function class $\cH$ consisting of functions defined on $\Omega$. For any $\delta \in [0,1/4],\epsilon \in [0,1/4]$, if there exists a probability distribution $D$ on $(\Omega, \cF)$ such that
		\begin{equation*}
			(i) \, \, (1 - \epsilon) P_1 + \epsilon Q \geq (1 - \delta) D \quad \text{ and }\quad (ii)\,\, \int_{\Omega} h(x) d D(x) = \int_{\Omega} h(x) d P_2 (x) \text{ for all } h \in \cH,
		\end{equation*} then for any nonnegative $h \in \cH$, 
		\begin{equation*}
			\int_{\Omega} h(x) d P_2 (x) \geq \frac{2}{3} \inf \left\{ \int_{\Omega} h(x) dD' (x): D' \text{ is a distribution such that } P_1 \geq \left( 1- \frac{4\delta}{3} \right)D' \right\}.
		\end{equation*}
	\end{Lemma} The proof of Lemma \ref{lm:structural-result} will be provided in the subsequent subsection. We next prove \eqref{ineq:wttheta-statement} by  contradiction. Given any $\delta \in [0,0.01]$, $r > 10\sqrt{\frac{10}{\log(1/\delta)}}$, and suppose $P_{\epsilon, \theta, Q} \geq (1-\delta) N(\wh{\theta}, I_d)$ holds. Define $\cH = \{ \indi(x^\top u \geq t) : u\in\cS^{d-1}, t \in \bbR  \}$. Then, if we set $P_1 = N(\theta - \wh{\theta}, I_d)$, $D = N(0, I_d)$ and $P_2 =  N(0, I_d)$, we know the assumption in Lemma \ref{lm:structural-result} for $D$ is satisfied over the class $\cH$. 
    As a result, for $h(x) = \indi(x^\top u \geq t)$ with  $u = (\theta-\wh{\theta})/\|\wh{\theta}- \theta\| $ and $t = \Phi^{-1}(1 - \delta)$, 
    Lemma \ref{lm:structural-result} implies
	\begin{equation} \label{ineq:signal-lower-bound}
		\bbP_{N(0, I_d)}(X^\top u \geq t ) \geq \frac{2}{3} \inf \left\{ \bbP_{D'}(X^\top u \geq t):  N(\theta - \wh{\theta}, I_d) \geq  \left( 1- \frac{4\delta}{3} \right) D'   \right\}.
	\end{equation} Notice that for any distribution $D'$ such that $N(\theta - \wh{\theta}, I_d) \geq  \left( 1- \frac{4\delta}{3} \right) D'$, there exists a distribution $Q'$ such that $N(\theta - \wh{\theta}, I_d)  =  \left( 1- \frac{4\delta}{3} \right) D' + \frac{4\delta}{3} Q' $ and this implies that
	\begin{equation} \label{ineq:density-dominance-implication}
	\begin{split}
		& \left( 1- \frac{4\delta}{3} \right) \bbP_{D'}(A) +   \frac{4\delta}{3} \bbP_{Q'}(A) = \bbP_{N(\theta - \wh{\theta}, I_d)}(A) \text{ for all measurable sets } A \\
		\Longrightarrow & \bbP_{D'}(A) \geq   \bbP_{N(\theta - \wh{\theta}, I_d)}(A) - \frac{4\delta}{3} \text{ for all measurable sets } A.
	\end{split}
	\end{equation} 
    Then, by plugging \eqref{ineq:density-dominance-implication} into \eqref{ineq:signal-lower-bound} with $A = \{x:x^\top u \geq t\}$, we have 
	\begin{equation} \label{ineq:implication}
	\begin{split}
		\frac{3}{2} \bbP_{N(0, I_d)}(X^\top u \geq t ) \geq  \bbP_{N(\theta - \wh{\theta}, I_d)}(X^\top u \geq t ) - \frac{4\delta}{3} & \overset{(a)}\Longleftrightarrow \frac{3}{2}  \bbP_{N(0,1)}(X \geq t) + \frac{4\delta}{3} \geq \bbP_{N(0,1)}(X \geq t -r)\\
		&  \overset{(b)}\Longrightarrow  3 \bbP_{N(0,1)}(X \geq t)\geq \bbP_{N(0,1)}(X \geq t -r),
	\end{split}
	\end{equation} where (a) is because $u = (\theta-\wh{\theta})/\|\wh{\theta}- \theta\| $ and we set $r = \|\wh{\theta}- \theta\| $; (b) is because $t = \Phi^{-1}(1 - \delta)$.

	We claim that \eqref{ineq:implication} implies that $r \leq 2/t$. To see this, if $r > 2/t$, then 
	\begin{equation} \label{ineq:contradiction}
		\begin{split}
			\bbP_{N(0,1)}(X \geq t -r) &>  \bbP_{N(0,1)}(X \geq t -2/t)  = \int_{x \geq t} \phi(x-2/t) dx \\
			& \overset{(a)}\geq \int_{x \geq t} 3\phi(x) dx = 3\bbP_{N(0,1)}(X \geq t),
		\end{split}
	\end{equation} where (a) is because $$\phi(x-2/t)/\phi(x) = \exp(2x/t - 2/t^2) \geq  \exp(2 - 2/t^2) \geq  3$$ for all $x \geq t$ given that $\delta \in [0,0.01]$ and $t = \Phi^{-1}(1 - \delta)$. Notice that \eqref{ineq:contradiction} contradicts \eqref{ineq:implication}. Therefore, \eqref{ineq:implication} implies 
	\begin{equation*}
		r \leq \frac{2}{t} = \frac{2}{\Phi^{-1}(1 -  \delta)} \overset{(a)}\leq \frac{4}{\sqrt{\log(1/ \delta)}} \;,
	\end{equation*} 
    where (a) is by Lemma \ref{lm:normal-quantile} and $\delta \in [0,0.01]$. 
    This contradicts our assumption $r > 10\sqrt{\frac{10}{\log(1/\delta)}}$ and finishes the proof.

\subsubsection{Proof of Lemma \ref{lm:structural-result} } Let $D$ be the distribution in the assumption. Then for any nonnegative $h \in \cH$,
\begin{equation*}
	\begin{split}
	&\int_{\Omega} h(x) d P_2 (x) =  \int_{\Omega} h(x) d D(x) \\
		\geq & \inf \left\{ \int h(x) dD'(x) : (1 - \epsilon) P_1 + \epsilon Q \geq (1 - \delta) D' \right\} \\
		= & \inf \left\{ \int h(x) dD'(x): \begin{array}{c}
		 D' =  (1 - a) D_1 + a D_2 \text{ for two distributions } D_1, D_2 \\ \text{ and some } a \in [0,1] \text{ such that }\\
			 (1-\epsilon) P_1 \geq (1 - \delta) (1 - a) D_1,
			 \epsilon Q  \geq (1-\delta) a D_2
		\end{array}\right\} \\
		\overset{h \geq 0}\geq &  \inf \left\{ \int (1-a)h(x) dD_1(x): \begin{array}{c}
		D_1, D_2 \text{ are two distributions}, a \in [0,1] \text{ such that }\\
			 (1-\epsilon) P_1 \geq (1 - \delta) (1 - a) D_1,
			 \epsilon Q  \geq (1-\delta) a D_2
		\end{array}\right\}\\
		\overset{(a)}\geq & \inf \left\{ \int \left(1-\frac{\epsilon}{1 - \delta}\right)h(x)dD_1(x): \begin{array}{c}
		D_1 \text{ is a distribution}\text{ such that }\\
			 (1-\epsilon) P_1 \geq (1 - \delta) \left(1 - \frac{\epsilon}{1 - \delta}\right) D_1
		\end{array}\right\}\\
		\overset{(b)}\geq & \inf \left\{ \int \frac{2}{3} h(x) dD'(x): \begin{array}{c}
		D' \text{ is a distribution}\text{ such that } P_1 \geq \left(1 - \frac{\delta}{1 -\epsilon} \right)  D'
		\end{array}\right\}\\
		\overset{(c)}= &\inf \left\{ \int \frac{2}{3} h(x) dD'(x): \begin{array}{c}
		D' \text{ is a distribution}\text{ such that } P_1 \geq \left(1 - \frac{4\delta}{3} \right)  D'
		\end{array}\right\}
	\end{split}
\end{equation*}  where (a) is because $ \epsilon Q  \geq (1-\delta) a D_2$ implies $a \leq \frac{\epsilon}{1 - \delta}$; (b) and (c) are due to $\delta, \epsilon \in [0,1/4]$.

\subsection{Proof of Theorem \ref{th:hat-epsilon-guarantee}}


For notational convenience, we denote $r=\| \wh{\theta} - \theta \|_2$.

\vskip.3cm
{\noindent \bf (Part 1: show $r \leq \sqrt{ \frac{1000}{\log(1/\wt{\epsilon})} } $)}. First, when $\wh{\theta} = \theta$, the statement clearly holds. Thus, we focus on $r \in (0,1]$. Without loss of generality, we can assume $\wt{\epsilon} \leq \exp(-1000)$, otherwise the inequality $\exp(- \frac{1000}{r^2}) \leq \wt{\epsilon}$ clearly holds. 

Recall that \(m_\delta\) is the largest even integer not exceeding \(\log(1/\delta)/9\). 
By simple manipulation, to prove the claim, it suffices to show 
	that $r \leq \frac{8}{\sqrt{ m_{\wt{\epsilon}} }}$ given $\wt{\epsilon} \leq \exp(-1000)$.
	 Let us define
 	\begin{equation} \label{eq:program-P2}
		\begin{split}
			\texttt{F}(\delta) = \left\{D: P_{\epsilon, \theta - \wh{\theta}, Q} \geq (1-\delta) D \text{ and }D\equiv_{m_\delta}N(0,I_d)\right\}.
		\end{split}
	\end{equation} It is easy to check that $\wt{\epsilon} = \inf\{\delta \in [0,1]: \texttt{F}(\delta) \neq \emptyset \}$.

	 Next we will show that:
\begin{equation} \label{eq:coverage-population}
	\text{For all } \delta \in [0, \exp(-1000)], \quad r > \frac{8}{\sqrt{ m_{\delta} }} \Longrightarrow\texttt{F}(\delta) = \emptyset .
\end{equation}
Therefore, $\texttt{F}({\wt{\epsilon}}) \neq \emptyset$ implies that $r\leq \frac{8}{\sqrt{ m_{\wt{\epsilon}} }}$. Now we proceed to prove \eqref{eq:coverage-population}. For notational convenience, for the rest of the proof in Part 1, 
we are going to write $m_{\delta}$ as $m$, while we should keep in mind that it implicitly depends on $\delta$.

We prove \eqref{eq:coverage-population} by contradiction. 
Given any $ \delta \in [0, \exp(-1000)]$ and $ r > \frac{8}{\sqrt{ m }}$, 
suppose that $\texttt{F}(\delta) \neq \emptyset$. 
Taking $\cH =  \cP_{d,m_{\delta}}$, $P_1 = N(\theta - \wh{\theta}, I_d)$ 
and $P_2 = N(0, I_d)$, we know that there exists a $D$ such that 
the assumption in Lemma \ref{lm:structural-result} is satisfied for $\cH$. 
Recalling that $m$ is even, by Lemma \ref{lm:structural-result}, 
for the special polynomial $ p(x) = (x^\top u)^{m}$ where $u = (\theta - \wh{\theta})/\|\wh{\theta}- \theta\| $, we have 
\begin{equation}\label{ineq:population-coverage-analysis}
\begin{split}
	&\bbE_{X \sim N(0, I_d)}[p(X)] \\
	 \geq & \frac{2}{3} \inf \left\{ \int  p(x) dD'(x): \begin{array}{c}
		D' \text{ is a distribution}\text{ such that } N(\theta - \wh{\theta}, I_d) \geq \left (1 - \frac{4}{3} \delta\right)  D'
		\end{array}\right\}\\
		\overset{(a)}\geq & \frac{2}{3} \inf \left\{ \int x^m dD_u(x): \begin{array}{c}
		D_u \text{ is a distribution in } \bbR \text{ such that }N(r, 1) \geq (1 - 2\delta)  D_u
		\end{array}\right\}\\
		\overset{(b)}= & \inf \left\{ \int \frac{2}{3(1 - 2\delta)} x^m g(x) dx: \begin{array}{c}
		\int g(x) dx = 1 - 2 \delta, 0\leq g(x) \leq  \phi(x-r), \forall x \in \bbR
		\end{array}\right\} \;,
\end{split}
\end{equation} where in (a), we plug in $ p(x) = (x^\top u)^{m}$ and $D_u$ can be viewed as the marginal distribution of $D'$ along the direction $u$; in (b), we rewrite the density of $D_u(x)$ as $g(x)/(1 - 2\delta)$.

Since $m$ is an even integer, $x^m$ increases as $|x|$ increases. Therefore, to minimize $\int \frac{2}{3(1 - 2\delta)} x^m g(x) dx$, we would like to put as much mass as possible close to the origin. In particular, the optimal $g$ function is 
\begin{equation*}
	g(x) = \phi(x-r) \indi(|x| \leq t^*),
\end{equation*} where $t^*$ satisfies $\bbP_{N(r,1)}(|X| \geq t^*) = 2 \delta$. If we denote the CDF of $Z=X^{m}$ where $X \sim N(r,1)$ as $F$, then it is easy to check $t^{*m} = F^{-1}(1 - 2 \delta)$. Combining this argument and \eqref{ineq:population-coverage-analysis}, we have 
\begin{equation} \label{ineq:contradiction-one-side}
	\begin{split}
		&\bbE_{X \sim N(0, I_d)}[p(X)] \geq  \frac{2}{3(1 - 2\delta)}\int_{x^m \leq  F^{-1}(1 - 2 \delta)}   x^m \phi(x-r) dx \\
		& = \frac{2}{3(1 - 2\delta)} \bbE_{Z}[Z \indi(Z \leq F^{-1}(1 - 2 \delta))] \overset{(a)}\geq \frac{1}{3\sqrt{2}e(1 - 2 \delta)} \exp(-r^2/2 + r \sqrt{m/e}) (m/e)^{m/2}\\
		& \overset{(b)}\geq \frac{1}{4\sqrt{2e}e} \exp(r \sqrt{m/e}) (m/e)^{m/2} \;,
	\end{split}
\end{equation} where (a) follows from Lemma \ref{lm:gaussian-tail}(iv) in Appendix \ref{sec:proof-additional-lemmas} given that $m \geq 20$, $\delta \leq \exp(-1000)$ and $\Phi^{-1}(1 - \delta) - r \geq \sqrt{2m} $ holds (this indeed holds, as $r \leq 1$ by our assumption and $\Phi^{-1}(1 - \delta) - r \geq \sqrt{\log(1/ (\sqrt{2\pi}\delta))} - r \geq \sqrt{2m}$, where the first inequality is by Lemma \ref{lm:normal-quantile} in Appendix \ref{sec:proof-additional-lemmas} and second inequality is because $m \leq \log(1/\delta)/9$ and $\delta \leq \exp(-1000)$); (b) is because $r \leq 1$ and $\delta \leq \exp(-1000)$.

On the other hand, we note that $\bbE_{X \sim N(0, I_d)}[p(X)] = (m-1)!! \leq \sqrt{2}e (m/e)^{m/2}$ by Lemma \ref{lm:gaussian-tail} (i). Therefore, when $r \geq \frac{8}{\sqrt{m}}$, for this special polynomial $ p(x) = (x^\top u)^{m}$, we have 
\begin{equation*}
	\begin{split}
		 & \bbE_{X \sim N(0, I_d)}[p(X)] \overset{ \eqref{ineq:contradiction-one-side}}\geq  \frac{1}{4\sqrt{2e}e} \exp(r \sqrt{m/e}) (m/e)^{m/2} > \sqrt{2}e (m/e)^{m/2} \geq \bbE_{X \sim N(0, I_d)}[p(X)] \;, 
	\end{split}
\end{equation*} 
a contradiction. Therefore, the assumption ``$\texttt{F}(\delta) \neq \emptyset$'' does not hold. This finishes the proof of Part 1.

\vskip.3cm
{\bf \noindent (Part 2: show $\wt{\epsilon} \leq C( \sqrt{r } + \epsilon)$)}. To show that, we are going to construct a $D$ such that \eqref{eq:D-condition2} is satisfied with $\delta = \sqrt{r} + \epsilon$. The key result we use is Lemma \ref{lm:population-existence-analysis} below, whose proof is provided in the subsequent subsections.
\begin{Lemma} \label{lm:population-existence-analysis}
	Suppose $u$ is a fixed unit vector in $\bbR^d$. Then given any $r\in [0, 1]$, any $\delta \in [ \sqrt{r} , 1]$ and any nonnegative integer $m \leq \frac{\delta^2}{256 r^2}$, there exists a multivariate distribution $D$ supported on $\bbR^d$ such that
	\begin{equation} \label{eq:original-program-popu}
		\begin{split}
			\quad &  D \geq (1 - \delta) N(ru, I_d),\quad  N(ru, I_d) \geq (1 - \delta )D \quad \text{ and }\quad  D\equiv_{m}N(0,I_d).
		\end{split}
	\end{equation}
\end{Lemma}

Now we divide the proof into two cases.  We will take $c_0$ to be a sufficiently small universal constant, so that for any $r \in [0,c_0]$ 
we have $\log(1/\sqrt{r}) \leq \frac{1}{256 r}$. Set $C = \frac{1}{c_0^2}$. 

{\bf \noindent Case 1: $ \sqrt{r } + \epsilon \geq c^2_0$}. In this case,  $\wt{\epsilon} \leq C( \sqrt{r } + \epsilon)$ clearly holds as $\wt{\epsilon} \leq 1$.

{\bf \noindent Case 2:  $ \sqrt{r } + \epsilon \leq c^2_0$}. In this case, we know $r \leq c_0$. By applying Lemma \ref{lm:population-existence-analysis} with $u = (\theta - \wh{\theta})/\|\wh{\theta}- \theta\| $ and $r = \|\wh{\theta}- \theta\| $, we know when $\delta = \sqrt{r}$, there exists a multivariate distribution $D$ such that \eqref{eq:original-program-popu} is satisfied for all $m \leq \log(1/\delta) \leq \frac{\delta^2}{256 r^2}$.  Therefore, for any contamination distribution $Q$ and any $\epsilon \in [0, \epsilon_{\max}]$, this $D$ also satisfies
		\begin{equation*}
			\begin{split}
				& (1 - \epsilon)N(\theta - \wh{\theta}, I_d) + \epsilon Q \geq (1 - \sqrt{r})(1 - \epsilon) D \text{ and } D \equiv_{m_{\delta}} N(0,I_d)	\;,		\end{split}
		\end{equation*} 
        since $m_{\delta} \leq \log(1/\delta)$. 
        This further implies that this $D$ satisfies
		\begin{equation*}
			\begin{split}
				& (1 - \epsilon)N(\theta - \wh{\theta}, I_d) + \epsilon Q \geq (1 - \delta') D \text{ and } D \equiv_{m_{\delta'}} N(0,I_d) \;,
			\end{split}
		\end{equation*}	where $\delta' = \sqrt{r} + \epsilon$. This is because $m_{\delta}$ is a decreasing function of $\delta$.	 Therefore, $\texttt{F}(\sqrt{r} + \epsilon)$ is non-empty. Since $\wt{\epsilon}$ is defined to be the smallest $\delta$ such that $\texttt{F}(\delta)$ is non-empty, we have $\wt{\epsilon} \leq \sqrt{r} + \epsilon$. This finishes the proof.


\subsection{Proof of Lemma \ref{lm:population-existence-analysis} }
First, if $\frac{\delta^2}{256 r^2} \in [0,1)$, then $m = 0$ and 
we can simply take $D = N(ru, I_d)$. Similarly, when $r = 0$, we simply take $D = N(0, I_d)$. Therefore, for the rest of the proof, we focus on $\frac{\delta^2}{256 r^2} \geq 1$, $r \in (0,1]$ and $m \geq 1$. As we have already mentioned in Section \ref{sec:quasi-poly-algorithm}, the construction of $D$ is based on NGCA. In particular, we set $D= P_u^A$ where $A(x) = (1 - \delta  )\phi(x - r) + \delta q(x)$ and 
\begin{equation} \label{eq:qx-density2}
q(x) = \frac{\phi(x) - (1 - \delta) \phi(x-r) }{\delta} \indi(|x| \leq t) + p(x) \indi(|x| \leq b)
\end{equation} with $b = t/4$.
Here $p(x)$ is a polynomial to be specified later, and $t \geq 4$ is some tuning parameter. It is easy to check that $D = P_{\delta, ru, Q_u}$ where $Q_u = P_u^q$.
To show the three results in \eqref{eq:original-program-popu}, we just need to show
\begin{equation*}
	\begin{split}
	&(\texttt{C1})\quad D \geq (1 - \delta ) N(ru, I_d); \quad (\texttt{C2})\quad   N(ru, I_d) \geq (1 - \delta) D;\\
	&(\texttt{C3})\quad   (1 - \delta) \int x^k \phi(x - r) dx + \delta \int x^k q(x) dx = \int x^k \phi(x) dx, \quad \forall k = 0, \cdots, m.
	\end{split}
\end{equation*} We note that when $k = 0$, (\texttt{C1}) and (\texttt{C3}) above ensure that $q(x)$ is a valid probability distribution. 
\vskip.2cm
{\bf \noindent (Check for (\texttt{C1}))} 
To check (i), it is enough to check that $q(x) \geq 0, \forall x \in \bbR$. 
To check that, we need
	\begin{equation} \label{ineq:three-requirement}
	\begin{split}
		 &(i)\quad \phi(x) - (1 - \delta) \phi(x-r) \geq 0, \quad \forall x \in [-t,t] \setminus [-b,b],\\
		 &(ii)\quad \phi(x) - (1 - \delta) \phi(x-r) +\delta p(x) \geq 0, \quad \forall x \in [-b,b].
	\end{split}
\end{equation}

First,
	\begin{equation*}
	\begin{split}
		(i) &\Longleftrightarrow \phi(x) - (1 - \delta) \phi(x-r) \geq 0, \quad \forall x \in [-t,t] \setminus [-b,b] \\
		&\Longleftrightarrow t \leq \frac{r}{2} + \frac{\log ( \frac{1}{1 - \delta } ) }{r} \Longleftarrow t \leq \frac{\log ( \frac{1}{1 - \delta } ) }{r} \Longleftarrow t \leq \frac{\delta}{r}.
	\end{split}
\end{equation*} In addition, since $$\phi(x - r) = \exp(xr-r^2/2) \phi(x) \leq \exp(br) \phi(x), \quad \forall x\in [-b,b],$$ we have
	\begin{equation*}
		 \phi(x) - (1 - \delta) \phi(x-r)  \geq (1 - (1 - \delta)\exp(br))\phi(x) \overset{(a)}\geq  \phi(x)\delta/2, \quad \forall x\in [-b,b],
	\end{equation*} where in (a), we assume that 
    $ r \leq \frac{\delta}{2( 1- \delta) b}$ which implies that 
    $(1 - \delta)\exp(br) \leq 1 - \delta/2$. 
    Therefore, given $ r \leq \frac{\delta}{2(1 - \delta) b}$, 
    (ii) is satisfied if $\sup_{x \in [-b,b]} |p(x)| \leq \frac{1}{2} \phi(b)$.
	
In summary, (\texttt{C1}) is satisfied if 
\begin{equation} \label{ineq:C1}
\begin{split}
	&r \leq \frac{\delta}{2(1 - \delta) b}, \quad t \leq \frac{\delta}{r} \quad \text{and}\quad \sup_{x \in [-b,b]} |p(x)| \leq \frac{1}{2} \phi(b) \\
	\overset{(a)}\Longleftarrow & t \leq \frac{\delta}{r} \quad \text{and}\quad \sup_{x \in [-b,b]} |p(x)| \leq \frac{1}{2} \phi(b),
\end{split}
\end{equation} where in (a), we plug in $b = t/4$. 

\vskip.2cm
{\bf \noindent (Check for (\texttt{C2}))}
First, by construction of $D$,
\begin{equation} \label{ineq:forth-condition}
	\begin{split}
		& N(ru, I_d) \geq (1 - \delta) D  \Longleftrightarrow \phi(x - r) \geq  (1 - \delta)A(x), \forall x \in \bbR\\
		& \overset{(a)}\Longleftarrow \phi(x-r) \indi(|x| \leq t) \geq (1 - \delta) \phi(x) \indi( |x| \leq t )+ \delta (1 - \delta) p(x) \indi(|x| \leq b) \;,
	\end{split}
\end{equation} where in (a), we plug in the definition of $A(x)$. To further check the last condition in \eqref{ineq:forth-condition}, we just need to consider two scenarios:
\begin{itemize}
	\item ($|x| \leq b$) First, we can check that $2br \leq \delta/2$ implies that 
	\begin{equation} \label{ineq:interme}
		\exp(-2br) \geq 1 - \delta/2.
	\end{equation} Then 
	\begin{equation} \label{ineq:interme2}
		\phi(x-r) = \phi(x) \exp(xr -r^2/2) \overset{(a)}\geq \phi(x) \exp(-2br) \overset{ \eqref{ineq:interme} } \geq (1 - \delta/2) \phi(x) \quad \forall x \in [-b,b], 
	\end{equation} where in (a), we use the fact $b \geq r/2$. Then
	\begin{equation*}
		\begin{split}
			\eqref{ineq:forth-condition} \text{ holds for all }x \in [-b,b] &\Longleftrightarrow \phi(x-r) \geq (1 - \delta) \phi(x) + \delta(1 - \delta) p(x), \forall x \in [-b,b] \\
			&  \overset{\eqref{ineq:interme2}}\Longleftarrow  \sup_{x \in [-b,b]} |p(x)| \leq \frac{1}{ 2 (1 - \delta) } \phi(b).
		\end{split}
	\end{equation*} 
	
	\item ($b<|x| \leq t$)
	\begin{equation*}
		\begin{split}
			\eqref{ineq:forth-condition} &\Longleftrightarrow \phi(x-r) \geq (1 - \delta) \phi(x), \forall |x| \in (b,t] \\
			& \Longleftrightarrow \exp(-tr -r^2/2) \geq 1 - \delta \overset{(a)}\Longleftarrow  t \leq \frac{\log\left( \frac{1}{1- \delta} \right)}{2r} \Longleftarrow t \leq \frac{\delta}{2r} \;,  
		\end{split}
	\end{equation*} where in (a) we use the fact $r/2 \leq t$.
\end{itemize} In summary, to ensure $N(ru, I_d) \geq (1 - \delta) D$, we need
\begin{equation} \label{ineq:second-condition2}
	\begin{split}
		&2br \leq \delta/2, \quad t \leq\frac{\delta}{2r} \quad \text{ and }\quad \sup_{x \in [-b,b]} |p(x)| \leq \frac{1}{ 2(1 - \delta) } \phi(b)\\
		\overset{(a)}\Longleftarrow  & t \leq\frac{\delta}{2r} \quad \text{ and }\quad \sup_{x \in [-b,b]} |p(x)| \leq \frac{1}{ 2(1 - \delta) } \phi(b),
	\end{split}
\end{equation}  where in (a), we plug in $b = t/4$.

\vskip.2cm
{\bf \noindent (Check for (\texttt{C3}))} First, we can write 

\begin{equation} \label{eq:moment-matching}
\begin{split}
		 & (1 - \delta) \int x^k \phi(x - r) dx + \delta \int x^k q(x) dx = \int x^k \phi(x) dx, \quad \forall k \in  0\cup [m] \\
		\overset{ \eqref{eq:qx-density2} }\Longleftrightarrow & (1 - \delta) \int x^k \phi(x-r) \indi(|x| > t) + \delta \int_{-b}^b x^k p(x) dx = \int x^k \phi(x) \indi(|x| > t) dx,  \quad \forall k \in  0\cup [m]  \\
		\Longleftrightarrow & \int_{|x| > t} x^k \phi(x) dx + \frac{1 - \delta}{\delta} \int_{|x| > t} x^k (\phi(x) - \phi(x-r))  dx= \int_{-b}^{b} x^k p(x) dx ,  \quad \forall k \in  0\cup [m] .
\end{split}
\end{equation} 
By Lemma \ref{lm:polynomial-matching-bound}, there exists a unique polynomial $p(x)$ of degree at most $m$ satisfying \eqref{eq:moment-matching} and in addition it also satisfies
\begin{equation*}
	\sup_{x \in [-b,b]} |p(x)| \leq  2m^2  \sum_{i=0}^{m} \left(2e\right)^i \left(  \max_{0 \leq k \leq i}|a_k| b^{-k-1} \right)  ,
\end{equation*} where $a_k = \int x^k \phi(x) \indi(|x| > t) dx + \frac{1 - \delta}{\delta} \int x^k (\phi(x) - \phi(x-r)) \indi(|x| > t) dx$.  

Next, we provide an upper bound for $|a_k|$ for any $k = 0, \ldots, m$. 
\begin{equation*}
	\begin{split}
		|a_k| \leq \underbrace{\left|\int x^k \phi(x) \indi(|x| > t) dx\right|}_{(I)} + \underbrace{\left|\frac{1 - \delta}{\delta} \int x^k (\phi(x) - \phi(x-r)) \indi(|x| > t) dx\right|}_{(II)}.
	\end{split}
\end{equation*} First, for any $k = 0, \ldots, m$,
\begin{equation} \label{ineq:bound-for-I}
	\begin{split}
		(I)  \leq \frac{2}{\sqrt{2\pi }} \int_{t}^{\infty} \exp(-x^2/2) x^k dx &\leq \frac{2}{\sqrt{2 \pi }} \left(\sup_{x \geq t} x^k \exp(-x^2/4)\right)  \int_{t}^{\infty} \exp(-x^2/4) dx \\
		& \overset{(a)}\leq \sqrt{ \frac{2}{\pi} } \cdot \frac{2}{t}  \exp(-t^2/4) \sup_{x \geq t} \exp(-x^2/4 + k \log x) \\
		& \overset{(b)}\leq \sqrt{ \frac{2}{\pi} } \cdot \frac{2}{t}  \exp(-t^2/2 + k \log t) \;,
	\end{split}
\end{equation} 
where (a) is because $\int_{t}^{\infty} \exp(-x^2/4) dx \leq \frac{2}{t} \exp(-t^2/4) $ by Theorem 1.2.3 of \cite{durrett2019probability} and (b) is because the function $g(x) = -x^2/4 + k \log x$ increases when $x \leq \sqrt{2k}$  and decreases when $x \geq \sqrt{2k}$ and we assume that we take $t \geq \sqrt{2m}$; 
therefore, $ \sup_{x \geq t} \exp(-x^2/4 + k \log x) \leq \exp(-t^2/4 + k \log t)$ for all $k = 0, \ldots, m$.

\begin{equation*}
	\begin{split}
		(II) &= \frac{1 - \delta}{\delta} \left| \int_t^{\infty} x^k (\phi(x) - \phi(x-r)) dx + \int_{- \infty}^{-t} x^k (\phi(x) - \phi(x-r)) dx \right| \\
		& \overset{ (a) }\leq \frac{1 - \delta}{\delta} \left( \int_t^{\infty} x^k \sup_{\xi \in [x-r, x]} | r \xi \frac{1}{\sqrt{2\pi}} \exp(-\xi^2/2)| dx + \int^{-t}_{-\infty} |x^k| \sup_{\xi \in [x-r, x]} |r \xi \frac{1}{\sqrt{2\pi}} \exp(-\xi^2/2)| dx \right) \\
		& \overset{(b)}\leq  \frac{1 - \delta}{\delta} \left( \int_t^{\infty} x^k  | r (x-r) \frac{1}{\sqrt{2\pi}} \exp(-(x-r)^2/2)| dx + \int^{-t}_{-\infty} |x^k|  |rx \frac{1}{\sqrt{2\pi}} \exp(-x^2/2)| dx \right) \\
		& \leq \frac{(1 - \delta)r}{\delta}  \left( \int_t^{\infty} x^k   (x-r) \frac{1}{\sqrt{2\pi}} \exp(-(x-r)^2/2) dx + \int_{t}^{\infty} x^{k+1} \frac{1}{\sqrt{2\pi}} \exp(-x^2/2) dx \right)\\
		& \overset{(c)}\leq  \frac{(1 - \delta)r}{\delta} \left( \int_t^{\infty} x^{k+1} \frac{1}{\sqrt{2\pi}} \exp(-x^2/4) dx + \int_{t}^{\infty} x^{k+1} \frac{1}{\sqrt{2\pi}} \exp(-x^2/2) dx \right)\\
		& \overset{(d)}\leq \frac{(1 - \delta)r}{ \sqrt{\pi} \delta} \sqrt{2} \cdot \frac{2}{t} \exp(-t^2/4 + (k+1) \log t) \;,
	\end{split}
\end{equation*} 
where in (a), we use Taylor expansion, and in (b) we use the facts that 
$t -r \geq 1$ and that the function $h(x) = |x \exp(-x^2/2)|$ increases on $x \leq -1$ and $x \in [0,1]$ and decreases on $x \in [-1,0]$ and $x \geq 1$; (c) holds as long as $t \geq \frac{\sqrt{2}}{\sqrt{2} - 1} r $; (d) follows from a similar analysis as in \eqref{ineq:bound-for-I} given that we choose $t \geq 2 \sqrt{m+1}$.  

In summary, as long as $t \geq 2\sqrt{m+1}$ and $r \leq 1$, there exists a polynomial $p(x)$ of degree at most $m$ such that condition in (\texttt{C3}) is satisfied and 
\begin{equation} \label{ineq:density-bound}
\begin{split}
	\sup_{x \in [-b,b]} |p(x)| &\leq  2m^2 (1+ \frac{r}{\delta})   \sum_{i=0}^{m} \left(2e\right)^i \left(  \max_{0 \leq k \leq i}\exp( -t^2/4 + (k+1) \log t - (k+1) \log b) \right)  \\
	& =  2m^2  (1+ \frac{r}{\delta})   \sum_{i=0}^{m} \left(2e\right)^i  \exp( -t^2/4 + (i+1) \log (t/b)) \\
	& \leq   2m^2  (1+ \frac{r}{\delta}) \frac{t}{b} \frac{ \exp(-t^2/4 + m \log(2et/b) )}{\log(2et/b) - 1} =    8m^2  (1+ \frac{r}{\delta})   \frac{ \exp(-t^2/4 + m \log(8e) )}{\log(8)}\\
	& \leq  16m^2   \frac{ \exp(-t^2/4 + m \log(8e) )}{\log(8)} \;,
\end{split}
\end{equation} where in the last equality, we plug in $b = t/4$ and in the last inequality, we use the fact $\delta \geq r$.

	 Combining \eqref{ineq:C1}, \eqref{ineq:second-condition2} and \eqref{ineq:density-bound}, it follows that we need to choose $t$ such that $ t \leq \frac{\delta}{2r}$ and 
\begin{equation*}
	\begin{split}
		\sup_{x \in [-b,b]} |p(x)| \leq \frac{1}{2}\phi(b).
	\end{split}
\end{equation*} In view of \eqref{ineq:density-bound}, a sufficient condition for the second one is 
\begin{equation*}
	16m^2   \frac{ \exp(-t^2/4 + m \log(8e) )}{\log(8)} \leq \frac{1}{2\sqrt{2\pi}} \exp(-t^2/32).
\end{equation*}
 By simple calculation, we can check that if $t \geq 8 \sqrt{m}$, then the above condition can be satisfied for any $m \geq 1$. So if we take any $(r, \delta)$ such that $\delta \in [\sqrt{r}, 1]$ and $\frac{\delta^2}{256 r^2} \geq 1$, and positive integer $m \leq \frac{\delta^2}{256 r^2}$, then the choice $t = \frac{\delta}{2r} $, $b=t/4$ can ensure all (\texttt{C1})-(\texttt{C3}) are satisfied.


\section{Proof of Theorem \ref{thm:finite-sample-analysis} in Section \ref{sec:quasi-finite-sample}  } \label{proof:finite-sample} 

\subsection{Preliminary analysis and high-probability events} \label{proof:preliminary}
In the following, we perform preliminary work and list a few high-probability events to be used later. Notice that since $\wh{\theta}$ is independent of $\{X_i\}$s, $X_1 -  \wh{\theta}, \ldots, X_n - \wh{\theta} \overset{iid}\sim   (1-\epsilon) N(\theta - \wh{\theta}, I_d) + \epsilon Q$ for some unknown $Q$. They can also be generated as follows: let $Y_1, \ldots, Y_n \overset{iid}\sim N(\theta - \wh{\theta}, I_d)$, $O_1, \ldots, O_n \overset{iid} \sim Q $ and $n_0 \sim \text{Binomial}(n, 1-\epsilon)$ be generated independently of $\{Y_i\}$s and $\{O_i\}$s. If we denote a random shuffle of $\{ Y_1, \ldots, Y_{n_0}, O_{n_0+1}, \ldots, O_{n} \}$ as $\{X_1', \cdots, X_n'\}$, then $X_1', \cdots, X_n'  \overset{iid}\sim   (1-\epsilon) N(\theta - \wh{\theta}, I_d) + \epsilon Q$. So without loss of generality, we can simply view $\{X_1 -  \wh{\theta}, \ldots, X_n - \wh{\theta}  \}$ as a random shuffle of $\{ Y_1, \ldots, Y_{n_0}, O_{n_0+1}, \ldots, O_{n} \}$. We also denote $r = \|\wh{\theta}- \theta\| $ and $u = (\theta - \wh{\theta})/\|\wh{\theta}- \theta\| $. 

First, it is easy to check $|\Delta| \leq \log n$. For each $\delta \in \Delta$, let $Z_{\delta, i} =V_i^{m_{\delta}}$ where $V_i = Y_i^\top u$ for $i = 1 \ldots, n$. Notice that conditioning on $\wh{\theta}$, $V_i \overset{iid}\sim N(r,1)$, $Z_{\delta, i} \overset{iid} \sim P_{\delta}$ where $P_{\delta}$ is the distribution of $V^{m_{\delta}}$ given $V \sim N(r,1)$. We denote the CDF of $Z_{\delta, i}\sim P_{\delta}$ as $F_{\delta}$.

 The following high probability events will be frequently used in the rest of the analysis: for $n_0 \geq 1$, let $F_{n_0}(\cdot)$ be the empirical CDF of $\{V_i \}_{i=1}^{n_0}$ and $F(\cdot)$ be the CDF of $N(r,1)$,
\begin{equation}\label{eq:high-prob-events}
	\begin{split}
			\cE_1 &= \left\{ \|\wh{\theta} -\theta\|_2 \leq C\left( \sqrt{\frac{d}{n}} + \sqrt{\frac{\log(1/\alpha)}{n} } + \epsilon \right)^{0.1} \right\},\\
	\cE_2 & = \left\{ |n_0 - n(1 - \epsilon)| \leq \frac{n\epsilon}{2} + 2 \log(8/\alpha) \right\},\\
	\cE_3 & = \Big \{  \left| \frac{1}{n_0} \sum_{i=1}^{n_0} Z_{\delta, i} \indi( Z_{\delta, i} \leq F_{\delta}^{-1}(1 - 2\delta) ) - \bbE_{P_{\delta}}[Z \indi( Z \leq F_{\delta}^{-1}(1 - 2\delta) )]  \right| \\
	& \quad \quad \leq 2 \sqrt{\frac{2 \log n}{\alpha n_0}} \exp(3r\sqrt{2m_{\delta}}/4) (2m_{\delta}/e)^{m_{\delta}/2} \text{ for all }\delta \in \Delta \cap [0,1/2) \Big \},\\
	\cE_4 & = \left\{ \sup_{t \in \bbR} |F_{n_0}(t) - F(t)| \leq \sqrt{ \frac{\log(8/\alpha)}{2n_0} } \right\}.
	\end{split}
\end{equation}

The claim is that
\begin{equation} \label{ineq:high-prob-event}
\begin{split}
&\bbP(\cE_1) \geq 1 - \alpha/4,\quad  \bbP(\cE_2) \geq 1 - \alpha/4,\quad  \bbP(\cE_3) \geq 1 - \alpha/4, \quad \bbP(\cE_4) \geq 1 -\alpha/4,
\end{split}
\end{equation} where the randomness is taken with respect to $\wh{\theta}$, $\{X_i\}_{i=1}^n$ and $n_0$. The proof of this claim will be postponed to Section \ref{sec:proof-high-prob-event}.

 In the following two subsections, we will first show the coverage guarantee using \eqref{ineq:high-prob-event} in Section \ref{sec:coverage-analysis} and then show the diameter guarantee in Section \ref{sec:radius-analysis}.

\subsection{Coverage Analysis} \label{sec:coverage-analysis}
This is an extension of the first inequality in Theorem \ref{th:hat-epsilon-guarantee}. We will show that given events $\cE_1- \cE_4$ in \eqref{eq:high-prob-events}, then $r\leq \frac{11}{\sqrt{ m_{\wh{\epsilon}} }}$. This directly implies the coverage of $\wh{\CS}$ with probability at least $1 - \alpha$. First, we note $r \leq 1$ on event $\cE_1$ under our assumption on $n$. So if $\wh{\epsilon} \geq \exp(-1000)$ or $0 < m \leq 100$, then $r\leq \frac{11}{\sqrt{ m_{\wh{\epsilon}} }}$ clearly holds. Therefore, we just need to focus on the case $\wh{\epsilon} \in [0, \exp(-1000)]$ and $m \geq 100$. Let
\begin{equation*}
	\texttt{LP}(\delta) = \left\{  w \in \cW_{n,\delta}: T_{m_\delta}(w,\widehat{\theta})
\le
2\left(1+\frac{4}{\alpha}\right)
\sqrt{\frac{1}{n(1-\delta)}} \right\}.
\end{equation*}

Next, we are going to show that given $\cE_1- \cE_4$ occur the following holds: for any $\delta \in \Delta \cap [0,\exp(-1000)]$,
\begin{equation} \label{ineq:finite-sample-coverage}
	r > \frac{11}{\sqrt{ m_{\delta} }} \Longrightarrow\texttt{LP}(\delta)= \emptyset.
\end{equation}
Therefore, $\texttt{LP}({\wh{\epsilon}}) \neq \emptyset$ for $\wh{\epsilon} \in \Delta \cap [0,\exp(-1000)] $ implies that $r\leq \frac{11}{\sqrt{ m_{\wh{\epsilon}} }}$. 
Next, we show \eqref{ineq:finite-sample-coverage} by contradiction. Suppose $\texttt{LP}(\delta) \neq \emptyset$ when $ r > \frac{11}{\sqrt{ m_{\delta} }}$ for some $\delta \in \Delta \cap [0,\exp(-1000)]$, i.e., there exists $w \in \cW_{n,\delta}$ such that $T_{m_\delta}(w,\widehat{\theta})
\le
2\left(1+\frac{4}{\alpha}\right)
\sqrt{\frac{1}{n(1-\delta)}}$. Then, we argue that this implies 
\begin{equation} \label{ineq:extend-to-all-polynomial}
	\sup_{p \in \cP_{m_{\delta},d}}\left|\sum_{i=1}^n w_i p(X_i - \wh{\theta} ) - \bbE_{N(0, I_d)}[p(X)]\right| \leq 2(1 + 4/\alpha)  \sqrt{ \bbE_{N(0, I_d)}[p^2(X)] }  \frac{N}{\sqrt{n( 1- \delta)}},
\end{equation} where $N = {d+m_{\delta} \choose m_{\delta}}$. To show this, without loss of generality, we can assume $\bbE_{N(0, I_d)}[p^2(X)] = 1$, i.e., $p \in \cF_{m_{\delta},d}$ where $\cF_{m_{\delta},d} = \{p \in \cP_{m_{\delta},d}: \bbE_{N(0, I_d)}[p^2(X)] = 1 \}$. Recall that $\cH_{m_{\delta},d}$ forms a basis of $\cF_{m_{\delta},d}$ and denote the basis functions in it as $h_1, \ldots, h_N$. Then any $p \in \cF_{m_{\delta},d}$ can be written as 
\begin{equation*}
	p(x) = \langle \eta, \Phi_{m_{\delta}}(x) \rangle, \text{ for some } \eta \in \bbR^N \text{ with }\|\eta\|_2 = 1,
\end{equation*} where $ \Phi_{m_{\delta}}(x)  \in \bbR^N$ and each coordinate represents a distinct basis function evaluated on $x$, i.e., $\Phi_{m_{\delta}}(x) = \{h_j(x)\}_{j=1}^N$. Therefore,
\begin{equation*}
	\begin{split}
	& \sup_{p \in \cF_{m_{\delta},d} } \left|\sum_{i=1}^n w_i p(X_i - \wh{\theta} ) - \bbE_{N(0, I_d)}[p(X)]\right| = \sup_{\eta \in \bbR^N, \|\eta\|_2 = 1} \left|  \sum_{i=1}^n w_i  \langle \eta, \Phi_{m_{\delta}}(X_i -\wh{\theta} ) \rangle  -  \bbE_{N(0, I_d)}[\langle \eta, \Phi_{m_{\delta}}(X ) \rangle] \right| \\
	& =  \sup_{\eta \in \bbR^N, \|\eta\|_2 = 1} \left| \sum_{j=1}^N \eta_j \left(  \sum_{i=1}^n w_i  h_j(X_i - \wh{\theta})  -  \bbE_{N(0, I_d)}[ h_j(X)] \right) \right| \\
		& \leq  \sqrt{ \sum_{j=1}^N \left(  \sum_{i=1}^n w_i  h_j(X_i - \wh{\theta})  -  \bbE_{N(0, I_d)}[ h_j(X)] \right) ^2 }\leq  2(1 + 4/\alpha) \frac{N}{ \sqrt{n( 1- \delta)}  } \;,
	\end{split}
\end{equation*} where the first inequality is Cauchy-Schwarz 
and the second inequality is because $T_{m_\delta}(w,\widehat{\theta})
\le
2\left(1+\frac{4}{\alpha}\right)
\sqrt{\frac{1}{n(1-\delta)}}$.


Next, let us focus on the special polynomial $p(x) = (x^\top u)^m$ where $u = (\theta - \wh{\theta})/\|\wh{\theta}- \theta\| $. We write $\{Z_{\delta,(i)}\}_{i=1}^n$ as the order statistics of $\{Z_{\delta,i}\}_{i=1}^n$. We first derive a convenient lower bound for the order statistics $ Z_{\delta,( \lfloor n_0 - n\delta \rfloor)}$. Recall that $Z_{\delta, i} = V_i^{m_{\delta}}$ and conditioning on $\wh{\theta}$, $V_i \overset{iid}\sim N(r,1)$. Since $V_i$ has a continuous density, with probability one, there will be no ties among $\{Z_{\delta,i}\}_{i=1}^n$. Denote the empirical CDF of $\{Z_{\delta,i}\}_{i=1}^{n_0}$ as $F_{\delta,n_0}(\cdot)$ and the corresponding population CDF as $F_{\delta}(\cdot)$. Then
	\begin{equation} \label{ineq:DKW-implication}
		\begin{split}
			 \sup_{t \in \bbR} |F_{\delta,n_0}(t) - F_{\delta}(t)| &=  \sup_{t \geq 0} \left| \frac{1}{n_0} \sum_{i=1}^{n_0} \indi(|V_i| \leq t^{1/m_{\delta}})  - \bbP_{N(r,1)}(|V| \leq t^{1/m_{\delta}})\right| \\
			 &\leq 2 \sup_{t \in \bbR}  |F_{n_0}(t) - F(t)| \overset{ \cE_4 }\leq \sqrt{ \frac{2\log(8/\alpha)}{n_0} }. 
		\end{split}
	\end{equation} 
Given $\cE_2$, we know $n_0 \geq 0.9n$ since $\epsilon \leq 0.05$. Thus, for any $\delta \in \Delta \cap [0,\exp(-1000)]$, 
\begin{equation} \label{ineq:Z-quantile-ineq}
	\begin{split}
		& F_{\delta}( Z_{\delta,( \lfloor n_0 - n\delta \rfloor)}) \overset{\eqref{ineq:DKW-implication} }\geq F_{\delta,n_0}( Z_{\delta,( \lfloor n_0 - n\delta \rfloor )}) -  \sqrt{ \frac{2\log(8/\alpha)}{n_0} } \overset{(a)}\geq  1 -  \frac{n \delta+1}{n_0} -  \sqrt{ \frac{2\log(8/\alpha)}{n_0} } \overset{(b)}\geq 1 - 2 \delta\\
		\Longrightarrow & Z_{\delta,( \lfloor n_0 - n\delta \rfloor)}  \geq F_{\delta}^{-1}(1 - 2 \delta) \;,
	\end{split}
\end{equation} 
where (a) is because $F_{\delta,n_0}( Z_{\delta,( \lfloor n_0 - n\delta \rfloor )})  \geq 1 -  \frac{n \delta+1}{n_0} $; and (b) is because $4\sqrt{\frac{ \log(8/\alpha)}{n} }   \leq  \delta$ since $\delta \in \Delta$ and $n_0 \geq 0.9n$.

Then, for $p(x) = (x^\top u)^{m_{\delta}}$,
\begin{equation} \label{ineq:converage-mean-lower}
	\begin{split}
		&\min_{w \in \cW_{n,\delta}} \sum_{i=1}^n w_i p(X_i - \wh{\theta}) = \min_{w \in \cW_{n,\delta}} \left(\sum_{i=1}^{n_0} w_i p(Y_i) + \sum_{i = n_0+1}^n w_i p(O_i) \right)\\
		& \overset{(a)}\geq  \min_{\substack{w \in  [0, \frac{1}{n(1 - \delta)} ]^{ n_0}\\ \sum_{i=1}^{n_0} w_i = 1 - \frac{n-n_0}{n(1 - \delta)} }} \sum_{i=1}^{n_0} w_i p(Y_i)  = \min_{\substack{w \in  [0, \frac{1}{n_0- n\delta} ]^{ n_0}\\ \sum_{i=1}^{n_0} w_i = 1  }} \frac{n_0- n \delta}{n (1 - \delta)} \sum_{i=1}^{n_0} w_i Z_{\delta,i}  \\
		& \geq \frac{ n_0- n \delta}{n (1 - \delta)} \sum_{i=1}^{\lfloor n_0- n \delta \rfloor }  \frac{1}{n_0- n\delta}  Z_{\delta,(i)} =  \frac{1}{n (1 - \delta)} \sum_{i=1}^{n_0}  Z_{\delta,i} \indi(Z_{\delta,i} \leq Z_{\delta,(\lfloor n_0- n \delta \rfloor )})\\
		& \overset{ \eqref{ineq:Z-quantile-ineq} }\geq  \frac{1}{n (1 - \delta)} \sum_{i=1}^{n_0}  Z_{\delta,i} \indi(Z_{\delta,i} \leq F_{\delta}^{-1}(1 - 2 \delta)) \\ 
		& \overset{\cE_3}\geq \frac{n_0}{n(1 - \delta)} \left( \bbE_{Z\sim P_{\delta}}[Z \indi(Z \leq F_{\delta}^{-1}(1 - 2 \delta) )] -  2 \sqrt{\frac{2 \log n}{\alpha n_0}} \exp(3r\sqrt{2m_{\delta}}/4) (2m_{\delta}/e)^{m_{\delta}/2}\right) \\
		& \overset{(b) }\geq  \frac{n_0}{n(1 - \delta)} \left(  \frac{1}{2\sqrt{2}e} \exp(-r^2/2 + r \sqrt{m_{\delta}/e}) (m_{\delta}/e)^{m_{\delta}/2}-  2 \sqrt{\frac{2 \log n}{\alpha n_0}} \exp(3r\sqrt{2m_{\delta}}/4) (2m_{\delta}/e)^{m_{\delta}/2} \right)\\
		& \overset{(c)}\geq \frac{1}{2}\left(  \frac{1}{2\sqrt{2e}e} \exp( r \sqrt{m_{\delta}/e}) (m_{\delta}/e)^{m_{\delta}/2}-  4 \sqrt{\frac{ \log n}{\alpha n}} \exp(3r\sqrt{2m_{\delta}}/4) (2m_{\delta}/e)^{m_{\delta}/2}\right)\\
		& \overset{(d)}\geq \frac{1}{4}\cdot \frac{1}{2\sqrt{2e}e} \exp( r \sqrt{m_{\delta}/e}) (m_{\delta}/e)^{m_{\delta}/2} \;,
	\end{split}
\end{equation} where (a) is because $p(O_i) \geq 0$ and the maximum total weight we can put on the set of $\{O_i\}$s is $\frac{n-n_0}{n(1 - \delta)}$ and on the event $\cE_2$, $1 - \frac{n-n_0}{n(1 - \delta)} = \frac{n_0- n \delta}{n(1 - \delta)} \in [0,1]$; (b) is because Lemma \ref{lm:gaussian-tail}(iv) given $m_{\delta} \geq 20$ for all $\delta \in \Delta \cap [0,\exp(-1000)]$ and $\Phi^{-1}(1 - \delta) - r \geq \sqrt{2m_{\delta}} $ holds (this indeed holds as under our assumption as on $\cE_1$, $r \leq 1$ and $\Phi^{-1}(1 - \delta) - r \geq\sqrt{\log(1/ (\sqrt{2\pi}\delta))}	- r \geq \sqrt{2m_{\delta}}$, where the first inequality is by Lemma \ref{lm:normal-quantile} and second inequality is because $m_{\delta} \leq \log(1/\delta)/9$, $\delta \in [0,\exp(-1000)]$); (c) is because $r \leq 1$ and on $\cE_2$, $n_0 \geq 0.9n$; (d) is because under the assumptions of $m_{\delta} \leq (\log_d n)/4$, i.e., $d^{4m_{\delta}} \leq n$ and $n \gtrsim 1/\alpha^4$, we have $ 4 \sqrt{\frac{ \log n}{\alpha n}} \exp(3r\sqrt{2m_{\delta}}/4) (2m_{\delta}/e)^{m_{\delta}/2} \leq \frac{1}{2}  \frac{1}{2\sqrt{2e}e} \exp( r \sqrt{m_{\delta}/e}) (m_{\delta}/e)^{m_{\delta}/2} $.


At the same time, by \eqref{ineq:extend-to-all-polynomial}, we know for $p(x) = (x^\top u)^{m_{\delta}}$, 
\begin{equation} \label{ineq:right-hand-upper}
	\begin{split}
		 \min_{w \in \cW_{n,\delta}} \sum_{i=1}^n w_i p(X_i - \wh{\theta}) \leq & \bbE_{N(0, I_d)}[p(X)] +  2\sqrt{ \bbE_{N(0, I_d)}[p^2(X)] } (1 + 4/\alpha)  \frac{N}{\sqrt{n(1 - \delta)}} \\
		 = & \bbE_{N(0,1)} [X^{m_{\delta}}] + 2\sqrt{\bbE_{N(0,1)}[X^{2m_{\delta}}]} (1 + 4/\alpha)  \frac{N}{\sqrt{n(1 - \delta)}} \\
		 \leq & \sqrt{2}e (m_{\delta} /e)^{m_{\delta} /2} + 2{ \sqrt{\sqrt{2}e  (2m_{\delta} /e)^{m_{\delta} } }} (1 + 4/\alpha)  \frac{N}{\sqrt{n(1 - \delta)}} \\
		 \leq &\sqrt{2}e  (m_{\delta} /e)^{m_{\delta} /2} + 2{ \sqrt{ \sqrt{2}e (2m_{\delta} /e)^{m_{\delta} } }} (1 + 4/\alpha)  \frac{(ed/m_{\delta}  + e)^{m_{\delta} }}{\sqrt{n(1 - \delta)}}   \\
		 \overset{(a)}\leq & \sqrt{2}e (m_{\delta} /e)^{m_{\delta} /2} + \frac{1}{\alpha} (m_{\delta} /e)^{m_{\delta} /2}\frac{d^{m_{\delta} }}{\sqrt{n}}  \\
		 \leq & 5\sqrt{ (m_{\delta} /e)^{m_{\delta} } } \;,
	\end{split}
\end{equation} 
where (a) is because $m_{\delta}  \geq 20$ given $\delta \in [0,\exp(-1000)]$ and $m \geq 100$ and in the last step, we use the fact that $m_{\delta}  \leq {\log_{d}(n)}/4$ and $n \geq 1/\alpha^4$. 

Finally, when $r \geq \frac{11}{\sqrt{m_{\delta} }} $, for $p(x) = (x^\top u)^{m_{\delta}}$,
\begin{equation*}
	\begin{split}
		\min_{w \in \cW_{n,\delta}} \sum_{i=1}^n w_i p(X_i - \wh{\theta}) \overset{ \eqref{ineq:converage-mean-lower} }\geq \frac{1}{4}\cdot \frac{1}{2\sqrt{2e}e} \exp( r \sqrt{m_{\delta} /e}) (m_{\delta} /e)^{m_{\delta} /2}
&> 5 \sqrt{ (m_{\delta} /e)^{m_{\delta} } } \\
&\overset{ \eqref{ineq:right-hand-upper} }\geq \min_{w \in \cW_{n,\delta}} \sum_{i=1}^n w_i p(X_i - \wh{\theta}).
	\end{split}
\end{equation*} So, we have a contradiction. Thus, the assumption $\texttt{LP}(\delta) \neq \emptyset$ does not hold. 

\subsection{Diameter Analysis}\label{sec:radius-analysis}
This corresponds to extending the second inequality in Theorem \ref{th:hat-epsilon-guarantee} to the finite sample setting. First, we note that since $1 \in \Delta$ and $\texttt{LP}(1) \neq \emptyset$, $\wh{\epsilon} \leq 1$ (we assume the convention that $1/0 = \infty$).  To show the diameter claim, we will show that with probability at least $1 - \alpha$,
\begin{equation} \label{ineq:hat-epi-upper-bound}
	 \widehat{\epsilon} \leq C' \left( \sqrt{\frac{d}{n}} + \sqrt{\frac{\log(1/\alpha)}{n} } + \epsilon \right)^{0.05},
\end{equation} for a sufficiently large universal constant $C' > 0$ to be specified later. By simple manipulation, we can see that given \eqref{ineq:hat-epi-upper-bound},
\begin{equation*}
\begin{split}
	\frac{1}{\sqrt{m}_{ \wh{\epsilon} }} &\lesssim   \frac{1}{\sqrt{\log (n/d )} }+\frac{1}{\sqrt{ \log(1/\epsilon) }}+ \frac{1}{\sqrt{ \log(n/\log(1/\alpha)) }} +\frac{1}{\sqrt{\log_d n}} \\
	&\asymp  \frac{1}{\sqrt{\log (n/d )} }+\frac{1}{\sqrt{ \log(1/\epsilon) }}+ \frac{1}{\sqrt{ \log(n/\log(1/\alpha)) }} +\frac{1}{\sqrt{m}} \;,
\end{split}
\end{equation*}where the second equivalence is due to $n \asymp d^m$. This implies the desired result.

To show \eqref{ineq:hat-epi-upper-bound}, let us first define another set as follows:
	\begin{equation} \label{def:finite-sample-P}
		\texttt{F}'(\delta)= \left\{			D: P_{\epsilon, \theta - \wh{\theta}, Q} \geq (1-\delta/2) D 
			 \text{ and } D\equiv_{2m_\delta}N(0,I_d) \right\} \;,
	\end{equation} 
    where $m_{\delta}$ is defined in \eqref{def:m-delta}. 
	It is easy to check that
	\begin{equation} \label{ineq:P-monotone}
		\texttt{F}'(\delta)\neq \emptyset \Longrightarrow \texttt{F}'(\delta_0)\neq \emptyset \text{ for any } \delta_0 \geq \delta
	\end{equation}
	  since $m_{\delta}$ is a decreasing function of $\delta$. We are going to construct a multivariate distribution $D$ such that the constraint in \eqref{def:finite-sample-P} is satisfied. Notice that since $Q$ is unknown and could be arbitrary, our construction of $D$ should be agnostic to $Q$. The key technical lemma we leverage is Lemma \ref{lm:population-existence-analysis}. Let us take a sufficiently small constant $c_0$ such that for any $a \in [0,c_0]$, we have $\log(1/\sqrt{a}) \leq \frac{1}{256a}$. Now, let us take $C' = \frac{1}{c_0^2}$ and we are going to show $\widehat{\epsilon} \leq 4C' \left(\sqrt{r} + \epsilon + \sqrt{ \frac{\log(1/\alpha)}{n} } \right)$. We divide the proof into two cases.

{\bf \noindent Case 1: $\left(\sqrt{r} + \epsilon + \sqrt{ \frac{\log(1/\alpha)}{n} } \right) \geq c_0^2/4$.} In this case, $\widehat{\epsilon} \leq 4C' \left(\sqrt{r} + \epsilon + \sqrt{ \frac{\log(1/\alpha)}{n} } \right)$ clearly holds as $\wh{\epsilon} \leq 1$.

{\bf \noindent Case 2: $\left(\sqrt{r} + \epsilon + \sqrt{ \frac{\log(1/\alpha)}{n} } \right) \leq c_0^2/4$.} In this case, we have $r \leq c_0$ and by applying Lemma \ref{lm:population-existence-analysis} with $u = (\theta - \wh{\theta})/\|\wh{\theta}- \theta\| $, $r = \|\wh{\theta}- \theta\| $ and $\delta = 2\sqrt{r}$, we have for any positive integer $m' \leq \log(2/\delta) \leq \frac{\delta^2}{1024 r^2}$, there exists a multivariate distribution $D$ supported on $\bbR^d$ such that
	\begin{equation} \label{eq:original-program}
		\begin{split}
			\quad &N(\theta - \wh{\theta}, I_d) \geq (1 - \delta/2) D\quad \text{ and }\quad D \equiv_{m'} N(0,I_d)	.
		\end{split}
	\end{equation} Since $2m_{\delta} \leq \log(2/\delta)$ and $m_{\delta}$ is a decreasing function of $\delta$, for any contamination distribution $Q$ and $\epsilon \in [0,\epsilon_{\max}]$, this $D$ also satisfies
		\begin{equation*}
			\begin{split}
				& P_{\epsilon, \theta - \wh{\theta}, Q} \geq (1 - \sqrt{r})(1 - \epsilon) D \text{ and }D \equiv_{2m_{\delta}} N(0,I_d)\\
				& \Longrightarrow  P_{\epsilon, \theta - \wh{\theta}, Q} \geq (1 - \delta'/2) D \text{ and } D \equiv_{2m_{\delta'}} N(0,I_d),
			\end{split}
		\end{equation*} where $\delta' = 2(\sqrt{r} + \epsilon)$. Therefore, $\texttt{F}'(2(\sqrt{r} + \epsilon))$ is non-empty.


From now on, we set
\begin{equation} \label{eq:final-delta}
	\delta = \arg\min \left\{ \delta' \in \Delta : \delta'/2 \geq  \sqrt{r} + \epsilon + \sqrt{ \frac{\log(1/\alpha)}{n} } \right\} \;.
\end{equation}
By \eqref{ineq:P-monotone}, $\texttt{F}'(\delta)$ is non-empty. In addition, given 
any $D \in \texttt{F}'(\delta)$, we let 
$\frac{d D}{d P_{\epsilon, \theta - \wh{\theta}, Q} }$ 
be the Radon-Nikodym derivative of $D$ with respect to 
$P_{\epsilon, \theta - \wh{\theta}, Q} $. Denote $X_i' = X_i - \wh{\theta}$. 
Then given $\wh{\theta}$, $X_1', \cdots, X_n' \sim P_{\epsilon, \theta - \wh{\theta}, Q} $. In addition, given $\wh{\theta}$, we let $\wt{P}$ be a distribution on 
$(X', B) \in \bbR^d \times \{0,1\}$ constructed by first sampling 
$X' \sim P_{\epsilon, \theta - \wh{\theta}, Q}  $, then sampling 
$B|X' \sim \text{Bernoulli}\left( (1- \delta/2) \cdot \frac{d D }{ d P_{\epsilon, \theta - \wh{\theta}, Q} }(X') \right)$. It is easy to check, marginally, that 
$B\sim \text{Bernoulli}(1 - \delta/2$). We note that the $\{B_i\}$s can be viewed 
as rejection sampling indicators of getting samples drawn from $D$ 
given samples drawn from $P_{\epsilon, \theta - \wh{\theta}, Q} $. 
In particular, $X'|B = 1 \sim D$, since for any measurable set $\mathcal{B}$ 
in $\bbR^d$, 
\begin{equation} \label{ineq:rejection-sampling}
	\begin{split}
		\wt{\bbP}(X' \in \mathcal{B} |B = 1 ) = \frac{\wt{\bbP}(X' \in \mathcal{B},B = 1 )}{\bbP(B = 1)} = \frac{ \int_{\mathcal{B}} \wt{\bbP}(B=1 | X') d P_{\epsilon, \theta - \wh{\theta}, Q}(X') }{\bbP(B = 1)} = \int_{\mathcal{B}} dD(X').
	\end{split}
\end{equation} Let us define a set of weights $w_i = \frac{B_i}{\sum_{i=1}^n B_i}$ for $i = 1, \ldots, n$. Next, we are going to list two more high-probability events and show that conditioning on these two high-probability events, as well as $\cE_1$, $\texttt{LP}(\delta)$ is non-empty given this set of weights $w$ and $\delta$ in \eqref{eq:final-delta}.
\begin{equation*}
	\begin{split}
		\cE_5 &= \left\{ \sum_{i=1}^n B_i \geq (1 - \delta)n \right\},\\
		\cE_6 &= \left\{ \sup_{p \in \cF_{m_{\delta},d} } \left|\sum_{i=1}^n w_i p(X_i - \wh{\theta}) - \bbE_{N(0,I_d)}[p(X)]\right| \leq 2 (1 + 4/\alpha) \sqrt{ \frac{N}{ \sum_{i=1}^n B_i } }  \right\}.
	\end{split}
\end{equation*} The claim is that 
\begin{equation}
	\begin{split}
		\bbP(\cE_5) \geq 1 - \alpha/4 \text{ and } \bbP(\cE_6) \geq 1 - \alpha/4,
	\end{split}
\end{equation} where the randomness is taken with respect to $\wh{\theta}$, $\{X_i\}_{i=1}^n$ and $B_i$. The proof of this claim will be postponed to 
Section \ref{sec:proof-high-prob-event}.

Given $\cE_1$, $\cE_5$ and $\cE_6$ happen, it is easy to check that $\texttt{LP}(\delta)$ is non-empty with the constructed weights $w_i = \frac{B_i}{\sum_{i=1}^n B_i}$. Therefore, given events $\cE_1, \cE_5, \cE_6$ happen, we have
\begin{equation*}
	\wh{\epsilon} \overset{\cE_5, \cE_6}\leq \delta \text{ in }\eqref{eq:final-delta} \lesssim  \left(\sqrt{r} + \epsilon + \sqrt{ \frac{\log(1/\alpha)}{n} } \right) \overset{\cE_1}\lesssim   \left( \sqrt{\frac{d}{n}} + \sqrt{\frac{\log(1/\alpha)}{n} } + \epsilon \right)^{0.05},
\end{equation*} where in the first inequality, we also use the fact that $\wh{\epsilon}$ is the smallest $\delta \in \Delta$ such that $\texttt{LP}(\delta)$ is non-empty. Finally, the desired claim follows as $\bbP(\cE_1 \cap \cE_5 \cap \cE_6 ) \geq 1 - \alpha$.


\subsection{Proof of Claim \eqref{ineq:high-prob-event} for High Probability Events } \label{sec:proof-high-prob-event}
{\noindent \bf  Event $\cE_1$.} $\cE_1$ is achieved by the choice of the robust estimator we use.

\vskip.3cm
{\noindent \bf  Event $\cE_2$.} Here the result follows from a direct application of Bernstein's inequality for the Binomial; see Lemma \ref{lm:Binomial-prop} (setting $C = 2$ there).

\vskip.3cm
{\noindent \bf  Event $\cE_3$.} To show $\bbP(\cE_3) \geq 1 - \alpha/4$, it suffices to show that $\bbP(\cE_3| n_0, \wh{\theta} ) \geq 1 - \alpha/4$ for all $n_0$ 
and $\wh{\theta}$. Notice that given any $\wh{\theta}$, $n_0$ and $\delta \in \Delta$, by Chebyshev's inequality we have
\begin{equation*}
	\begin{split}
		&\bbP\left(   \left| \frac{1}{n_0} \sum_{i=1}^{n_0} Z_{\delta, i} \indi( Z_{\delta, i} \leq F_{\delta}^{-1}(1 - 2\delta) ) - \bbE_{Z \sim P_{\delta}}[Z \indi( Z \leq F_{\delta}^{-1}(1 - 2\delta) )]  \right| >   \sqrt{\frac{4\bbE_{P_{\delta} }[Z^2] \log n}{ n_0 \alpha} }  \right) \\
		& \leq \frac{\alpha}{ 4 \log n}. 
	\end{split}
\end{equation*}  Then the final result follows from a union bound on $\delta \in \Delta$ given $|\Delta| \leq \log n$ and the fact $ \bbE_{P_{\delta}}[Z^2] = \bbE_{N(r,1)}[X^{2m}] \leq 2 \exp(3r\sqrt{2m}/2) (2m/e)^{m}$ where the inequality is by Lemma \ref{lm:gaussian-tail} (iii).

\vskip.3cm
{\noindent \bf  Event $\cE_4$.} Given $\wh{\theta}$ and $n_0$, it follows by the 
the DKW inequality (see Lemma \ref{lm:DKW}). Since the result holds for any $\wh{\theta}$ and $n_0$, the result holds unconditionally as well.

\vskip.3cm 
{\noindent \bf  Event $\cE_5$.} Since $B_i \overset{iid}\sim \text{Bernoulli}(1 - \delta/2)$, 
\begin{equation*}
	\begin{split}
		 \bbP\left(\sum_{i=1}^n B_i < n(1 - \delta)\right) =  \bbP\left(\text{Binomial}(n,\delta/2) > n \delta \right) \leq  e^{-n\delta/6} \leq \alpha/4 \;,
	\end{split}
\end{equation*} where the first inequality is by the one-sided Chernoff bound and the second inequality is because $\delta \geq \sqrt{ \frac{\log(1/\alpha)}{n} }$ by definition \eqref{eq:final-delta} and $n \gtrsim \log(1/\alpha)$.

\vskip.3cm
{\noindent \bf  Event $\cE_6$.} First, given $\wh{\theta}$, by \eqref{ineq:rejection-sampling}, we know that given $B_i = 1$, $X_i - \wh{\theta} \sim D$. Moreover, we know the distribution $D$ satisfies $D\equiv_{2m_{\delta}} N(0, I_d)$. The following lemma provides the key concentration inequality we need; its proof is given in the subsequent subsections.
\begin{Lemma} \label{eq:assump-concentration}
	 Assume $ d, m$ are positive integers. Suppose the distribution $D$ supported on $\bbR^d$ satisfies that $\bbE_D[p(X)] = \bbE_{N(0, I_d)}[p(X)]$ for any polynomial $p \in \cP_{2m,d}$. Let $X_1, \ldots, X_{n}$ be i.i.d. samples drawn from $D$. Then with probability at least $1 - \alpha$  
\begin{equation} \label{ineq:D-concentration}
	\sup_{p \in \cF_{m, d} }\left|\frac{1}{n} \sum_{i=1}^n p(X_i) - \bbE_{N(0,I_d)}[p(X)]\right| \leq 2(1 + 1/\alpha) \sqrt{ \frac{N}{n} },
\end{equation} where $N = {d+m \choose m}$ and $\cF_{m,d} = \{p \in \cP_{m,d}: \bbE_{N(0, I_d)}[p^2(X)] = 1 \}$.
\end{Lemma} By Lemma \ref{eq:assump-concentration} and the property of $D$, we know that 
\begin{equation*}
\begin{split}
	&\bbP\left\{ \sup_{p \in \cF_{m_{\delta},d} } \left|\sum_{i=1}^n w_i p(X_i - \wh{\theta}) - \bbE_{N(0,I_d)}[p(X)]\right| \leq 2 (1 + 4/\alpha) \sqrt{ \frac{N}{ \sum_{i=1}^n B_i } }  \right\} \\
	& \overset{(a)} = \bbP\left(\bbP\left\{ \sup_{p \in \cF_{m_{\delta},d} } \left|\sum_{i=1}^T \frac{1}{T}  p(W_i) - \bbE_{N(0,I_d)}[p(X)]\right| \leq 2 (1 + 4/\alpha) \sqrt{ \frac{N}{ T } }  \right\} \Big| \sum_{i=1}^n B_i = T, \wh{\theta} \right) \\
	& \geq 1 - \alpha/4 \;,
\end{split}
\end{equation*} 
where in the second line $W_i \overset{iid}\sim D$.

\subsection{Proof of Lemma \ref{eq:assump-concentration} }
First, we observe that to prove the result, we just need to show that given an i.i.d.\ sample $X_1, \ldots, X_{n}$ drawn from $D$,
\begin{equation} \label{ineq:equivalent-statement}
	Z:=\sup_{p \in \cF_{m,d}}\left|\frac{1}{n} \sum_{i=1}^n p(X_i) - \bbE_{D}[p(X)]\right| \leq 2  (1 + 1/\alpha) \sqrt{ \frac{N}{n} }
\end{equation} holds with probability at least $1 - \alpha$. 

For any polynomial $p \in \cF_{m,d}$, we can represent it using the Hermite basis $ \cH_{m,d}:= \{ h_{\alpha}(x) = \prod_{i \in [d]} h_{\alpha_i}(x_i): \alpha \in \bbN^{d}, \sum_{i=1}^d \alpha_i \leq m \}$:
\begin{equation*}
	p(x) = \langle \eta, \Phi_{m}(x) \rangle, \text{ for some } \eta \in \bbR^N \text{ and }\|\eta\|_2 = 1.
\end{equation*} where $N = {d+m \choose m}$ denotes the cardinality of the Hermite basis. Here $ \Phi_{m}(x)  \in \bbR^N$ and each coordinate represents a canonical basis of the Hermite polynomial with respect to $N(0, I_d)$. Equivalently,
$\bbE_{N(0, I_d)}[\Phi_{m}(X) \Phi^{\top }_{m}(X) ] = I_N$. 

Next, we show an upper bound for $\bbE[Z]$. By the symmetrization inequality, see Lemma 11.4 in \cite{boucheron2013concentration}, we have 
\begin{equation*}
\begin{split}
	\bbE \left( \sup_{p \in \cF_{m,d}} \left|\frac{1}{n} \sum_{i=1}^n p(X_i) - \bbE_{D}[p(X)]\right|  \right) &\leq 2 \bbE \sup_{p \in \cF_{m,d}} \left |\frac{1}{n}\sum_{i=1}^n \epsilon_i  p(X_i) \right| \\
	&=2 \bbE_{D} \sup_{\eta\in \bbR^N: \|\eta\|_2 = 1} \left| \langle \eta, \frac{1}{n} \sum_{i=1}^n \epsilon_i \Phi_{m}(X_i) \rangle    \right| \\
	& = 2 \bbE_{D}  \left\| \frac{1}{n} \sum_{i=1}^n \epsilon_i \Phi_{m}(X_i)   \right\|_2 \;,
\end{split}
\end{equation*} 
where $\epsilon_i$ are i.i.d.\ Rademacher random variables. 
Then
\begin{equation*}
	\begin{split}
		\bbE_{D}  \left\| \frac{1}{n} \sum_{i=1}^n \epsilon_i \Phi_{m}(X_i)   \right\|_2 &\leq \frac{1}{n} \sqrt{ \bbE_{D}  \left\|  \sum_{i=1}^n \epsilon_i \Phi_{m}(X_i)   \right\|^2_2  } \\
		&= \frac{1}{n} \sqrt{ n\bbE_{D} \left\|\epsilon \Phi_{m}(X)   \right\|^2_2  }  \overset{(a)}= \frac{1}{n} \sqrt{n \bbE_{N(0, I_d)}  \left\| \epsilon \Phi_{m}(X)   \right\|^2_2  }\\
		& =\sqrt{ \frac{\bbE_{N(0, I_d)}\left\|\Phi_{m}(X)   \right\|^2_2}{n} } = \sqrt{ \frac{N}{n} } \;,
	\end{split}
\end{equation*} 
where (a) is because $\left\| \epsilon \Phi_{m}(X)   \right\|^2_2$ 
is a polynomial of degree at most $2m$ and the moment matching condition of $D$. Then, combining the above results, we get 
\begin{equation} \label{ineq:expectation-Z-bound}
	\bbE[Z] \leq 2 \sqrt{\frac{N}{n} } \;.
\end{equation}
The deviation bound can be obtained by a simple application of the Markov inequality and the upper bound in \eqref{ineq:expectation-Z-bound},
\begin{equation*}
	\bbP\left(Z \geq  2(1+\frac{1}{\alpha}) \sqrt{ \frac{N}{n} }  \right) \leq \frac{\bbE[Z] }{2(1+\frac{1}{\alpha}) \sqrt{ \frac{N}{n} }} \leq \frac{\alpha}{1 + \alpha} \leq \alpha.
\end{equation*}

\section{Additional Lemmas}  \label{sec:proof-additional-lemmas}

In the following lemma, we collect a few standard Gaussian integration properties.
\begin{Lemma} \label{lm:gaussian-tail} Suppose $m$ is an even positive integer. 
	\begin{enumerate}
		\item[(i)] Then $\bbE_{N(0,1)}[X^m] = (m-1)!! = \sqrt{\frac{2}{\pi}} \int_{0}^{\infty} x^m  \exp(-x^2/2) dx \in \left[ \frac{\sqrt{2}}{e} (m/e)^{m/2}, \sqrt{2}e (m/e)^{m/2} \right]$.
		\item[(ii)] For any $m \geq 20$ and $t \geq \sqrt{2m}$, $\int_{\sqrt{m/e}}^{t} x^m  \exp(-x^2/2) dx \geq \frac{\sqrt{\pi}}{2e} (m/e)^{m/2}$.
		\item[(iii)] Given any $r > 0$, we have $\bbE_{N(r,1)}[X^m] \leq 2 \exp(3r\sqrt{m}/2) (m/e)^{m/2}$.
		\item[(iv)] Given any $r > 0$, denote the CDF of $X^m$ where $X \sim N(r,1)$ as $F(\cdot)$. For any $m \geq 20$, any $\delta \in (0,1/2)$, when $\Phi^{-1}(1 - \delta/2) - r \geq \sqrt{2m}$, we have
		\begin{equation*}
			\bbE_{N(r,1)}[X^m \indi(X^m \leq F^{-1}(1 - \delta))] \geq  \frac{1}{2\sqrt{2}e} \exp(-r^2/2 + r \sqrt{m/e}) (m/e)^{m/2}.
		\end{equation*}
	\end{enumerate}
\end{Lemma}\begin{proof}
	{\bf Proof of Part (i).} The first equality in (i) is a standard fact. Moreover, by the equality $(m-1)!! = \frac{m!}{2^{m/2} (m/2)!} $ and the bound for the factorial \citep{robbins1955remark},
	\begin{equation*}
		\sqrt{2\pi m} (m/e)^m \leq m! \leq e\sqrt{2\pi m} (m/e)^m,
	\end{equation*} we get $(m-1)!! \in \left[ \frac{\sqrt{2}}{e} (m/e)^{m/2}, \sqrt{2}e (m/e)^{m/2} \right]$.
	
	{\bf Proof of Part (ii).} 
	In view of what we have proved in Part (i), to show the second statement, we just need to show 
	\begin{subequations}
		\begin{align}
			&\int^{\sqrt{m/e}}_0 x^m  \exp(-x^2/2)dx \leq\frac{1}{4} \frac{\sqrt{\pi}}{e} (m/e)^{m/2}, \label{ineq:lower-tail}\\
			&\int_{t}^{\infty} x^m  \exp(-x^2/2)dx\leq\frac{1}{4} \frac{\sqrt{\pi}}{e} (m/e)^{m/2}. \label{ineq:upper-tail}
		\end{align}
	\end{subequations}
	
First, notice that $x^m \exp(-x^2/2) = \exp( -x^2/2 + m \log x )$ and this function increases when $x \in [0, \sqrt{m}]$ and decreases when $x \geq \sqrt{m}$. Thus
	\begin{equation*}
		\begin{split}
			\int^{\sqrt{m/e}}_0 x^m  \exp(-x^2/2)dx \leq \sqrt{m/e} (m/e)^{m/2} \exp(-m/(2e)) &\overset{(a)}\leq \frac{1}{4} \frac{\sqrt{\pi}}{e} (m/e)^{m/2}, 
		\end{split}
	\end{equation*} where (a) holds whenever $m \geq 20$.
	
	Next,
	\begin{equation*}
		\begin{split}
			\int_{t}^{\infty} x^m  \exp(-x^2/2)dx &\leq \sup_{x \geq t} \{x^m \exp(-x^2/4) \} \int_{t}^{\infty} \exp(-x^2/4) dx  \overset{(a)}\leq (\sqrt{2m})^{m} \exp(-m/2)  \int_{t}^{\infty} \exp(-x^2/4) dx\\
			& \overset{(b)}\leq 2(2m/e)^{m/2} \exp(-t^2/4) \overset{(c)}\leq 2(2m/e)^{m/2} \exp(-m/2) = 2(\frac{2m}{e^2} )^{m/2}\\
			& \leq \frac{1}{4} \frac{\sqrt{\pi}}{e} (m/e)^{m/2}, 
		\end{split}
	\end{equation*} where (a) is because $\exp(m \log x -x^2/4)$ is increasing when $x \in [0, \sqrt{2m}]$ and decreasing when $x \geq \sqrt{2m}$; (b) is because of the Gaussian tail integration bound, see Theorem 1.2.3 in \cite{durrett2019probability}; (c) is because $t \geq \sqrt{2m}$ and $m \geq 20$. 

{\bf Proof of Part (iii).} 
\begin{equation} \label{ineq:shift-normal-mean-up}
	\begin{split}
		\bbE_{N(r,1)}[X^m] &= \int_{-\infty}^{\infty} x^m \phi(x-r) dx \leq 2 \int_{0}^{\infty} x^m \phi(x-r) dx \\
		& = \frac{2}{\sqrt{2\pi}} \exp(-r^2/2) \int_{0}^{\infty} \exp(-x^2/2 + xr + m \log x ) dx.
	\end{split}
\end{equation} Let $f(x) = -x^2/2 + xr + m \log x$, then $f'(x) = -x + r + m/x$ and $f''(x) = -1 - m/x^2 \leq -1$. Notice that $f'(x)$ has a positive root $x^* = \frac{r + \sqrt{r^2 + 4m} }{2} $. So, by a Taylor expansion, we have for all $x \geq 0$, $f(x) \leq f(x^*) - \frac{1}{2} (x - x^*)^2$ and at the same time,
\begin{equation*}
\begin{split}
	f(x^*) &  = -m + \frac{x^{*2}}{2} + m \log(x^*)\\
	& = -\frac{m}{2} + \frac{1}{4}r \sqrt{r^2 + 4m} + \frac{r^2}{4} + m \log\left( \frac{r + \sqrt{r^2 + 4m}}{2} \right) \\
	& \leq -\frac{m}{2}+\frac{r^2}{2} + \frac{r\sqrt{m}}{2} + m \log(\sqrt{m} + r) \quad(\text{since} \sqrt{a+b} \leq \sqrt{a} + \sqrt{b})\\
	& \leq  -\frac{m}{2}+ m \log \sqrt{m} +\frac{r^2}{2} + \frac{3r\sqrt{m}}{2}.
\end{split}
\end{equation*} By plugging the above inequalities into \eqref{ineq:shift-normal-mean-up},
\begin{equation*}
	\begin{split}
		\bbE_{N(r,1)}[X^m] &\leq  \frac{2}{\sqrt{2\pi}} \exp(-r^2/2) \int_{0}^{\infty} \exp\left(  -\frac{m}{2}+ m \log \sqrt{m}+ \frac{r^2}{2} + \frac{3r\sqrt{m}}{2} - \frac{1}{2} (x - x^*)^2 \right) dx \\
		& \leq 2 \exp(3r\sqrt{m}/2)  (m/e)^{m/2}.
	\end{split}
\end{equation*} 

{\bf Proof of Part (iv).} Let $t = F^{-1}(1 - \delta)$. Then by definition, $\bbP_{N(r,1)} \left( X^m \geq t  \right) = \delta$. Since $\delta \in (0,1/2)$, $t^{1/m} \geq r$. This means that 
\begin{equation} \label{ineq:t-range}
\begin{split}
\bbP_{N(0,1)}((X+r)^m \geq t )= \delta  &\Longrightarrow \bbP_{N(0,1)}(|X| \leq t^{1/m}- r) \leq 1 - \delta \text{ and } \bbP_{N(0,1)}(|X| \leq t^{1/m}+ r) \geq 1 - \delta \\
	&  \Longleftrightarrow \Phi^{-1}(1 - \delta/2) - r \leq t^{1/m} \leq \Phi^{-1}(1 - \delta/2) + r.
\end{split}	
\end{equation} Thus,
\begin{equation*}
	\begin{split}
		\bbE_{N(r,1)}[X^m \indi(X^m \leq F^{-1}(1 - \delta))] & = \int^{t^{1/m}}_{-t^{1/m}} x^m \phi(x - r) d x \\
		& = \frac{1}{\sqrt{2 \pi}} \exp(-r^2/2) \int^{t^{1/m}}_{-t^{1/m}} x^m \exp(-x^2/2 + xr) d x \\
		& \geq  \frac{1}{\sqrt{2 \pi}} \exp(-r^2/2) \int^{t^{1/m}}_{0} x^m \exp(-x^2/2 + xr) d x \\
		& \overset{ \eqref{ineq:t-range} }\geq  \frac{1}{\sqrt{2 \pi}} \exp(-r^2/2) \int^{\Phi^{-1}(1 - \delta/2) - r }_{0} x^m \exp(-x^2/2 + xr) d x \\
		& \geq \frac{1}{\sqrt{2 \pi}} \exp(-r^2/2) \int^{\Phi^{-1}(1 - \delta/2) - r }_{\sqrt{m/e}} x^m \exp(-x^2/2 + xr) d x \\
		& \geq \frac{1}{\sqrt{2 \pi}} \exp(-r^2/2 + r \sqrt{m/e}) \int^{\Phi^{-1}(1 - \delta/2) - r }_{\sqrt{m/e}} x^m \exp(-x^2/2) d x \\
		& \overset{(a)}\geq  \frac{1}{2\sqrt{2}e} \exp(-r^2/2 + r \sqrt{m/e}) (m/e)^{m/2} \;,
	\end{split}
\end{equation*} where (a) holds since $\Phi^{-1}(1 - \delta/2) - r \geq \sqrt{2m}$ and Lemma \ref{lm:gaussian-tail} (ii).
This finishes the proof.
\end{proof}

The following is the Dvoretzky-Kiefer-Wolfowitz-Massart inequality (DKW inequality) from \cite{massart1990tight}.
\begin{Lemma}\label{lm:DKW}
	Suppose $n$ is a positive integer. Let $X_1, \ldots, X_n$ be i.i.d.\ real-valued random variables drawn from a distribution with CDF $F(\cdot)$. Then for any $\alpha \in (0,1)$,
	\begin{equation*}
		\bbP\left( \sup_{t \in \bbR} |F_n(t) - F(t)| > \sqrt{ \frac{\log(2/\alpha)}{2n} } \right) \leq \alpha,
	\end{equation*}where $F_n(\cdot)$ is the empirical CDF.
\end{Lemma}

The following lemma can be viewed as a uniform DKW inequality for the projection of i.i.d. data drawn from a multivariate distribution onto all possible directions in the unit sphere.

\begin{Lemma}\label{lm:multi-DKW}
	Suppose $n$ is a positive integer. Let $X_1, \ldots, X_n$ be i.i.d. real-valued random vectors drawn from a distribution supported on $\bbR^d$. Given every $u \in \cS^{d-1}$, define the projected CDF and empirical CDF on the direction $u$ as $F(t; u) = \bbP(X^\top u \leq t)$ and $F_n(t; u) = \frac{1}{n} \sum_{i \in [n]} \indi(X_i^\top u \leq t)$, respectively. Then for any $\alpha \in (0,1)$,
	\begin{equation*}
		\bbP\left( \sup_{t \in \bbR, u \in \cS^{d-1}} |F_n(t;u) - F(t;u)| > 150 \sqrt{ \frac{(d+1)}{n} } + \sqrt{ \frac{\log(1/\alpha)}{2n} } \right) \leq \alpha.
	\end{equation*}
\end{Lemma} 
\begin{proof}
	The proof of this lemma leverages a standard application of Dudley's entropy integral \citep{dudley1978central} and VC classes \citep{vapnik2015uniform}. The version with the explicit constants stated in this lemma can be found in Lemma 7.3 of \cite{chen2018robust} where we replace their $3 + 2 |S|$ by $d+1$ since the VC dimension for their function class is $3 + 2 |S|$, while it is $d+1$ for the function class $\{\indi(x^\top u \leq t): u \in \cS^{d-1}, t \in \bbR\}$ interested in this lemma (see Example 8.3.5 in \cite{vershynin2018high}).
\end{proof}

\begin{Lemma}\label{lm:normal-quantile}
	Let $\Phi(\cdot)$ be the CDF of the standard Gaussian distribution. Given some $q < 1/2$, $x = \Phi^{-1}(1 - q)$. Suppose $x \geq 2$, then
	\begin{equation*}
	\begin{split}
		\frac{1}{2} \sqrt{\log( 1/q ) } \leq \sqrt{\log(1/ (\sqrt{2\pi}q))}	\leq x \leq  \sqrt{2 \log( 1/q ) }\quad  \text{ and }\quad \frac{1}{\sqrt{2\pi}} \exp(-x^2) \leq q \leq \frac{1}{2\sqrt{2\pi}} \exp(-x^2/2).
	\end{split}
	\end{equation*} 
\end{Lemma}
\begin{proof}
	By the definition of $x$, we have $\int_x^\infty \frac{1}{\sqrt{2\pi}} \exp(-t^2/2) dt = q$. In addition, for any $x > 0$ (see Theorem 1.2.3 in \cite{durrett2019probability}),
	\begin{equation}
		(\frac{1}{x} - \frac{1}{x^3}) \exp(-x^2/2) \leq \int_{x}^{\infty} \exp( - t^2/2 ) dt \leq \frac{1}{x} \exp(-x^2/2).
	\end{equation}
	Since $x \geq 2$, $\frac{1}{x} - \frac{1}{x^3} \geq \frac{1}{x}( 1- \frac{1}{4} ) = \frac{3}{4x}$ and $\frac{1}{x}\leq \frac{1}{2}$. At the same time, it is easy to check $\frac{3}{4x} \exp(-x^2/2) \geq \exp(-x^2)$ for all $x \geq 2$. Thus, we have
	\begin{equation*}
		\frac{1}{\sqrt{2\pi}} \exp(-x^2) \leq q \leq \frac{1}{2\sqrt{2\pi}} \exp(-x^2/2)
	\end{equation*} holds for any $x \geq 2$. Thus
	\begin{equation*}
 	\frac{1}{2} \sqrt{\log( 1/q ) } \overset{(a)}\leq \sqrt{\log(1/ (\sqrt{2\pi}q))}	\leq x \leq \sqrt{2 \log( 1/(2 \sqrt{2\pi} q) ) } \leq \sqrt{2 \log( 1/q ) },
	\end{equation*} where in (a) we use the fact $q \leq 0.023$ when $x \geq 2$, so the inequality always holds when $x \geq 2$. 
\end{proof}

\begin{Lemma}(Bernstein's inequality) \label{lm:Binomial-prop} Given any positive integer $n$, $\alpha \in (0,1)$ and $0\leq p \leq 1$,
	\begin{equation*}
		\bbP\left(|\textnormal{Binomial}(n,p) - np| \geq \frac{4\log(2/\alpha) }{3}  \vee 2\sqrt{ p(1-p)n\log(2/\alpha) } \right) \leq \alpha.
	\end{equation*} 
    Since $ \frac{np(1-p)}{C} + C \log(2/\alpha) \geq \frac{4\log(2/\alpha) }{3}  \vee 2\sqrt{ p(1-p)n\log(2/\alpha) }  $ holds for any $C \geq 4/3$, the inequality above holds if we replace $\frac{4\log(2/\alpha) }{3}  \vee 2\sqrt{ p(1-p)n\log(2/\alpha) } $ by $ \frac{np(1-p)}{C} + C \log(2/\alpha) $ for any $C \geq 4/3$.
\end{Lemma}

The next lemma concerns non-Gaussian component analysis. Recall the definition of high-dimensional hidden direction distribution $P_u^A$ provided in Definition \ref{def:high-dim-distribution}. 
Then Lemma 3.4 in \cite{diakonikolas2017statistical} provides the following upper bound on the $\chi^2$-correlation between $P^A_u$ and $P^A_{u'}$.
\begin{Lemma}\label{lm:NGCA}
	Let $m$ be any positive integer. If the distribution $A$ over $\bbR$ agrees with the first $m$ moments of $N(0,1)$, then for all $u, u' \in \cS^{d-1}$, we have that 
	\begin{equation*}
		\Big| \int \frac{P^A_u(x) P^A_{u'}(x)}{\phi_d(x) } dx - 1 \Big| \leq |u^\top u'|^{m+1} \chi^2(A, N(0,1)).
	\end{equation*}
\end{Lemma}


The following lemma is adapted from Lemma 8.18 and Exercise 8.3 of \cite{diakonikolas2023algorithmic}. We provide a proof here for the reader's convenience. 
\begin{Lemma} \label{lm:polynomial-matching-bound}
	Let $C > 0$ and $m \in \mathbb{Z}_+$. For any $a_0, a_1, \dots, a_m \in \mathbb{R}$, there exists a unique degree at most $m$ polynomial $p \colon \mathbb{R} \to \mathbb{R}$ such that for each integer $0 \le t \le m$ we have that
\begin{equation}
\int_{-C}^{C} p(x)x^t dx = a_t \,.
\end{equation}
Furthermore, for each $x \in [-C, C]$ we have that $|p(x)| \le   2m^2  \sum_{k=0}^{m} \left(2e\right)^k \left(  \max_{0 \leq t \leq k}|a_t| C^{-t-1} \right)$.
\end{Lemma}

\begin{proof} First, the uniqueness of $p(x)$ is proved in Lemma 8.18 of \cite{diakonikolas2023algorithmic}. We focus on providing an upper bound on $\sup_{x \in [-C,C]}|p(x)|$ with explicit dependence on $m$.   First, let us define
\[
        q(y):=p(Cy), \qquad y\in[-1,1].
\]
It is easy to check \(q\) has degree at most \(m\), and for each \(0\leq t\leq m\), $ \int_{-C}^{C} p(x)x^t\,dx =  C^{t+1}\int_{-1}^{1} q(y)y^t\,dy$. 
Hence, $\int_{-1}^{1}q(y)y^t\,dy
        =
        b_t$, where $ b_t:=a_t C^{-t-1}$. 
Moreover, $\sup_{x\in[-C,C]} |p(x)|
        =
        \sup_{y\in[-1,1]} |q(y)|$.
Thus, it is enough to prove the desired bound for \(q\) on \([-1,1]\).

Next, let us write $q(x)$ as a Legendre polynomial expansion. A few standard properties of Legendre polynomials are provided in Lemma \ref{lm:legendre-polynomial}. 
Set $\phi_k(x):=\sqrt{\frac{2k+1}{2}}\,L_k(x)$ where $L_k(x)$ denotes the $k$th Legendre polynomial.
Then by Lemma \ref{lm:legendre-polynomial} (ii), we know \(\{\phi_0,\ldots,\phi_m\}\) is an orthonormal basis of the space of polynomials of degree at most \(m\) with respect to the inner product.
Therefore,
\[
        q(x)=\sum_{k=0}^{m}\beta_k\phi_k(x),
        \qquad
        \beta_k=\int_{-1}^{1}q(x)\phi_k(x)\,dx.
\]
Write $\phi_k(x)=\sum_{t=0}^{k}\gamma_{k,t}x^t$.
Using the prescribed moments of \(q\), we obtain
\[
        \beta_k
        =
        \int_{-1}^{1}q(x)\phi_k(x)\,dx
        =
        \sum_{t=0}^{k}\gamma_{k,t}
        \int_{-1}^{1}q(x)x^t\,dx
        =
        \sum_{t=0}^{k}\gamma_{k,t}b_t.
\]
Next, we use the explicit formula of Legendre polynomial in Lemma \ref{lm:legendre-polynomial} (v). If \(0\leq t\leq k\) and \(k-t\) is even, then
\[
        |\gamma_{k,t}|
        =
        \sqrt{\frac{2k+1}{2}}\,
        2^{-k}
        \binom{k}{(k-t)/2}
        \binom{k+t}{k} = 
        \sqrt{\frac{2k+1}{2}}\,
        2^{-k}
        \binom{k}{(k-t)/2}
        \binom{k+t}{t}.
\]
If \(k-t\) is odd, then \(\gamma_{k,t}=0\).

Using the Legendre expansion of \(q\), we get
\[
\begin{aligned}
        \sup_{x\in[-1,1]} |q(x)|
        &\leq
        \sum_{k=0}^{m}|\beta_k|
        \sup_{x\in[-1,1]}|\phi_k(x)|  \overset{\text{Lemma }\ref{lm:legendre-polynomial} (iii) }\leq
        \sum_{k=0}^{m}
        \sum_{t=0}^{k}
        |\gamma_{k,t}|\,|b_t|\,
        \sqrt{\frac{2k+1}{2}}  \\
        &=
        \sum_{k=0}^{m}
        \sum_{\substack{0\leq t\leq k\\ k-t\ \mathrm{even}}}
        \frac{2k+1}{2}
        2^{-k}
        \binom{k}{(k-t)/2}
        \binom{k+t}{t}
        |b_t|\\
        & \overset{(a)} \leq  \sum_{k=0}^{m}
        \sum_{\substack{0\leq t\leq k\\ k-t\ \mathrm{even}}}
        \frac{2k+1}{2} \left(\frac{2ek}{t \vee 1 }\right)^t |b_t| \leq b_0 + \sum_{k=1}^{m} \left( 2k^2 \max_{1 \leq t \leq k}  \left(\frac{2ek}{t }\right)^t |b_t| \right)\\
        & \overset{(b)}\leq  b_0 + \sum_{k=1}^{m} 2k^2  \left(2e\right)^k \left(  \max_{1 \leq t \leq k}|b_t| \right) \leq 2m^2  \sum_{k=0}^{m} \left(2e\right)^k \left(  \max_{0 \leq t \leq k}|b_t| \right) \\
        & \overset{ (c) }\leq   2m^2  \sum_{k=0}^{m} \left(2e\right)^k \left(  \max_{0 \leq t \leq k}|a_t| C^{-t-1} \right)  \;,
\end{aligned}
\] where (a) is because $ 2^{-k}\binom{k}{(k-t)/2}\leq 1$ and 
$\binom{k+t}{t}
        \leq
        \left(\frac{e(k+t)}{t}\right)^t \leq \left(\frac{2ek}{t}\right)^t  $ when $t \geq 1$ and $\binom{k+t}{t} = 1$ when $t = 0$; (b) is because $\max_{1 \leq t \leq k}  \left(\frac{2ek}{t }\right)^t = (2e)^k$; (c) is because $b_t =a_t C^{-t-1}$.
\end{proof}

In the following lemma, we list a few common properties of Legendre Polynomials \citep{szeg1939orthogonal}, which is the set of orthogonal polynomials supported on $[-1,1]$. These properties can also be found in Fact 5.3 and Corollary 5.4 of \cite{diakonikolas2017statistical}.
\begin{Lemma}\label{lm:legendre-polynomial}
	The Legendre polynomial, $L_k(x)$, for $k \in \mathbb{Z}_{+}$, satisfies the following properties:
	\begin{itemize}
		\item (i) $L_k(x)$ is a degree-$k$ polynomial, $L_0(x) = 1$ and $L_1(x) = x$.
		\item (ii) $\int_{-1}^1 L^2_i(x) d x = \frac{2}{2i + 1}$ for all $i \geq 0$ and $\int_{-1}^1 L_i(x) L_j(x) d x = 0$ for all $i, j \geq 0, i \neq j$.
		\item (iii) $|L_k(x)| \leq 1$ for all $|x| \leq 1$.
		\item (iv) $L_k(x) = (-1)^k L_k(-x)$.
		\item (v) $L_k(x) = \frac{1}{2^k} \sum_{i=0}^{ \lfloor k/2 \rfloor } (-1)^i {k \choose i} {2k-2i \choose k} x^{k-2i}$.
		\item (vi) $\int_{-1}^1 |L_k(x)|dx \leq O(1/\sqrt{k})$.
	\end{itemize}
\end{Lemma}

\begin{Lemma} \label{lm:t-epsilon-bound}
	For any $\epsilon_{\max} < 1/2$, positive integer $n$ and $\alpha \in (0,1)$ such that $n/(d + \log(1/\alpha))$ is sufficiently large, then for all $\epsilon \in [0, \epsilon_{\max}]$,
	\begin{equation*}
		1/t_{\epsilon} \asymp \frac{1}{\sqrt{\log(1/\epsilon)}} + \frac{1}{\sqrt{\log( n/\log(1/\alpha) )}} + \frac{1}{\sqrt{\log( n/d )}},
	\end{equation*} where $t_{\epsilon} = \Phi^{-1}\left( 1- \epsilon - 300 \sqrt{ \frac{d}{n} } - \sqrt{ \frac{\log(1/\alpha)}{n} }  \right)$.
\end{Lemma}
\begin{proof} We divide the proof into two cases. First, when $\epsilon \in [0.01, \epsilon_{\max}]$, $1/t_{\epsilon}$ is of constant level, so the result follows. 
	At the same time, when $\epsilon \in [0,  0.01]$, we have $t_{\epsilon} \geq 2$ when $n/(d + \log(1/\alpha))$ is sufficiently large, by Lemma \ref{lm:normal-quantile}, we have
	\begin{equation*}
	\begin{split}
		t_{\epsilon} & \asymp  \sqrt{ \log\left( \frac{1}{\epsilon + \sqrt{ \frac{d}{n} } + \sqrt{ \frac{\log(1/\alpha)}{n} }}  \right)  } \\
		& \asymp \sqrt{ \log\left( 1/\epsilon  \right)} \wedge \sqrt{ \log\left(n/d \right) } \wedge \sqrt{ \log\left(\frac{n}{\log \left( \frac{1}{\alpha} \right)} \right) }.
	\end{split}
	\end{equation*} Thus, $1/t_{\epsilon} \asymp \frac{1}{\sqrt{\log(1/\epsilon)}} + \frac{1}{\sqrt{\log( n/d )}} + \frac{1}{\sqrt{\log( n/\log(1/\alpha) )}}$.
\end{proof}

\end{sloppypar}

\end{document}